\documentclass[numbers,webpdf,imanum]{ima-authoring-template}%
\makeatletter
\def\@journaltitle{}
\def\@DOI{}
\def\@vol{}
\def\@copyrightyear{}
\def\thelastpage{}
\def\@access{}
\def\@appnotes{}
\makeatother

\graphicspath{{Fig/}}

\theoremstyle{thmstyletwo}%
\newtheorem{theorem}{Theorem}[section]
\newtheorem{proposition}[theorem]{Proposition}%
\newtheorem{lemma}[theorem]{Lemma}

\newtheorem{remark}{Remark}[section]%

\newtheorem{definition}{Definition}[section]

\numberwithin{equation}{section}
\newtheorem{assumption}{Assumption}[section]
\usepackage{bm}
\usepackage{eucal}
\allowdisplaybreaks
 \providecommand{\triplenorm}[1]{\left|\!\left|\!\left| #1 \right|\!\right|\!\right|} 
 \usepackage{upgreek}
 \usepackage{diagbox}

\begin{document}



\title[VEM for coupled fourth-order 
PNP-NS model]{Conforming Virtual Element Method for   Biharmonic Poisson--Nernst--Planck 
Navier--Stokes systems}
\author{Ankur Ankur and Andrea Cangiani* 
\address{\orgdiv{Mathematics Area}, 
\orgname{International School for Advanced Studies (SISSA)}, 
\orgaddress{\street{via Bonomea 265}, \postcode{I-34136}, \state{Trieste}, \country{Italy}}}}

\authormark{A. Ankur and  A. Cangiani}

\corresp[*]{Corresponding author: \href{andrea.cangiani@sissa.it}{andrea.cangiani@sissa.it}}

\received{12}{05}{2026}


\abstract{We develop and analyze 
a conforming Virtual Element Method (VEM) for the fourth-order Poisson-Nernst-Planck-Navier-Stokes (PNP-NS) system. 
The proposed scheme is based on compatible  discretizations of each component: an \(H^2\)-conforming VEM for a fourth-order electrostatic potential equation, an \(H^1\)-conforming VEM for the Nernst--Planck equations, and a \textit{divergence-free} and \textit{pressure-robust} VEM for the Navier--Stokes equations. 
Time integration is performed using a  backward Euler scheme to ensure stability.
We establish the well-posedness of the continuous problem up to three-dimensions and also establish existence and uniqueness of the fully discrete solution via a fixed-point argument. Further, we derive a priori error estimates 
showing that the electrostatic potential, concentration and velocity converge optimally  in Bochner norms \(L^\infty\bigl(0,T; H^2(\Omega)\bigr)\), \(L^2\bigl(0,T; H^1(\Omega)\bigr)\), and \(L^2\bigl(0,T; \bm{H}^1(\Omega)\bigr)\), respectively.
The analysis requires a sophisticated argument to avoid any restrictive assumptions on the coercivity and continuity constants and to handle the 
trilinear form involving three different variables. 
The pressure-robust design permits the use of lowest-order pressure approximations without compromising convergence rates of the other variables, reducing computational cost. Numerical experiments confirm the theoretical convergence rates for various polynomial orders and demonstrate the scheme's robustness, including in low-viscosity regimes.}


\maketitle
\vspace{0.2cm}
\section{Introduction}

The classical theory of electrokinetic flow underpins a wide range of disciplines, including micro and nanofluidics, electrochemistry, and biophysics. Several theoretical frameworks have been developed to describe electrokinetic phenomena, notably the Poisson-Boltzmann equation, the Nernst-Planck-Poisson equation, and the more comprehensive Poisson-Nernst-Planck-Navier-Stokes (PNP-NS) equations \cite{probstein2005physicochemical}. Among these, the PNP-NS system provides the most general description, as it captures coupled physical-chemical processes and electrohydrodynamic phenomena. This framework is particularly effective for colloidal systems, enabling detailed analysis of separation, aggregation, and sedimentation of charged particles. Moreover, it has found broad application in areas such as drug transport across biomembranes, semiconductor technologies, and the design of microfluidic devices \cite{cioffi2006modeling, wang2017modeling, lu2010poisson}.

The PNP-NS system is a set of coupled partial differential equations (PDEs) used to describe the behavior of charged species in a fluid under the influence of electric fields, while also accounting for fluid flow dynamics. 
The Poisson equation determines the electrostatic potential arising from the spatial distribution of charges. Ionic transport is governed by the Nernst-Planck equation, which incorporates diffusion and electromigration in response to the electrostatic potential. The Navier-Stokes equation describes the motion of the fluid due to internal forces, external forces, and pressure gradients. Coupling all these equations enables a self-consistent description of electrokinetic flow driven by the interaction between electric fields, ionic transport, and hydrodynamics. Various numerical schemes and analytical results for such coupled models have been developed in the literature; we refer to \cite{ern2012mathematical, chainais2022long, he2021mixed, correa2023new, he2018mixed, prohl2010convergent, dehghan2023optimal} for representative contributions. 

Despite its broad applicability, classical electrokinetic theory has well-known limitations, motivating continued developments in the field \cite{vlachy1999ionic, bazant2009towards, jiwari2026finite}. The classical theory relies on a mean-field approximation, wherein the electric field acting on an ion is determined by the local average charge density. This assumption fails in systems such as multivalent electrolytes and solvent-free ionic liquids. Addressing these limitations requires more refined frameworks that move beyond the mean-field assumption.
New formulations are necessary to capture the complexities of room-temperature ionic liquids (RTILs). RTILs are composed of large organic cations paired with either similarly sized organic anions or smaller inorganic anions. They show great potential as solvent-free electrolytes for various advanced technologies, including  batteries, supercapacitors, solar cells, and electroactuators. Recently, Bazant, Storey, and Kornyshev (BSK) \cite{de2020continuum, storey2012effects, bazant2011double}  proposed a Landau-Ginzburg-type continuum model for RTILs, resulting in a  modified fourth-order problem for the electrostatic potential:  \vspace{-0.0cm}
\begin{align} 
&\epsilon \;( l^2_c \;\Delta^2 \phi - \Delta \phi ) = \rho. \label{eqn1} 
\end{align}
Here, $\epsilon$ denotes permittivity, $l_c$ electrostatic correlation length, $\phi$ electrostatic potential, and $\rho$ charge density. 
The model can be interpreted as a gradient-based approximation of nonlocal electrostatics for interacting effective charges, in which the permittivity operator is modified by the inclusion of a correlation-length scale. This framework captures key microscopic effects observed in molecular simulations and enables simplified modeling of electrokinetic flows in correlated ionic fluids.


Discretizing the fourth-order BSK-type equation requires $H^2$-conforming  finite elements. On general unstructured meshes, constructing such spaces \cite{wang2013minimal} requires high polynomial degrees (at least 5 in 2D and 9 in 3D), resulting in a very large number of degrees of freedom (dofs) and considerable computational overhead. 
While nonconforming and mixed methods \cite{he2021mixed} typically yield only $C^0$ regular discrete solutions, mixed methods also introduce auxiliary variables that substantially increase the system size.
Considering the complexity of the  full modified PNP-NS model (governed by four equations \eqref{eqn3a}-\eqref{eqn3d}), 
computational efficiency becomes critical as inefficient discretizations for any single variable can make the entire coupled problem intractable. Further computational challenges arise from the formation of electric double (Debye) layers, which generate sharp potential gradients near charged surfaces.

\subsection{Key methodological features and analytical contributions}

To address the challenges listed above, we develop a novel conforming discretization for the fourth-order PNP-NS model based on the recently introduced virtual element method (VEM)~\cite{beirao2013basic}. 
The proposed scheme combines: 
an $H^2$ conforming VEM for the potential  equation~\cite{brezzi2013virtual, antonietti2016c, chen2022conforming}, an $H^1$ conforming VEM for the NP system~\cite{beirao2013basic,Cangiani17}, and a divergence-free VEM formulation for the NS system~\cite{da2018virtual, da2017divergence}. 
This combination of discretizations balances conformity requirements with efficiency and avoids unnecessary increases in the global number of degrees of freedom. 
In particular, we obtain a simple $H^2$-conforming discretization with a reduced number of degrees of freedom compared to mixed FEMs. The divergence-free NS discretization is naturally pressure-robust. As such, it allows  piecewise constant approximations of the pressure without degrading the convergence order of the remaining variables, leading to a significant reduction in computational cost. Moreover, the method remains robust in low-viscosity regimes.  To the best of our knowledge, this is the first conforming numerical analysis for a coupled PNP-NS type system involving a biharmonic electrostatic potential.

From the analytical point of view, one of the main difficulties arises from the strong nonlinear coupling between equations of different types.
We establish well-posedness of the coupled continuous problem. In particular, unlike several related PNP-NS models where uniqueness is typically available only in two dimensions, or in three dimensions only after removing the convective term from the Navier-Stokes equations \cite[Theorem 2]{schmuck2009analysis}, we prove uniqueness up to three dimensions under suitable smallness assumptions on the data. 
 
We present a complete numerical analysis of the fully discrete VEM scheme. The analysis relies on carefully designed projection operators and consistency estimates for the nonlinear trilinear forms and their discrete counterparts.
We show well-posedness of the discrete problem by a fixed-point argument. We then derive optimal a priori error estimates for the electrostatic potential, concentrations and velocity in the Bochner spaces \(L^\infty(0,T; H^2(\Omega))\), \(L^2(0,T; H^1(\Omega))\) and \(L^2(0,T; \bm{H}^1(\Omega))\), respectively. A key feature of the analysis is that the error estimates are obtained without imposing restrictive assumptions on the coercivity and continuity constants of the continuous and discrete bilinear forms, unlike many existing results in the VEM literature \cite[Proposition 5.3]{antonietti2023virtual}. This is achieved through a careful use of Sobolev embedding arguments, which is delicate to apply within the VEM framework. 
Another significant technical hurdle stems from the non skew-symmetric trilinear forms involving three distinct variables within a single form (e.g., velocity-concentration-potential couplings). Unlike the standard skew-symmetric case where the convective term vanishes under the natural test function, here we must carefully quantify the consistency error between the continuous trilinear form and its discrete VEM counterpart (see Lemma \ref{lemma_S_C}). As this consistency error enters the error analysis explicitly, the non skew-symmetric contributions no longer cancel and require a more refined treatment.

Lastly, we confirm the theoretical results through numerical experiments. Our tests demonstrate optimal convergence across various polynomial orders and verify the scheme's stability and robustness for low viscosity coefficients.
\vspace{-0.2cm}
\section{Preliminaries and the continuous problem} Let $\Omega\subset \mathbb{R}^d\; (d\leq 3)$ be a bounded Lipschitz domain  and $\Omega_T=\Omega \times (0,T]$ with time $T > 0.$ The fourth-order modified PNP-NS coupled model \cite{he2021mixed} with no net flux is given by  \vspace{-0.1cm}
\begin{subequations} \label{eqn3}
    \begin{align}
        & 
        \Delta^2\phi - \Delta\phi = n^F + f_{\phi}, \quad \text{in} \;\Omega_T,\label{eqn3a}\\[-0.12em]
        &\partial_t C_i-\nabla \cdot (\nabla C_i + z_i C_i\nabla \phi) + \nabla \cdot ( \bm{u}C_i)=f_{c_i}, \quad \text{in} \;\Omega_T \;\text{with}\; i=1,2,\label{eqn3b}\\[-0.12em]
        & \partial_t \bm{u} + (\bm{u}\cdot \nabla)\bm{u} - \upnu \Delta \bm{u} + \nabla p = - n^F \nabla \phi + \bm{f}_{u}, \quad \text{in} \;\Omega_T,\label{eqn3c}\\[-0.12em]
        &\nabla \cdot \bm{u}=0, \quad \text{in} \;\Omega_T,\label{eqn3d}\\[-0.12em]
        &\nabla\phi\cdot \bm{n}=\nabla(\Delta\phi)\cdot \bm{n}=0, \;\; (\nabla C_i +z_i C_i\nabla \phi).\bm{n} =0,\; 
        \bm{u}=\bm{0}, \;\; \text{on} \;\partial\Omega\times (0,T],\label{eqn3e}\\[-0.12em]
        &\bm{u}(x, 0)=\bm{u}_0(x), 
        \;\; C_i(x, 0)=C_{i,0}(x) \;\;\text{in} \;\Omega \;\text{with}\; i=1,2, \label{eqn3f} 
    \end{align} 
\end{subequations} 
where $\bm{n}$ is the outward unit normal, $\phi$ represents the electrostatic potential, while $C_1$ and $C_2$ denote the ionic concentrations. We define the free charge density by $n^F = z_1 C_1 + z_2 C_2$, where $z_1=1$ and $z_2=-1$ are the ionic charges. The variables $\bm{u}$ and $p$ correspond to the velocity and pressure of the fluid. The functions $f_{\phi}:\Omega_T\to\mathbb{R}$, $f_{c_i}:\Omega_T\to\mathbb{R}$ for $i=1,2$, and $\bm{f}_u:\Omega_T\to\mathbb{R}^d$ denote the given external forcing terms.
Since $\phi$ and $p$ are determined up to additive constants, we impose mean-zero constraints:\vspace{-0.2cm}
\begin{equation} \label{eqn5}
    \int_{\Omega} \phi\,dx = \int_{\Omega} p \,dx= 0.  \vspace{-0.2cm}
\end{equation}
We also note that compatibility between equation \eqref{eqn3a} and boundary condition \eqref{eqn3f} implies the electro-neutrality condition \vspace{-0.1cm}
\begin{equation}
    \int_{\Omega} \left(n^F + f_{\phi}\right)\,dx=0  .  \label{eqn4} \vspace{-0.05cm}
\end{equation}
For analytical simplicity, we set the viscosity coefficient $\upnu = 1$. Incorporating the viscosity coefficient is possible by a straightforward modification of the analysis, following~\cite{da2018virtual}, where it is also shown numerically that the VEM is more robust than standard mixed FEM. Our numerical experiments likewise confirm the robustness of the VEM with respect to small values of $\upnu$.
\vspace{-0.1cm}
\subsection{Notations and weak formulation}
For any subdomain $\mathcal{F}$ in $\Omega$, we use standard notation for  Sobolev spaces $W^{l,q}(\mathcal{F})$ with their associated norms and seminorms $\| \cdot \|_{l,q,\mathcal{F}}$ and $| \cdot |_{l,q,\mathcal{F}}$, respectively. For $q=2$, we set $W^{l,2}(\mathcal{F}) := H^l(\mathcal{F})$, and for the vector-valued counterpart,  $\bm{H}^l(\mathcal{F})$ := $\left[H^l(\mathcal{F})\right]^d$,  with $\|\cdot\|_{l,2,\mathcal{F}} := \|\cdot\|_{l,\mathcal{F}}$ and ${|\cdot|_{l,2,\mathcal{F}}} := |\cdot|_{l,\mathcal{F}}$. The $L^2$ inner product on $\mathcal{F}$ is denoted by $(\cdot,\cdot)_{0,\mathcal{F}}$. If $\mathcal{F} = \Omega$, we omit the subscript for simplicity. We also consider Bochner spaces $L^q(0,T;\mathcal{X})$ with standard norms \cite{quarteroni2008numerical}, and denote by $\mathcal{X}'$ the dual of Banach space $\mathcal{X}$ with action $\langle f,x\rangle_{\mathcal{X}',\mathcal{X}}$, $f\in \mathcal{X}'$, $x\in \mathcal{X}$. Throughout the paper, $c$ and $c_i, i\in \mathbb{N}$ denote  generic constants whose values may vary at different occurrences. Additionally, to simplify notation, we shall omit the explicit time dependence in expressions such as \( f(\cdot, t) \), and simply write \( f \) when no confusion arises. 
Next, we denote $M:=H^1(\Omega)$, $\bm{M}:=M\times M$ and introduce the following functional spaces: \vspace{-0.1cm}
\begin{align}
U &:= H^2_{\partial}(\Omega) = \left\{ \varphi \in H^2(\Omega) : \quad\partial_{\bm{n}} \varphi = 0 \; \;\text{on} \;\; \partial \Omega \right\}, \nonumber\\[-0.1em]
\mathring{U} &:= \mathring{H}{}_{\partial}^{2}(\Omega) = \left\{ \varphi \in U : \quad (\varphi,1)_0 = 0 \right\}, \nonumber\\[-0.1em]
\bm{V} &:= \bm{H}^1_0(\Omega) = \left\{ \bm{v} \in \bm{H}^1(\Omega) :\quad \bm{v} = \bm{0} \;\; \text{on} \; \;\partial \Omega \right\},\nonumber \\[-0.1em]
W &:= L^2_0(\Omega) = \left\{ q \in L^2(\Omega) : \quad (q,1)_0 = 0 \right\}.\nonumber
\end{align}
We endow these spaces with the norms  \vspace{-0.15cm}
\[
\|{\varphi}\|_{\mathring{U}}^{2}= \|{\varphi}\|_{{U}}^{2}:=\Vert {\varphi}\Vert_{2}^{2}, \quad\quad\|(C_{1},C_{2})\|_{\bm{M}}^{2}:=\|C_{1}\|_{M}^{2}+\|C_{2}\|_{M}^{2}, \quad\quad \|\bm{v}\|_{\bm{V}}^{2}:=\|\bm{v}\|_{1}^{2}, \quad \quad\|q\|_{W}^{2}:=\|q\|_{0}^{2}, \vspace{-0.1cm}
\]
with $\| \cdot \|_M := \|\cdot\|_1$. Note that the sum norm $\|(C_{1},C_{2})\|_{\text{sum}}:=\|C_{1}\|_{M}+\|C_{2}\|_{M}$ is equivalent to the $\|(C_{1},C_{2})\|_{\bm{M}}$ norm. Using these spaces, we define the following associated functionals, bilinear and trilinear forms \vspace{-0.2cm}
\begin{equation} \label{eqn7}
\begin{aligned}
      &A_P(\psi,\varphi) := \int_{\Omega} \nabla \psi \cdot \nabla \varphi \,dx, 
      &\quad &J_P(\psi,\varphi) := \int_{\Omega} \nabla^2 \psi : \nabla^2 \varphi \,dx, 
      &\quad &M_C(C,\chi) := \int_{\Omega} C \chi \,dx,\\
      &A_C(C,\chi) := \int_{\Omega} \nabla C \cdot \nabla \chi \,dx, 
      &\quad &M_V(\bm{v},\bm{w}) := \int_{\Omega} \bm{v} \cdot \bm{w} \,dx, 
      &\quad &A_V(\bm{v},\bm{w}) := \int_{\Omega} \nabla \bm{v} : \nabla \bm{w} \,dx,\\
      &B(q,\bm{v}) := \int_{\Omega} q \nabla \cdot \bm{v} \,dx, 
      &\quad &S_C(\varphi;C,\chi) := \int_{\Omega} C \nabla \varphi \cdot \nabla \chi \,dx,  
      &\quad &Q_C(\bm{v};C,\chi) := \int_{\Omega} C\bm{v}  \cdot \nabla \chi \,dx,\\
      &Q_V(\bm{v};\bm{w},\bm{z}) := \int_{\Omega} (\nabla \bm{w}) \bm{v} \cdot \bm{z} \,dx,
      &\quad &S_V(C;\varphi,\bm{v}) := \int_{\Omega} C \nabla \varphi \cdot \bm{v} \,dx, 
      &\quad &l(C;\varphi) := \int_{\Omega} C \varphi \,dx, \\
      &F_{\phi}(\varphi) := \int_{\Omega} f_{\phi} \varphi\,dx,
      &\quad &F_{c_i}(\chi) := \int_{\Omega} f_{c_i} \chi\,dx,
      &\quad & \bm{F}_{u}(\bm{v}) := \int_{\Omega} \bm{f}_{u} \cdot \bm{v}\,dx, 
\end{aligned}
\end{equation} 
for all $C,\chi \in M$, $\varphi,\psi \in \mathring{U}$, $\bm{v},\bm{w},\bm{z} \in \bm{V}$, and $q \in W$. 
Using the above notations, we define the weak solution of system \eqref{eqn3} as follows 
\vspace{-0.2cm}
 \begin{definition} \label{def0}
For given initial data $\left(\bm{u}_0,\; C_{1,0},\; C_{2,0}\right)\in \bm{L}^2(\Omega)\times L^{\infty}(\Omega)\times L^{\infty}(\Omega)$ and forces $f_{\phi} \in L^{\infty}(0,T; L^2(\Omega)), f_{c_i} \in L^{2}(0,T; L^2(\Omega))$ and $\bm{f}_{u} \in L^{2}(0,T; \bm{L}^2(\Omega))$,  we call the tuple ${\{ \phi, (C_1, C_2), \bm{u},\allowbreak p\}}$ a weak solution of \eqref{eqn3} if \vspace{-0.05cm}
\begin{center}
\begin{tabular}{lll}
$C_i \in L^{2}(0,T; M) \cap L^{\infty}(\Omega_T),$ 
&
$\phi \in L^{\infty}(0,T; \mathring{U}),$ 
&
$\partial_t C_i \in L^2(0, T; M'),$ \\
$\bm{u} \in L^{2}(0,T; \bm{V}) \cap L^{\infty}(0,T; \bm{L}^2(\Omega)),$ 
&
$p \in L^{2}(0,T; W),$ 
&
$\partial_t \bm{u} \in L^2(0, T; \bm{V}'),$
\end{tabular}
\end{center}
 and satisfies:\vspace{-0.3cm}
\begin{subequations} \label{eqn6'}
    \begin{align}
        &J_P(\phi,\varphi)+ A_P(\phi,\varphi) = l(C_1;\varphi) - l(C_2;\varphi)+ F_{\phi}(\varphi), &\forall\; \varphi \in \mathring{U},\label{eqn6'a}\\
        &M_C(\partial_tC_i,\chi_i)+ A_C( C_i,\chi_i) + z_iS_C ( \phi; C_i, \chi_i) - Q_C(\bm{u};C_i, \chi_i)=F_{c_i}(\chi_i), &\forall\; \chi_i \in M,\label{eqn6'b}\\
        & M_V(\partial_t \bm{u}, \bm{v}) + A_V( \bm{u}, \bm{v}) + Q_V(\bm{u};\bm{u}, \bm{v}) - B(p, \bm{v}) = - S_V\left(C_1-C_2; \phi, \bm{v}\right)+ \bm{F}_{u}(\bm{v}), &\forall\; \bm{v} \in \bm{V},\label{eqn6'c}\\
        &B(q,\bm{u})=0, &\forall\; q \in W.\label{eqn6'd}
    \end{align}
\end{subequations} 
\end{definition} 
We define the following equivalent  formulation using a divergence-free subspace: 
 \begin{definition} \label{def-weak2}
Let $\widetilde{\bm{V}} := \{\bm{v}\in \bm{V} : B(q,\bm{v})=0 \;\;\forall\; q\in W\}$ and $\widetilde{\bm{V}}^0$  be the closure of ${\widetilde{D}(\Omega)= \{\bm{v} \in C^{\infty}_{0}(\Omega, \mathbb{R}^d)} \mid \operatorname{div} \bm{v} = 0\}\}$ in $\bm{L}^2(\Omega)$ \cite{da2018virtual,lions1996mathematical}. 
For given initial data $\left(\bm{u}_0,\; C_{1,0},\; C_{2,0}\right)\in \bm{L}^2(\Omega)\times L^{\infty}(\Omega)\times L^{\infty}(\Omega)$ and forces $f_{\phi} \in L^{\infty}(0,T; L^2(\Omega)), f_{c_i} \in L^{2}(0,T; L^2(\Omega))$ and $\bm{f}_{u} \in L^{2}(0,T; \bm{L}^2(\Omega))$, we call tuple $\{ \phi, (C_1, C_2), \bm{u}\}$ a weak solution of \eqref{eqn3} if \vspace{-0.1cm}
\begin{center}
\begin{tabular}{lll}
$C_i \in L^{2}(0,T; M) \cap L^{\infty}(\Omega_T),$ 
& $\phi \in L^{\infty}(0,T; \mathring{U}),$ 
&
$\partial_t C_i \in L^2(0, T; M'),$ \\
$\bm{u} \in L^{2}(0,T; \widetilde{\bm{V}}) \cap L^{\infty}(0,T; \widetilde{\bm{V}}^0),$ 
&
$\partial_t \bm{u} \in L^2(0, T; \widetilde{\bm{V}}'),$
\end{tabular}
\end{center}
 and satisfies: \vspace{-0.2cm}
\begin{subequations} \label{eqn6*}
    \begin{align}
        &J_P(\phi,\varphi)+ A_P(\phi,\varphi) = l(C_1;\varphi) - l(C_2;\varphi) + F_{\phi}(\varphi), &\quad \forall\; \varphi \in \mathring{U} ,\label{eqn6*a}\\[-0.1em]
        &M_C(\partial_tC_i,\chi_i)+ A_C( C_i,\chi_i) + z_iS_C ( \phi; C_i, \chi_i) - Q_C(\bm{u};C_i, \chi_i)= F_{c_i}(\chi_i), &\quad \forall\; \chi_i \in M ,\label{eqn6*b}\\[-0.1em]
        & M_V(\partial_t \bm{u}, \bm{v}) + A_V( \bm{u}, \bm{v}) + Q_V(\bm{u};\bm{u}, \bm{v})  = - S_V\left(C_1-C_2; \phi, \bm{v}\right) + \bm{F}_{u}(\bm{v}), &\quad \forall\; \bm{v} \in \widetilde{\bm{V}}.\label{eqn6*c}
    \end{align} 
\end{subequations} 
\end{definition} 
Finally, for fixed $\bm{v} \in \widetilde{\bm{V}}$, the bilinear forms $Q_V(\bm{v};\cdot,\cdot)$ and $Q_C(\bm{v};\cdot,\cdot)$ are skew-symmetric. Consequently, the corresponding terms in \eqref{eqn6'} and \eqref{eqn6*} can be replaced by their equivalent skew-symmetric representations \cite{da2018virtual}:\vspace{-0.2cm}
\begin{equation} \label{skew_2}
\begin{aligned}
Q_C^{\mathrm{skew}}(\bm{v};C,\chi) &= \frac{1}{2}\Big(Q_C(\bm{v};C,\chi)-Q_C(\bm{v};\chi,C)\Big),\\[-0.1em]
Q_V^{\mathrm{skew}}(\bm{v};\bm{w},\bm{z}) &= \frac{1}{2}\Big(Q_V(\bm{v};\bm{w},\bm{z})-Q_V(\bm{v};\bm{z},\bm{w})\Big).
\end{aligned}
\end{equation}
It is worth noting that, although these forms are equivalent at the continuous level, their discrete counterparts generally lead to distinct bilinear forms.
\vspace{-0.1cm}
\subsection{Well-posedness of the weak formulation}
This subsection highlights the fundamental properties of the continuous bilinear and trilinear forms, and outlines key results concerning the existence and uniqueness of the problem \eqref{eqn6*}.

Note that the coercivity of $A_C$ in $H^1(\Omega)$ norm is lacking due to absence of zero-order term. Following the classical strategy outlined in \cite[Ch.~6, p.~282]{ern2004theory}, this issue can be solved by change of variable. In particular, let us first define the following modified forms based on the bilinear forms introduced in \eqref{eqn7}:\vspace{-0.1cm}
\begin{align*}
    \widehat{A}_C(C,\chi) &= A_C(C,\chi) + (C,\chi)_{0},   && \forall\, C,\chi \in M,\\[-0.1em]
      \widehat{F}_{c_i}(\chi_i) &= e^{-t}\,F_{c_i}(\chi_i),  && \forall\, \chi_i \in M,\\[-0.1em]
    \widehat{l}(C;\varphi) &= e^{t}\,l(C;\varphi),  && \forall\, C \in M,\; \varphi \in \mathring{U},\\[-0.1em]
    \widehat{S}_V(C;\varphi, \bm{v}) &= e^{t}S_V(C;\varphi, \bm{v}),  && \forall\, \bm{v} \in \bm{V},\; C \in M,\; \varphi \in \mathring{U}.
\end{align*}
Now, we make a change of variables from $\{\phi, (C_1, C_2), \bm{u}, p\}$ to $\{\phi, (\widehat{C}_1, \widehat{C}_2), \bm{u}, p\}$ in the weak formulation \eqref{eqn6*} by setting $\widehat{C}_i = e^{-t}C_i$. It is easy to verify that the system remains structurally identical, except that $A_C$, $F_{c_i}$, $S_V$, and $l$ are replaced by $\widehat{A}_C$, $\widehat{F}_{c_i}$, $\widehat{S}_V$, and $\widehat{l}$, respectively, such that \vspace{-0.1cm}
\begin{align} \label{c-o-v}
    \widehat{A}_C(\chi,\chi) &\geq \|\chi\|_{1}^2, && \forall\, \chi \in M, \nonumber\\[-0.1em]
    |\widehat{l}(C;\varphi)| &\leq {\beta}_{l} \|C\|_0 \|\varphi\|_0, && \forall\, C \in M,\; \varphi \in \mathring{U},\nonumber\\[-0.1em]
    |\widehat{S}_V(C;\varphi, \bm{v})| &\leq {\beta}_{S_V} \|C\|_{1} \|\varphi\|_{1} \|\bm{v}\|_{1}, && \forall\, \bm{v} \in \bm{V},\; C \in M,\; \varphi \in \mathring{U},
\end{align}
where ${\beta}_{S_V} = {\beta}_{l}\widehat{\beta}_{S_V}$, with ${\beta}_{l}= e^T$ and $\widehat{\beta}_{S_V}$ denoting the associated embedding constant.

Therefore, up to a change of variables, we may always reformulate the system \eqref{eqn6*} so that $\widehat{A}_C$ is coercive. To avoid unnecessary notational complexity, we shall omit the hat superscripts in what follows and revert to the original notation. 
\begin{lemma}\label{lem:form-bounds}
Let $\varphi, \psi \in \mathring{U}$, $C, \chi \in M$, $\bm{v}, \bm{w}, \bm{z} \in \widetilde{\bm{V}}$, and $q \in W$. Then, in view of the change of variables discussed above, the forms defined in \eqref{eqn7} and \eqref{skew_2} satisfy the following properties 
\begin{equation} \label{cts_bound}
\begin{aligned}
    &A_P(\varphi,\varphi) \geq \alpha_{A_P} \|\varphi\|_{1}^2, & &   
    J_P(\varphi,\varphi) \geq  |\varphi|_{2}^2, & &  
    M_C(\chi,\chi) \geq  \|\chi\|_{0}^2, \\
    &A_C(\chi,\chi) \geq  \|\chi\|_{1}^2, & & 
    |l(C;\varphi)| \leq \beta_{l}\|C\|_{0} \|\varphi\|_{0}, & & 
    A_V(\bm{v},\bm{v}) \geq \alpha_{A_V} \|\bm{v}\|_{1}^2, \\
    & |S_C(\varphi; C, \chi)| \leq \beta_{S_C} \|\varphi\|_{2} \|C\|_{1} \|\chi\|_{1}, & & |Q_C^{\mathrm{skew}}(\bm{v}; C, \chi)| \leq \beta_{Q_C} \|\bm{v}\|_{1} \|C\|_{1} \|\chi\|_{1},
     & & Q_C^{\mathrm{skew}}(\bm{v}; C, C) = 0,\\
    & |S_V(C; \varphi, \bm{v})| \leq \beta_{S_V} \|C\|_{1} \|\varphi\|_{1} \|\bm{v}\|_{1}, & &  |Q_V^{\mathrm{skew}}(\bm{v}; \bm{w}, \bm{z})| \leq \beta_{Q_V} \|\bm{v}\|_{1} \|\bm{w}\|_{1} \|\bm{z}\|_{1},
     & & Q_V^{\mathrm{skew}}(\bm{v}; \bm{w}, \bm{w}) = 0
\end{aligned} 
\end{equation} 
where $\beta_{Q_C}$, $\beta_{Q_V}$, and $\beta_{S_C}$ are embedding constants, ${\beta}_{S_V}$ is introduced in equation~\eqref{c-o-v} and $\alpha_{A_P}$ and $\alpha_{A_V}$ are Poincar\'e constants. Moreover, the forms  $A_P, J_P, M_C, A_C, M_v, A_V, B, F_{\phi}, F_{C_i}$ and  $\bm{F}_{u}$ defined in \eqref{eqn7} are continuous with respect to the corresponding norms with continuity constant equal to $1$.   
\end{lemma}
\begin{theorem} [Existence\label{exist-cts}]  
    For $\bm{u}_0 \in \widetilde{\bm{V}}^0$ and $C_{i,0} \in L^{\infty}(\Omega), i=1,2$, there exists a weak 
 solution $\{ \phi, (C_1, C_2), \allowbreak \bm{u}\}$ to the system \eqref{eqn6*} in the sense of Definition \ref{def-weak2}. Furthermore there exists a positive constant $c_{\mathtt{stab}}$ depending on domain $\Omega$, such that \vspace{-0.1cm}
     \begin{align}
    \|\phi\|_{L^{\infty}(0,T; \mathring{U})} + \sum_{i=1}^{2}\|C_i\|_{L^{\infty}(0,T; L^2(\Omega))}+  \sum_{i=1}^{2}\|C_i\|_{L^{2}(0,T; M)}+ \|\bm{u}\|_{L^{\infty}(0,T; \bm{L}^2(\Omega))} + \|\bm{u}\|_{L^{2}(0,T; \bm{V})}  \leq c_{\mathtt{stab}}\Big(\|\bm{u}_0\|_0\nonumber\\ +\sum_{i=1}^2\left(\|C_{i,0}\|_0  + \|f_{c_i}\|_{L^{2}(0,T; L^2(\Omega))}\right) + \|f_{\phi}\|_{L^{\infty}(0,T; L^2(\Omega))}
 + \|\bm{f}_{u}\|_{L^{2}(0,T; \bm{L}^2(\Omega))} \Big). \nonumber
     \end{align}
\end{theorem}
\begin{proof}\!\!.\;
Since the bilinear form $A(\psi, \varphi) = A_P(\psi, \varphi) + J_P(\psi, \varphi)$ is coercive and continuous with respect to the $H^2(\Omega)$ norm, the existence of a weak solution to \eqref{eqn6*} follows from Schauder’s fixed point arguments, as developed in \cite{schmuck2009analysis, he2021mixed, lions1996mathematical}.
The stability estimate can be derived by following the steps outlined below:

\noindent $\bullet\;$ \textbf{a priori estimate for the concentrations:} 
Setting \( \chi_i = C_i(t) \) ($t$ being fixed)  in equation~\eqref{eqn6*b} and applying Hölder’s inequality together with Lemma~\ref{lem:form-bounds}, we have \vspace{-0.05cm}
\begin{equation} \label{priori-pot}
    \frac{1}{2}\frac{d}{dt}\|C_i\|^2_0 + \|C_i\|^2_1 \leq \|C_i\|_{0,4} \|\nabla \phi\|_{0,4} \|\nabla C_i\|_{0,2} + \|f_{c_i}\|_0\|C_i\|_0.
\end{equation}
Integrating the above over the interval \([0, t]\) and using \( C_i \in L^\infty(\Omega_T) \) for \( i = 1, 2 \), along with the Sobolev embedding \( H^1(\Omega) \subset L^4(\Omega) \) and Young’s inequality yields
\begin{equation} \label{tilde-c}
    \frac{1}{2}\|C_i(t)\|^2_0 + \frac{3}{4}\int_0^t \|C_i\|^2_1 \, d\tau 
    \leq \frac{1}{2} \|C_i(0)\|^2_0 + \tilde{c}\,\sqrt{\beta_{S_C}}\int_0^t \|\phi\|_2 \|C_i\|_1 \, d\tau + \int_0^t \|f_{c_i}\|_0^2\, d\tau,
\end{equation}
 where \( \tilde{c} = {\|C_{1}(0)\|_{L^\infty(\Omega)} + \|C_{2}(0)\|_{L^\infty(\Omega)}}\) such that ${\|C_{1}\|_{L^\infty(\Omega_T)}}, {\|C_{2}\|_{L^\infty(\Omega_T)}} \leq \tilde{c}$ (cf. Lemma 4, \cite{schmuck2009analysis}). We again apply Young’s inequality and sum the resulting inequalities for $i=1,2$ to obtain
\begin{equation} \label{eq50}
    \sum_{i=1}^2\left(\frac{1}{2}\|C_i(t)\|^2_0 
    + \frac{1} {2} 
    \int_0^t \|C_i\|^2_1 \, d\tau \right)
    \leq  \sum_{i=1}^2 \left(\frac{1}{2}\|C_i(0)\|^2_0 + \int_0^T\|f_{c_i}\|_0^2 d\tau\right)
    + 2\tilde{c}^2\,\beta_{S_C} \int_0^T \|\phi \|_2^2 \, d\tau.
\end{equation}
Now, we decouple \( C_i \) from \( \phi \). For that, we test equation~\eqref{eqn6*a} with \( \phi \)  and apply Lemma~\ref{lem:form-bounds} to get \vspace{-0.05cm}
\begin{equation} \label{eq51}
\widehat{\alpha}_{A_P}\|\phi\|^2_2 \leq \left(\beta_l\|C_1\|_{0} +\beta_l\|C_2\|_{0} + \|f_{\phi}\|_0\right)\|\phi\|_{2}  , \vspace{-0.05cm}
\end{equation}
where $\widehat{\alpha}_{A_P}=\min\{1, \alpha_{A_P}\}$. Using Young's inequality and integrating over $(0,T)$, above yields \vspace{-0.05cm}
\begin{equation} \label{eq52}
\int_0^T \|\phi\|_2^2 \, d\tau
\le c_{\phi}
\int_0^T \left(\|C_1\|_0^2 + \|C_2\|_0^2\right)\, d\tau + \frac{c_{\phi}}{\beta_l^2}\int_0^T \|f_{\phi}\|_0^2\, d\tau, \quad c_{\phi}=\frac{3\beta_l^2}{\widehat{\alpha}_{A_P}^2}. \vspace{-0.05cm}
\end{equation}
Using the above in \eqref{eq50}, applying Gronwall’s lemma, and taking the supremum over $0\le t\le T$, while choosing $c_c= \exp\left({4\tilde{c}^2\beta_{S_C}c_\phi}T\right)$ and $\mathcal{F}(T)
:=
\sum_{i=1}^2 \int_0^T \|f_{c_i}\|_0^2\, d\tau
+
\frac{2\tilde{c}^2\beta_{S_C}c_\phi}{\beta_l^2}
\int_0^T \|f_{\phi}\|_0^2\, d\tau$, we have  \vspace{-0.05cm}
\begin{equation} \label{bound_C_1} 
\sum_{i=1}^2\left(\sup_{0 \leq t \leq T}\|C_i(t)\|^2_0 
    + \int_0^T \|C_i\|^2_1 \, d\tau \right)
    \leq c_c  \left(\|C_1(0)\|^2_0 + \|C_2(0)\|^2_0+ 2 \mathcal{F}(T)\right). 
\end{equation}

\noindent $\bullet$ \textbf{a priori estimate for the potential:} 
From \eqref{eq51} and the bound \eqref{bound_C_1}, we obtain
\begin{equation}
 \|\phi\|_{L^{\infty}(0,T; \mathring{U})}^2
    \leq c_cc_{\phi} \left( \|C_1(0)\|_0^2 + \|C_2(0)\|_0^2 + 2\mathcal{F}(T)\right)+  \frac{c_{\phi}}{\beta_l^2} \|f_{\phi}\|_{L^{\infty}(0,T; L^2(\Omega))}^2.
\end{equation}
\noindent $\bullet$ \textbf{a priori estimate for velocity:} Setting $\bm{v} = \bm{u}(t)$ (for a given time $t$) in equation~\eqref{eqn6*c}, and applying Hölder's inequality along with Lemma~\ref{lem:form-bounds}, we obtain
\[\frac{1}{2}\frac{d}{dt}\|\bm{u}\|^2_0 + \alpha_{A_V}\|\bm{u}\|^2_1 \leq \;\|C_1\|_{0,4}\|\nabla \phi\|_{0,2}\|\bm{u}\|_{0,4} 
+ \|C_2\|_{0,4}\|\nabla \phi\|_{0,2}\|\bm{u}\|_{0,4} + \|\bm{f}_{u}\|_0\|\bm{u}\|_0.
\]
Integrating the inequality with respect to time from \( 0 \) to \( t \), and using the assumptions that \( C_i \in L^{\infty}(\Omega_T) \) for \( i = 1, 2 \), along with the Sobolev embedding \( H^1(\Omega) \subset L^4(\Omega) \) and  Young’s inequality, we arrive at
\begin{equation}
    \frac{1}{2}\|\bm{u}(t)\|^2_0 + \frac{3\alpha_{A_V}}{4} \int_0^t \|\bm{u}\|^2_1 \, d\tau 
    \leq \frac{1}{2} \|\bm{u}(0)\|^2_0 + \frac{1}{\alpha_{A_V}}\int_0^t \|\bm{f}_{u}\|_0^2 \,d\tau+ 2\tilde{c}\beta_l\sqrt{\widehat{\beta}_{S_V}}\int_0^t \|\phi\|_2 \|\bm{u}\|_1 \, d\tau.
\end{equation}
Here, \( \tilde{c} \) is defined in equation~\eqref{tilde-c}, and the remaining coefficients are introduced in equation~\eqref{c-o-v}. Applying the Young’s inequality, we get
\begin{equation}
    \frac{1}{2} \|\bm{u}(t)\|^2_0 + \frac{\alpha_{A_V} }{2}   \int_0^t \|\bm{u}\|^2_1 \, d\tau 
    \leq \frac{1}{2} \|\bm{u}(0)\|^2_0 +  \frac{1}{\alpha_{A_V}}\int_0^T \|\bm{f}_{u}\|_0^2 \,d\tau + \frac{4\tilde{c}^2\beta_l^2{\widehat{\beta}_{S_V}}}{\alpha_{A_V}} \int_0^T \|\phi\|_2^2 \, d\tau.
\end{equation}
Now, using the bound for \( \phi \) and taking supremum over time $0\leq t\leq T$ (RHS is independent of $t$), we get the final bound 
\begin{align}
\sup_{0 \leq t \leq T}\|\bm{u}(t)\|^2_0 
+ \int_0^T \|\bm{u}\|^2_1 \, d\tau 
&\leq c \Big( \|\bm{u}(0)\|^2_0 
+ \|C_1(0)\|_0^2 + \|C_2(0)\|_0^2 
 + \|f_{\phi}\|_{L^{2}(0,T; L^2(\Omega))}^2 \notag \\
&\qquad \quad+ \sum_{i=1}^{2}\|f_{c_i}\|_{L^{2}(0,T; L^2(\Omega))}^2 
 + \|\bm{f}_{u}\|_{L^{2}(0,T; \bm{L}^2(\Omega))}^2 \Big).
\end{align}
Finally, the combined a priori estimates of all yield the desired stability estimate.
\end{proof}
\begin{theorem}[Uniqueness\label{thm_uniq_continuous}] Under the conditions of Definition~\ref{def0}, suppose the problem data satisfy the following condition \vspace{-0.1cm}
\begin{equation} \label{c_sol}
    K := c_\mathtt{sol} \big( \|\bm{u}_0\|_0 + \|C_{1,0}\|_0 + \|C_{2,0}\|_0 \big) +  2\bar{c} \sqrt{c_{\phi}} < 1 
    \quad \text{with} \quad \alpha_{A_V}^{-1} \bar{c} \frac{1 + K}{1 - K} < 1,\vspace{-0.1cm}
\end{equation}
where $c_{\mathtt{sol}}$, $\bar{c}$, and $c_\phi$ are positive constants depending only on the domain \( \Omega \) and the final time \( T \), defined in proof. Then, the problem~\eqref{eqn6'} admits a unique solution \( \{ \phi, (C_1, C_2), \bm{u}, p \} \).
\end{theorem}
\begin{proof}\!\!. \;
Throughout the proof, we will use the notations of Theorem~\ref{exist-cts}. For brevity, let us first define \vspace{-0.05cm}
\begin{align} \label{*}
  c_1 = c_{\mathtt{stab}}\left(\|\bm{u}_0\|_0 + \|C_{1,0}\|_0 + \|C_{2,0}\|_0\right),   \vspace{-0.05cm}
\end{align}
where \( c_{\mathtt{stab}} \) is a constant defined in Theorem \ref{exist-cts}. We assume that \( \{ \phi_1, (C_{1,1}, C_{2,1}), \bm{u}_1, p_1 \} \) and  
\( \{ \allowbreak \phi_2, \allowbreak (C_{1,2}, \allowbreak C_{2,2}), \allowbreak \bm{u}_2, \allowbreak p_2 \} \) are two solutions to equation \eqref{eqn6'} and set $\hat{\bm{u}} = \bm{u}_1 - \bm{u}_2$, $\hat{p} = p_1 - p_2$, $\hat{\phi}=\phi_1-\phi_2$ and $\hat{C}_{i}= C_{i,1} - C_{i,2}$  with $i=1,2$. We now proceed with the proof in the following four steps:
\vspace{0.1cm}

\noindent \textbf{Step 1.} Testing \eqref{eqn6'a} with $\hat{\phi}$ for solutions $\phi_1$ and $\phi_2$, subtracting the resulting equations, and following the steps of \eqref{eq52}, we have  \vspace{-0.05cm}
  \begin{equation}\label{eq:phi_final_bound}
    \|\hat{\phi}\|_{L^2(0,T; H^2(\Omega))} 
    \leq \sqrt{c_{\phi}} \left( 
        \|\hat{C}_1\|_{L^2(0,T; H^1(\Omega))} 
        + \|\hat{C}_2\|_{L^2(0,T; H^1(\Omega))}
    \right). \vspace{-0.05cm}
\end{equation}
\textbf{Step 2.} Similarly, testing \eqref{eqn6'b} with $\hat{C}_{i}$ and adding and subtracting the appropriate terms, we get \vspace{-0.05cm}
    \begin{align}
    \frac{1}{2} \frac{d}{dt} \|\hat{C}_i\|_0^2 + &\|\hat{C}_{i}\|_1^2 = \int_\Omega \Big( \bm{u}_1  C_{i,1}  - \bm{u}_2  C_{i,2} \Big) \cdot \nabla \hat{C}_{i} \, dx - z_i \int_{\Omega} \Big( C_{i,1} \nabla \phi_1  - C_{i,2} \nabla \phi_2 \Big) \cdot \nabla \hat{C}_{i}  \, dx  \nonumber \\
    & = \int_{\Omega} \hat{\bm{u}} \; C_{i,1} \cdot \nabla\hat{C}_{i} \, dx + \int_{\Omega} \bm{u}_2\;  \hat{C}_{i} \cdot \nabla \hat{C}_{i} \, dx    - z_i \int_{\Omega} \left[ (\hat{C}_{i} \nabla \phi_1) \cdot \nabla \hat{C}_{i} + (C_{i,2} \nabla \hat{\phi}) \cdot \nabla \hat{C}_{i} \right] \, dx. \nonumber
    \end{align}
   Using Hölder's inequality and  noting that the second term vanishes (Lemma~\ref{lem:form-bounds}), we obtain
\begin{align} 
     \frac{1}{2}\frac{d}{dt}\|\hat{C}_i\|_0^2+ (1-c_1c_2)\|\hat{C}_{i}(\cdot, t)\|_1^2 &\leq c^{*}c_2 \|\hat{\bm{u}}(\cdot, t)\|_1\|\hat{C}_{i}(\cdot, t)\|_1 + c^{*}c_2 \|\hat{\phi}(\cdot, t)\|_2\|\hat{C}_{i}(\cdot, t)\|_1,  \nonumber 
\end{align}
 where, \( c_1 \) is defined at the beginning of the proof, \( c_2 \) denotes the square of the norm of the embedding \( H^1(\Omega) \hookrightarrow L^4(\Omega) \) and by~\cite{ern2004theory}, Lemma~6.2 (Aubin), there exists a constant \(c^* > 0\) such that \(\sum_{i,j=1}^2 \|C_{i,j}(t)\|_1 + \|\bm{u}(t)\|_1 \le c^*\) for a.e. \(t \in [0,T]\). 

\vspace{0.15cm}
Now, integrating over $[0,T]$ and applying the Cauchy--Schwarz inequality in time, while noting that $\hat{C}_i(0) = 0$, we obtain 
    \begin{equation} \label{c_1c_2}
    (1 - c_1 c_2 ) \|\hat{C}_{i}\|_{L^2(0,T; H^1(\Omega))} \leq c_2 c^{*}\left( \|\hat{\bm{u}}\|_{L^2(0,T; \bm{H}^1(\Omega))} + \|\hat{\phi}\|_{L^2(0,T; {H}^2(\Omega))}\right), \quad i=1,2.  \nonumber
    \end{equation}
     Summing the above inequality for $i=1,2$ and employing equation \eqref{eq:phi_final_bound}, we obtain 
     \begin{equation} \label{c_1c_2c}
    (1 - c_1 c_2- 2\sqrt{c_{\phi}}c_2 c^{*} )\sum_{i=1}^{2} \|\hat{C}_{i}\|_{L^2(0,T; H^1(\Omega))} \leq  2c_2 c^{*} \|\hat{\bm{u}}\|_{L^2(0,T; \bm{H}^1(\Omega))}. 
    \end{equation}
\textbf{Step 3.} In the same way, using the velocity field equation \eqref{eqn6*c} with Lemma~\ref{lem:form-bounds}, we get
    \begin{align*}
    \frac{1}{2} \frac{d}{dt} \| \hat{\bm{u}} \|_0^2 + \alpha_{A_V}\| \hat{\bm{u}} \|_1^2 \leq  \int_\Omega -\left[ (\nabla \bm{u}_1)\hat{\bm{u}} 
    +  (C_{1,2} - C_{2,2}) \nabla \hat{\phi} 
    + (\hat{C}_{1} - \hat{C}_{2}) \nabla \phi_1\right] \cdot \hat{\bm{u}} \, dx.
    \end{align*}
    Using Hölder's inequality and Sobolev embedding, we obtain
    \[
    \frac{1}{2} \frac{d}{dt} \| \hat{\bm{u}} \|_0^2 + \alpha_{A_V}\|\hat{\bm{u}}\|_1^2 \leq  c_2 c^{*} \|\hat{\bm{u}}\|_1^2 +  2c_2 c^{*}\|\hat{\phi}\|_2 \|\hat{\bm{u}}\|_1 + c_1 c_2  \|\hat{C}_1\|_1 \|\hat{\bm{u}}\|_1 + c_1 c_2  \|\hat{C}_2\|_1 \|\hat{\bm{u}}\|_1.
    \]
    Integrating over \([0, t]\), using Cauchy-Schwarz inequality and equation \eqref{eq:phi_final_bound}, we get\vspace{-0.2cm}
 \begin{equation*}
 \alpha_{A_V}\|\hat{\bm{u}}\|_{L^2(0,T; \bm{H}^1(\Omega))} 
    \leq c_2 c^{*} \|\hat{\bm{u}}\|_{L^2(0,T; \bm{H}^1(\Omega))} + ( c_1c_2 + 2c_2c^{*}\sqrt{c_{\phi}}) \sum_{i=1}^{2}\|\hat{C}_i\|_{L^2(0,T; H^1(\Omega))}. \vspace{-0.2cm}
\end{equation*}
    Now, we set $c_{\mathtt{sol}} := c_{\mathtt{stab}}c_2 $ and  \( \bar{c} = c_2 c^* \). Recalling equation~\eqref{*}, it is straightforward to verify that the constant $K$ defined in equation~\eqref{c_sol} satisfies   $K=  c_1 c_2+ 2c_2 c^{*}\sqrt{c_{\phi}}$. Using equation~\eqref{c_1c_2c} in the above inequality, we can rewrite it as \vspace{-0.2cm}
\begin{align*}
    \|\hat{\bm{u}}\|_{L^2(0,T; \bm{H}^1(\Omega))} &\leq \alpha_{A_V}^{-1} c_2 c^{*}\left(1 + 2K(1 - K  ) ^{-1} \right) \|\hat{\bm{u}}\|_{L^2(0,T; \bm{H}^1(\Omega))}\\
    &=  \alpha_{A_V}^{-1}\bar{c}(1+K)(1-K)^{-1} \|\hat{\bm{u}}\|_{L^2(0,T; \bm{H}^1(\Omega))}.
\end{align*}
Using the given hypothesis, we obtain that \( \bm{u}_1 \) and \( \bm{u}_2 \) are equal. It follows from \eqref{c_1c_2c} that \( C_{i,1} = C_{i,2}, \quad i = 1, 2 \) and hence, equation \eqref{eq:phi_final_bound} implies $\phi_1=\phi_2$.
\vspace{0.1cm}

\noindent\textbf{Step 4.} Finally, the function \( \hat{p} = p_1 - p_2 \) satisfies the relation \vspace{-0.1cm}
    \[
    B(p_1 - p_2, \bm{v}) = 0, \quad \forall\, \bm{v} \in \bm{V}. \vspace{-0.1cm}
    \]
    By the inf-sup condition associated with the bilinear form \( B(\cdot, \cdot) \) (which implies that \( \mathrm{div}(\bm{V}) = W \); see \cite{girault2012finite}, Chapter I, Section 2), there exists a test function \( \bm{v} \in \bm{V} \) such that \( \nabla \cdot \bm{v} = p_1 - p_2 \). Then it follows that \( p_1 = p_2 \).
\end{proof}
\section{Virtual formulation of the problem}\label{VE_formation}  \vspace{-0.1cm} In this section, we introduce the VE spaces along with the discrete bilinear and trilinear forms. We restrict our presentation to the two-dimensional case; analogous discretization in three dimensions follows directly from the availability of standard 3D elements \cite{chen2022conforming, beirao2020stokes}. Consider a sequence of decompositions $\{\Omega_h\}_h$ of the domain $\Omega$ into general polygonal elements $E$ with boundary $\partial E$,  diameter $h_E$ and measure $|E|$. The global mesh size is defined as $h$ := $\sup\limits_{E \in \Omega_h} h_E$. The mesh $\Omega_h$ consists of straight edges $e$ with length $h_e$, and for each edge $e\subset \partial E$, $\bm{n}_E^e$ represents the unit normal vector to $e$, pointing outward to $E$. The natural restriction of each of the forms in equation \eqref{eqn7} will be denoted by a superscript $E$, so for instance, $A_C^E( C,\chi) = \int_{E} \nabla C \cdot \nabla\chi \,dx.$ 
We now introduce the following regularity condition on each element $E$ in $\Omega_h$: \vspace{-0.3cm}
\begin{assumption} \label{assump}
There exists positive constants $c_1, c_2$ such that

   \textbf{(A1)} $E$ is star-shaped with respect to a ball of radius $\geq c_1h_E.$ 
   
   \textbf{(A2)} The distance between any two vertices of $E$ is at least $c_2h_E.$ 
\end{assumption} 
 Next, for any arbitrary $k\in \mathbb{N}$,  we define \vspace{-0.1cm}
\begin{itemize}
    \item $\mathbb{P}_k(\omega):$ the space of polynomials on set $\omega\subseteq \mathbb{R}^2$ of degree $\leq k$ (with extended notations $\mathbb{P}_{-2}(\omega) = \mathbb{P}_{-1}(\omega):= \{0\})$,
    \item $\mathbb{B}_k^0(\partial E):= \{v \in C^0(\partial E)\;\; \text{s.t.}\;\; v|_e \in \mathbb{P}_k(e), \quad \forall\;\;\; e \subset \partial E\}$, \vspace{0.1cm}
    \item $\mathbb{B}_k^1(\partial E) := \{ \varphi \in C^1(\partial E) \text{ s.t. } \varphi|_e \in \mathbb{P}_{\max\{3,k\}}(e),\, \partial_{\bm{n}} \varphi|_e \in \mathbb{P}_{k-1}(e), \; \forall\;  e \subset \partial E, \; k \geq 2\}.$

    \item The polynomial projections $\Pi_k^{0,E}:L^2(E)\to\mathbb{P}_k(E)$, $\;$
$\Pi_k^{\nabla,E}:H^1(E)\to\mathbb{P}_k(E)$, and
$\Pi_k^{\Delta,E}:H^2(E)\to\mathbb{P}_k(E)$ are defined as follows \vspace{-0.2cm}
\begin{align} \allowbreak
& m^E(\Pi_k^{0,E} v, q_k) = m^E(v, q_k), \;\; \forall \;  q_k\in\mathbb{P}_k(E)\quad  \text{s.t.} \quad 
   m^E(w,q_k):=\int_E w q_k\,dx, \quad   w \;\in L^2(E), \nonumber \\
& \left\{
    \begin{aligned}
    & a^E(\Pi_k^{\nabla,E} v, q_k) = a^E(v, q_k), \;\; \forall \ q_k\in\mathbb{P}_k(E) \quad \text{s.t.} 
      && a^E(w,q_k):=\int_E \nabla w\cdot\nabla q_k\,dx, \quad   \; w \;\in H^1(E),\\
    & \int_{\partial E}\Pi_k^{\nabla,E} v\,ds = \int_{\partial E} v\,ds, \quad  \forall \; v\; \in H^1(E),
    \end{aligned}
  \right. \nonumber \\
& \left\{
    \begin{aligned}
    & a^E_\Delta(\Pi_k^{\Delta,E} v, q_k) = a^E_\Delta(v, q_k), \;\; \forall \ q_k\in\mathbb{P}_k(E) \quad  \text{s.t.} 
      \quad  a^E_\Delta(w,q_k):=\int_E \nabla^2 w:\nabla^2 q_k\,dx, \quad w \;\in H^2(E),\\
    & \int_{\partial E}\Pi_k^{\Delta,E} v\,ds = \int_{\partial E} v\,ds, \;\;
     \int_{\partial E}\nabla(\Pi_k^{\Delta,E} v)\,ds = \int_{\partial E}\nabla v\,ds, \quad \forall \;v \;\in H^2(E).
    \end{aligned}
  \right. \nonumber
\end{align}
\end{itemize}
Similarly, we define the obvious exensions to the vector-valued counterparts of the above polynomial projections: the $\bm{L}^2$-projection by $\bm{\Pi}_k^{0,E}$, the $\bm{H}^1$-projection by $\bm{\Pi}_k^{\nabla,E}$ and the tensor-valued ${\mathbb{L}}^2$-projection by $\widetilde{\bm{\Pi}}_k^{0,E}$. We recall below a standard polynomial approximation result.
\begin{proposition} \cite{brenner2008mathematical}\label{brenner2008mathematica}
Under the mesh regularity assumptions \eqref{assump}, 
for every $s$ with $0 \le s \le k+1$ and every $v \in H^s(E)$, 
there exists a polynomial $v_{\pi} \in \mathbb{P}_k(E)$ such that \vspace{-0.1cm}
\begin{equation} \label{BH}
    |v - v_{\pi}|_{l,E} \le c\, h_E^{s-l} |v|_{s,E}, \quad 0 \le l \le s. \vspace{-0.1cm}
\end{equation}
\end{proposition}
\subsection{VE space for concentrations}
  For $\mathfrak{J}\geq 1$,  we define the following  enhanced local VE space \cite{beirao2013basic,ahmad2013equivalent} \vspace{-0.1cm}
\begin{align*}
{M}_{h}^{\mathfrak{J}}(E) := \{ \chi_h \in H^1(E)  \;:\; &\chi_h|_{\partial E} \in \mathbb{B}_\mathfrak{J}^0(\partial E), \ \Delta \chi_h|_E \in \mathbb{P}_{\mathfrak{J}}(E),\\ &m^E(\Pi_\mathfrak{J}^{\nabla,E}\chi_h, q)= m^E(\chi_h, q), \quad \forall q \in \mathbb{P}_{\mathfrak{J}}(E)\setminus \mathbb{P}_{\mathfrak{J}-2}(E)\},    
\end{align*}
for which the dofs are given by \cite{ahmad2013equivalent}: \vspace{-0.1cm}
\begin{itemize}
    \item \textbf{(D$_{c}$1)}: the values $v(x_v)$ at the vertices \( x_v \) of \( E \), \vspace{-0.05cm}
    \item \textbf{(D$_{c}$2)}: for \( \mathfrak{J} \geq 2 \), the  moments
    $\frac{1}{h_e}\displaystyle\int_e v\,q  \, ds$, \quad $\forall$ \( q \in \mathbb{P}_{\mathfrak{J}-2}(e) \) and  \( e \subset \partial E \), \vspace{-0.05cm}

    \item \textbf{(D$_{c}$3)}: for \( \mathfrak{J} \geq 2 \), the internal moments $
    \frac{1}{|E|} \displaystyle\int_E  v \,q \, dx$
   \quad $\forall$ \( q \in \mathbb{P}_{\mathfrak{J}-2}(E) \).\vspace{-0.15cm}
\end{itemize}
Then, the  global VE space  is defined as  \vspace{-0.2cm}
\begin{align*}
        &{M}_{h}^\mathfrak{J} :=\big\{ \chi_h \in M : \quad \chi_h|_E \in {M}_{h}^{\mathfrak{J}}(E), \quad \forall E\in \Omega_h\big\}, 
\end{align*}
together with the global projection operators $\Pi_\mathfrak{J}^{\nabla}$ and $\Pi_\mathfrak{J}^{0}$ defined through their element-wise restrictions. 
%
Finally, we define the vectorial counterpart for the discrete concentrations  $\bm{M}_h^\mathfrak{J}:=M_h^\mathfrak{J}\times M_h^\mathfrak{J}.$ \vspace{-0.15cm}
\begin{proposition} \label{interpolation_concentration} \cite{mora2015virtual} 
(Approximation property of discrete concentration space)
Under the mesh regularity assumption \eqref{assump} on $\Omega_h$, for any $\chi \in M \cap H^{s+1}(\Omega_h)$ with $0 \le s \le \mathfrak{J}$, $\mathfrak{J} \ge 1$, there exists an approximation $\chi_I \in M_h^\mathfrak{J}$ such that \vspace{-0.05cm}
\begin{equation}\label{known_interpolant_bound_fully_c}
\|\chi - \chi_I\|_{0,E} + h_E |\chi - \chi_I|_{1,E} 
\leq c \, h_E^{s+1} |\chi|_{s+1,E}, \quad \forall E \in \Omega_h.
\end{equation}
\end{proposition}
\subsection{Construction of Virtual element space of \texorpdfstring{$\mathring{U}$}{U}} Similarly to the concentration space, for $\mathcal{K} \ge 2$, we define the enhanced local space \cite{antonietti2016c, brezzi2013virtual} \vspace{-0.05cm}
\begin{align*}
    {U}_{h}^\mathcal{K}(E):=\big\{ \varphi_h \in H^2(E) \;:\; &\Delta^2 \varphi_h|_E  \in \mathbb{P}_{\mathcal{K}}, \;\varphi_h|_{\partial E} \in \mathbb{B}_\mathcal{K}^1(\partial E),\\ &m^E(\Pi_\mathcal{K}^{\Delta,E}\varphi_h, q)= m^E(\varphi_h, q), \quad \forall q \in \mathbb{P}_{\mathcal{K}}(E)\setminus \mathbb{P}_{\mathcal{K}-4}(E)\big\}, \nonumber
\end{align*}
with the dofs  \vspace{-0.1cm}
\begin{itemize}
    \item \textbf{(D$_{\phi}$1)}: for \( \mathcal{K} \geq 2 \), the values \( \varphi(x_v) \), \( \partial_x \varphi(x_v) \), and \( \partial_y \varphi(x_v) \) at the vertices \( x_v \) of \( E \), \vspace{-0.02cm}
    
    \item \textbf{(D$_{\phi}$2)}: for \( \mathcal{K} \geq 4 \), the  moments
    $\frac{1}{h_e}\displaystyle\int_e \varphi \,q  \, ds$, \quad $\forall$ \( q \in \mathbb{P}_{\mathcal{K}-4}(e) \),  \( e \subset \partial E \), \vspace{-0.02cm}
    
    \item \textbf{(D$_{\phi}$3)}: for \( \mathcal{K} \geq 3 \), the moments $
    \displaystyle\int_e  \partial_{\bm{n}} \varphi \,q \, ds $, \quad $\forall$ \( q \in \mathbb{P}_{\mathcal{K}-3}(e) \), \( e \subset \partial E \),\vspace{-0.05cm}

    \item \textbf{(D$_{\phi}$4)}: for \( \mathcal{K} \geq 4 \), the internal moments $
    \frac{1}{|E|} \displaystyle\int_E  \varphi\, q \, dx$,
   \quad $\forall$ \( q \in \mathbb{P}_{\mathcal{K}-4}(E) \), \vspace{-0.2cm}
\end{itemize} 
such that the global VE space is defined as \vspace{-0.1cm}
\begin{align}
&\hspace{3cm}\mathring{U}_h^\mathcal{K} := \big\{ \varphi_h \in \mathring{U} : \varphi_h|_E \in U_h^\mathcal{K}(E), \ \forall E \in \Omega_h \big\}, \nonumber
\end{align} 
As before, we denote by $\Pi_\mathcal{K}^{0}$, $\Pi_\mathcal{K}^{\Delta}$, and $\Pi_\mathcal{K}^{\nabla}$ the global projection operators.

\begin{proposition} \label{interpolation_potential} 
\cite{brezzi2013virtual, chen2022conforming, bonnet2025conforming}
(Approximation property of discrete potential space) 
Under the mesh regularity assumption \eqref{assump} on $\Omega_h$, for any 
$\varphi \in \mathring{U}\cap H^{s+2}(\Omega_h)$ with $0 \le s \le \mathcal{K}-1$, $\mathcal{K} \ge 2$, 
there exists an approximation $\varphi_I \in \mathring{U}_h^\mathcal{K}$ such that \vspace{-0.15cm}
\begin{equation} \label{known_interpolant_bound_fully_p}
|\varphi - \varphi_I|_{l,E} \le c \, h_E^{s-l+2} |\varphi|_{s+2,E}, \quad l=0,1,2, \quad \forall E \in \Omega_h. \vspace{-0.1cm}
\end{equation}
\end{proposition}
\subsection{Construction of the Virtual Element Space of \texorpdfstring{\;$\bm{V}\; \text{and}\; W$}{V×W}} For $\mathcal{L} \geq 2$, we define the following enhanced local VE space \cite{da2018virtual} \vspace{-0.05cm}
\begin{equation}
    {\bm{V}}_{h}^\mathcal{L}(E) :=
    \left\{ \bm{v}_h \in [H^1(E)]^2 \; \middle| \;
    \begin{aligned}
        &\bm{v}_h|_{\partial E} \in [\mathbb{B}_\mathcal{L}^0(\partial E)]^2, \\[-0.1em]
        &\text{(i) } \operatorname{div} \bm{v}_h \in \mathbb{P}_{\mathcal{L}-1}(E), \\[-0.1em]
        &\text{(ii) } - \Delta\bm{v}_h - \nabla s \in \mathbb{G}_\mathcal{L}(E)^\perp, \quad \text{for some}\; s \in L^2(E) \setminus \mathbb{R}, \\[-0.2em]
        & \text{(iii) } m^E(\bm{v}_h-\bm{\Pi}_\mathcal{L}^{\nabla,E}\bm{v}_h, \bm{g}_\mathcal{L}^\perp) = 0, \quad \forall\bm{g}_\mathcal{L}^\perp \in
  \mathbb{G}_{\mathcal{L}}(E)^\perp \setminus \mathbb{G}_{\mathcal{L}-2}(E)^\perp,       
    \end{aligned} 
    \right\}\nonumber
\end{equation} 
\noindent where $\mathbb{G}_\mathcal{L}(E) := \nabla \mathbb{P}_{\mathcal{L}+1}(E) \subset [\mathbb{P}_\mathcal{L}(E)]^2$ and $ \mathbb{G}_\mathcal{L}(E)^\perp := \bm{x}^\perp \mathbb{P}_{\mathcal{L}-1}(E) \subset [\mathbb{P}_\mathcal{L}(E)]^2$ together with ${\bm{x}^\perp := (x_2, -x_1)}$.

The corresponding dofs for ${\bm{V}}_{h}^{\mathcal{L}}(E)$ are given by the following linear operators $\bm{D}_{\bm{v}}$ \cite{da2018virtual} 
\begin{itemize}
    \item \textbf{(D$_{\bm{v}}$1)}: the values $\bm{v}_h(x_v)$ at the vertices \( x_v \)  of the element \( E \),
    \item \textbf{(D$_{\bm{v}}$2)}: the values of \( \bm{v}_h \) at \( \mathcal{L} - 1 \) distinct points along any edge \( e \subset \partial E \),
    \item \textbf{(D$_{\bm{v}}$3)}: the moments, $\displaystyle\int_E \bm{v}_h\cdot \bm{g}_{\mathcal{L}-2}^\perp  \, dx$, \quad $\forall$ \( \bm{g}_{\mathcal{L}-2}^\perp \in \mathbb{G}_{\mathcal{L}-2}(E)^\perp \),
    \item \textbf{(D$_{\bm{v}}$4)}: the moments,  $\displaystyle\int_E  (\nabla\cdot\bm{v}_h) q \, dx$,
   \quad $\forall$ \( q \in \mathbb{P}_{\mathcal{L}-1}(E)\setminus \mathbb{R}\).
\end{itemize}
The global VE space is given by \vspace{-0.05cm}
\begin{equation}
    {\bm{V}}_{h}^\mathcal{L} :=\big\{ \bm{v}_h \in \bm{V} : \quad \bm{v}_h|_E \in {\bm{V}}_{h}^{\mathcal{L}}(E), \quad \forall E\in \Omega_h\big\},\nonumber \vspace{-0.05cm}
\end{equation}
and the vector-valued global projections \( \bm{\Pi}_\mathcal{L}^{0} \), \( \bm{\Pi}_\mathcal{L}^{\nabla} \), and \( \widetilde{\bm{\Pi}}_\mathcal{L}^{0} \) 
 can be defined in the same manner as in the scalar case.
 
Choosing $\mathcal{L}^* = \mathcal{L}-1$, we define the finite-dimensional VE space 
${W}_h^{\mathcal{L}^*} \subset W$ by piecewise polynomials of degree $\mathcal{L}^*$: \vspace{-0.2cm}
\begin{align}
   &{W}_{h}^{\mathcal{L}^*} := \big\{ w_h \in W : w_h|_E \in {W}_h^{\mathcal{L}^*}(E), \ \forall E \in \Omega_h \big\}\;\; \text{s.t.}\;\;&{W}_{h}^{\mathcal{L}^*}(E) := \mathbb{P}_{\mathcal{L}^*}(E).\nonumber
\end{align}
\begin{proposition} \label{interpolation_velocity} \cite{da2018virtual}
(Approximation property of discrete velocity field space) 
Under the mesh regularity assumption \eqref{assump} on $\Omega_h$, for any 
$\bm{v} \in \bm{V}\cap \bm{H}^{s+1}(\Omega_h)$ with $0 < s \le \mathcal{L}$, $\mathcal{L} \ge 2$, 
there exists an approximation $\bm{v}_I \in \bm{V}_h^\mathcal{L}$ such that \vspace{-0.12cm}
\begin{equation} \label{known_interpolant_bound_fully_v}
\|\bm{v} - \bm{v}_I\|_{0,E} + h_E |\bm{v} - \bm{v}_I|_{1,E} 
\le c \, h_E^{s+1} |\bm{v}|_{s+1,E}, \quad \forall E \in \Omega_h. \vspace{-0.12cm}
\end{equation}
\end{proposition}
\subsection{The Discrete Forms} The next task in the construction of the method is to introduce the discrete version of each bilinear (and trilinear) forms defined in equation \eqref{eqn7} using the computable projections. These local forms are first defined elementwise and then assembled globally by summation over all elements.  
In view of the change of variable in \eqref{c-o-v}, for $C_h, \chi_h \in M_h^\mathfrak{J}(E)$, 
$\psi_h, \varphi_h \in \mathring{U}_h^\mathcal{K}(E)$, and $\bm{v}_h, \bm{w}_h \in \bm{V}_h^\mathcal{L}(E)$, 
we define the discrete local forms as follows \vspace{-0.1cm}
\begin{subequations} \label{eqn8}
\begin{align}
    A_{P,h}^{E}(\psi_h,\varphi_h) &= A_{P}^{E}(\Pi^{\nabla,E}_\mathcal{K} \psi_h, \Pi^{\nabla,E}_\mathcal{K} \varphi_h) + s_{P,\nabla}^E(\psi_h - \Pi^{\nabla,E}_\mathcal{K} \psi_h, \varphi_h - \Pi^{\nabla,E}_\mathcal{K} \varphi_h), \label{eqn8a} \\[-0.2em]
     J_{P,h}^{E}(\psi_h,\varphi_h) &= J_{P}^{E}(\Pi^{\Delta,E}_\mathcal{K} \psi_h, \Pi^{\Delta,E}_\mathcal{K} \varphi_h) + s_{P,\Delta}^E(\psi_h - \Pi^{\Delta,E}_\mathcal{K} \psi_h, \varphi_h - \Pi^{\Delta,E}_\mathcal{K} \varphi_h), \label{eqn8b} \\[-0.2em]
     M_{V,h}^{E}(\bm{v}_h,\bm{w}_h) &= M_{V}^{E}(\bm{\Pi}^{0,E}_\mathcal{L} \bm{v}_h, \bm{\Pi}^{0,E}_\mathcal{L} \bm{w}_h) + s_{V,0}^E(\bm{v}_h - \bm{\Pi}^{0,E}_\mathcal{L} \bm{v}_h, \bm{w}_h - \bm{\Pi}^{0,E}_\mathcal{L} \bm{w}_h), \label{eqn8c} \\[-0.2em]
    A_{V,h}^{E}(\bm{v}_h,\bm{w}_h) &= A_{V}^{E}(\bm{\Pi}^{{\nabla},E}_\mathcal{L} \bm{v}_h, \bm{\Pi}^{{\nabla},E}_\mathcal{L} \bm{w}_h) + s_{V,\nabla}^E(\bm{v}_h - \bm{\Pi}^{{\nabla},E}_\mathcal{L} \bm{v}_h, \bm{w}_h - \bm{\Pi}^{{\nabla},E}_\mathcal{L} \bm{w}_h), \label{eqn8d} \\[-0.2em]
    M_{C,h}^{E}(C_h,\chi_h) &= M_{C}^{E}(\Pi^{0,E}_\mathfrak{J} C_h, \Pi^{0,E}_\mathfrak{J} \chi_h) + s_{C,0}^E(C_h - \Pi^{0,E}_\mathfrak{J} C_h, \chi_h - \Pi^{0,E}_\mathfrak{J} \chi_h), \label{eqn8e} \\[-0.2em]
     A_{C,h}^{E}(C_h,\chi_h) &= (\nabla \Pi^{\nabla,E}_\mathfrak{J} C_h, \nabla \Pi^{\nabla,E}_\mathfrak{J} \chi_h) + s_{C,\nabla}^E(C_h - \Pi^{\nabla,E}_\mathfrak{J} C_h, \chi_h - \Pi^{\nabla,E}_\mathfrak{J} \chi_h)
     + M_{C,h}^{E}(C_h,\chi_h), \label{eqn8f} 
\end{align}
\end{subequations}
with the positive definite, symmetric, stabilizing bilinear forms satisfying \vspace{-0.2cm}
\begin{subequations} \label{eqn9}
\begin{align}   
    \beta_{\nabla,1} |\varphi_h|_{1,E} &\leq s_{P,\nabla}^E(\varphi_h, \varphi_h) \leq \alpha_{\nabla,1} |\varphi_h|_{1,E}, \quad  \forall \;\varphi_h \in \mathring{U}{}_{h}^{\mathcal{K}}(E)\; \text{with} \; \Pi^{\nabla,E}_\mathcal{K}(\varphi_h)=0, \label{eqn9a}\\[-0.2em]
     \beta_{\Delta} |\varphi_h|_{2,E} &\leq s_{P,\Delta}^E(\varphi_h, \varphi_h) \leq \alpha_{\Delta} |\varphi_h|_{2,E}, \quad  \forall \;\varphi_h \in \mathring{U}{}_{h}^{\mathcal{K}}(E)\; \text{with} \; \Pi^{\Delta,E}_\mathcal{K}(\varphi_h)=0,  \label{eqn9b}\\[-0.2em]
     \hat{\beta}_{0} \|\bm{w}_h\|_{0,E} &\leq s_{V,0}^E( \bm{w}_h, \bm{w}_h) \leq \hat{\alpha}_{0} \|\bm{w}_h\|_{0,E}, \quad  \forall \;\bm{w}_h \in \bm{V}_{h}^{\mathcal{L}}(E)\; \text{with} \; \bm{\Pi}^{0,E}_\mathcal{L}(\bm{w}_h)=\bm{0}, \label{eqn9c}\\[-0.2em]
     \hat{\beta}_{\nabla} |\bm{w}_h|_{1,E} &\leq s_{V,\nabla}^E( \bm{w}_h, \bm{w}_h) \leq \hat{\alpha}_{\nabla} |\bm{w}_h|_{1,E}, \quad  \forall \;\bm{w}_h \in \bm{V}_{h}^{\mathcal{L}}(E)\; \text{with} \; \bm{\Pi}^{\nabla,E}_\mathcal{L}(\bm{w}_h)=\bm{0}, \label{eqn9d}\\[-0.2em]
      \beta_0 \|\chi_h\|_{0,E} &\leq s_{C,0}^E(\chi_h, \chi_h) \leq \alpha_0 \|\chi_h\|_{0,E}, \quad  \forall \;\chi_h \in {M}_{h}^{\mathfrak{J}}(E)\; \text{with} \; \Pi^{0,E}_\mathfrak{J}(\chi_h)=0,\label{eqn9e}\\[-0.2em]
    \beta_{\nabla,2} \|\chi_h\|_{1,E} &\leq s_{C,\nabla}^E(\chi_h, \chi_h) \leq \alpha_{\nabla,2} \|\chi_h\|_{1,E}, \quad  \forall \;\chi_h \in {M}_{h}^{\mathfrak{J}}(E)\; \text{with} \; \Pi^{\nabla,E}_\mathfrak{J}(\chi_h)=0, \label{eqn9f}
\end{align} 
\end{subequations}
where, \(\beta_0, \alpha_0, \beta_{\nabla,1}, \alpha_{\nabla,1}, \beta_{\nabla,2}, \alpha_{\nabla,2}, \beta_{\Delta}, \alpha_{\Delta}, \hat{\beta}_0, \hat{\alpha}_0, \hat{\beta}_{\nabla}, \; \text{and}\; \hat{\alpha}_{\nabla}\) are positive constants independent of $h$.  
    The stabilization terms can be chosen using the standard dofi-dofi stabilization approach \cite{beirao2013basic, da2017divergence}. 
    
We do not introduce any approximation for local bilinear form $B^E(q_h,\bm{v}_h)$, as it can be computed exactly \cite{da2017divergence} using the dofs \textbf{(D$_{\bm{v}}$1)}, \textbf{(D$_{\bm{v}}$2)} and \textbf{(D$_{\bm{v}}$4)}.
For the remaining local bilinear and trilinear forms, using $\varphi_h \in \mathring{U}{}_{h}^{\mathcal{K}}(E)$, $C_h, \chi_h \in {M}_{h}^{\mathfrak{J}}(E)$ and $\bm{v}_h,\bm{w}_h, \bm{z}_h \in \bm{V}_{h}^{\mathcal{L}}(E)$, we set 
\begin{align} \label{eqn10}
    l_{h}^E(C_{h};\varphi_h) &= e^t\int_{E} (\Pi^{0,E}_\mathfrak{J}C_{h}) (\Pi^{0,E}_\mathcal{K}\varphi_h) \,dx, \nonumber\\[-0.1cm]
    S_{C,h}^E(\varphi_h; C_h,\chi_h) &= \int_{E} \big[(\Pi^{0,E}_\mathfrak{J} C_h) (\bm{\Pi}^{0,E}_{\mathcal{K}-1}\nabla \varphi_h)\big] 
    \cdot \big[\bm{\Pi}^{0,E}_{\mathfrak{J}-1}\nabla \chi_h\big] \,dx,  \nonumber\\[-0.1cm]
    Q_{C,h}^{\mathrm{skew},E}(\bm{v}_h;C_h,\chi_h) &= \frac{1}{2} \Bigg[\int_{E} \big[\bm{\Pi}^{0,E}_{\mathcal{L}}\bm{v}_h \Pi^{0,E}_\mathfrak{J} C_h \big]
    \cdot \big[\bm{\Pi}^{0,E}_{\mathfrak{J}-1}\nabla \chi_h\big] \,dx  - \int_{E} \big[\bm{\Pi}^{0,E}_{\mathcal{L}}\bm{v}_h\big] \cdot \big[\bm{\Pi}^{0,E}_{\mathfrak{J}-1}\nabla C_h 
    \Pi^{0,E}_{\mathfrak{J}}\chi_h \big] \Bigg]dx,   \nonumber\\[-0.1cm]
    Q_{V,h}^{\mathrm{skew},E}(\bm{v}_h;\bm{w}_h,\bm{z}_h) &= \frac{1}{2} \Bigg[\int_{E} 
(\widetilde{\bm{\Pi}}_{\mathcal{L}-1}^{0,E}\nabla \bm{w}_h )\big[\bm{\Pi}^{0,E}_{\mathcal{L}}\bm{v}_h\big] 
    \cdot \big[\bm{\Pi}^{0,E}_{\mathcal{L}}\bm{z}_h\big] \,dx \notag  - 
 (\widetilde{\bm{\Pi}}_{\mathcal{L}-1}^{0,E}\nabla \bm{z}_h ) \big[\bm{\Pi}^{0,E}_{\mathcal{L}}\bm{v}_h\big] 
    \cdot \big[\bm{\Pi}^{0,E}_{\mathcal{L}}\bm{w}_h\big] \Bigg]dx,  \nonumber\\[-0.1cm]
    S_{V,h}^E(C_h; \varphi_h, \bm{v}_h) &= e^t\int_{E} \big[({\Pi}^{0,E}_{\mathfrak{J}} C_h) (\bm{\Pi}^{0,E}_{\mathcal{K}-1}\nabla \varphi_h) \big]
    \cdot \big[\bm{\Pi}^{0,E}_{\mathcal{L}} \bm{v}_h\big] \,dx, \nonumber \\[-0.1cm]
    F_{\phi, h}^E(\varphi_h) & = \int_{E} \big[\Pi^{0,E}_{\mathcal{K}-1} f_{\phi}(t) \big]\, \varphi_h\,dx,    
\end{align} \vspace{-0.1cm}
The other linear forms $F_{c_i,h}^E(\chi_h)$ and $\bm{F}_{u,h}^E(\bm{v}_h)$ are defined analogously to $F_{\phi,h}^E(\varphi_h)$.\vspace{-0.1cm}
\begin{lemma} \label{lemma1} 
\textbf{(Continuity and Coercivity)}
    The discrete functionals and global bilinear forms defined using \eqref{eqn8} and \eqref{eqn9} satisfy the following properties \vspace{-0.2cm}
\begin{align*} 
    M_{C,h}(C_h,\chi_h) &\leq \widetilde{\beta}_{M_C} \|C_h\|_0 \|\chi_h\|_0, &&
    A_{P,h}(\psi_h,\varphi_h)  \leq \widetilde{\beta}_{A_P} |\psi_h|_1 |\varphi_h|_1, && A_{C,h}(C_h,\chi_h) \leq \widetilde{\beta}_{A_C} \|C_h\|_1 \|\chi_h\|_1\nonumber \\[-0.2em]
    J_{P,h}(\psi_h,\varphi_h)  &\leq \widetilde{\beta}_{J_P} |\psi_h|_2 |\varphi_h|_2, 
    && M_{V,h}(\bm{v}_h,\bm{w}_h)  \leq \widetilde{\beta}_{M_V} \|\bm{v}_h\|_0 \|\bm{w}_h\|_0,  &&
    A_{V,h}(\bm{v}_h,\bm{w}_h)  \leq \widetilde{\beta}_{A_V} |\bm{v}_h|_1 |\bm{w}_h|_1, \nonumber \\[-0.2em]
    M_{C,h}(\chi_h,\chi_h) &\geq \widetilde{\alpha}_{M_C} \|\chi_h\|_0^2, &&
    A_{P,h}(\varphi_h,\varphi_h)  \geq \widetilde{\alpha}_{A_P} \|\varphi_h\|_1^2, 
    && A_{C,h}(\chi_h,\chi_h) \geq \widetilde{\alpha}_{A_C} \|\chi_h\|_1^2, \nonumber \\[-0.2em]
    J_{P,h}(\varphi_h,\varphi_h)  &\geq \widetilde{\alpha}_{J_P} |\varphi_h|_2^2, 
    && M_{V,h}(\bm{w}_h,\bm{w}_h) \geq \widetilde{\alpha}_{M_V} \|\bm{w}_h\|_0^2, &&
    A_{V,h}(\bm{w}_h,\bm{w}_h)  \geq \widetilde{\alpha}_{A_V} \|\bm{w}_h\|_1^2, \nonumber \\[-0.2em]
    |F_{\phi,h}(\varphi_h)| &\leq \widetilde{\beta}_{F_{\phi}} \,\|f_{\phi}\|_{0}\,\|\varphi_h\|_{0}, && |F_{c_i,h}(\chi_h)| 
\leq \widetilde{\beta}_{F_{c_i}} \,\|f_{c_i}\|_{0}\,\|\chi_h\|_{0}, 
     && |\bm{F}_{u,h}(\bm{v}_h)| 
\leq \widetilde{\beta}_{\bm{F}_{u}} \,\|\bm{f}_{u}\|_{0}\,\|\bm{v}_h\|_{0}, \nonumber
\end{align*}
for $i=1,2$, and for all $C_h, \chi_h \in {M}_{h}^{\mathfrak{J}}$,  $\psi_h,\varphi_h \in \mathring{U}{}_{h}^{\mathcal{K}}$ and $\bm{v}_h,\bm{w}_h \in \bm{V}_{h}^{\mathcal{L}}$. In particular, $\widetilde{\alpha}_{M_C}=\min\{1, \beta_0\}$, $\widetilde{\alpha}_{M_V}=\min\{1, \hat{\beta_0}\}$ and $\widetilde{\alpha}_{J_P}=\min\{1, \beta_{\Delta}\}$.
\end{lemma}\vspace{-0.4cm}
\begin{proof}\!\!.\;
    The continuity bounds follow  directly from the application of the Cauchy-Schwarz inequality, combined with the stability properties \eqref{eqn9} and the coercivity estimates are an immediate consequence of the stability properties  \eqref{eqn9}, combined with Young’s and the triangle inequalities \cite{dehghan2023optimal}.
\end{proof}
The forms in \eqref{eqn10} can be naturally extended to the corresponding continuous spaces \cite{da2018virtual}. Moreover, they are uniformly continuous on their respective parent spaces: \vspace{-0.2cm}
\begin{lemma} \label{lemma2}
There exist positive constants $\widetilde{\beta}_{l}$, $\widetilde{\beta}_{S_C}$, $\widetilde{\beta}_{Q_V}$, $\widetilde{\beta}_{Q_C}$ and  $\widetilde{\beta}_{S_V}$such that, for all 
$C,\chi \in M$, $\varphi \in \mathring{U}$, and $\bm{v}, \bm{w}, \bm{z}\in \bm{V}$, the following continuity estimates hold \vspace{-0.2cm} 
\begin{align} 
    l_{h}(C;\varphi) &\leq \widetilde{\beta}_{l}\|C\|_0 \|\varphi\|_0,
    \qquad
    &S_{C,h}(\varphi; C, \chi) \leq \widetilde{\beta}_{S_C} \|C\|_1 \|\varphi\|_2 \|\chi\|_1, \nonumber \\[-0.2em]
    Q_{V,h}^{\mathrm{skew}}(\bm{v};\bm{w},\bm{z}) &\le \widetilde{\beta}_{Q_V}\,\|\bm{v}\|_1 \|\bm{w}\|_1 \|\bm{z}\|_1,\quad & Q_{C,h}^{\mathrm{skew}}(\bm{v};C,\chi)
    \le \widetilde{\beta}_{Q_C}\,\|\bm{v}\|_1 \|C\|_1 \|\chi\|_1, \nonumber \\[-0.2em]
        S_{V,h}(C;\varphi,\bm{v}) &\leq \widetilde{\beta}_{S_V} \|C\|_1 \|\varphi\|_1 \|\bm{v}\|_1. \nonumber 
\end{align}
\end{lemma}\vspace{-0.4cm}
\begin{proof}
    The continuity of $Q_{V,h}^{\mathrm{skew}}$ follows from \cite{da2018virtual, dehghan2023optimal}; the remaining estimates can be established similarly.
\end{proof}
Before closing this section, we recall the Sobolev inequality (see \cite{aldbaissy2018full}, Lemma 2.2):  \vspace{-0.15cm}
\begin{equation} \label{sobolev_ine}
    \|f\|_{0,4} \leq c \|f\|_{1}^{\frac{1}{2}} \|f\|_{0}^{\frac{1}{2}}.\vspace{-0.15cm}
\end{equation}
\section{Fully-discrete Scheme and Well-Posedness} \label{fuul_discete}
In this section, we introduce the fully discrete formulation of the problem by discretizing the temporal domain in addition to the spatial discretization by VEM. The temporal domain is discretized using the backward Euler method with uniform step size $\tau = T/N$ such that $t^n=n\tau$ and $0=t^0<t^1\ldots t^N=T$. For a sequence of functions $\{\sigma^n\}_{n=0}^{N}\subset L^2(\Omega)$, we define \vspace{-0.15cm}
\[
\sigma^n := \sigma(t^n), \qquad 
\delta_t \sigma^n := \frac{\sigma^n - \sigma^{n-1}}{\tau}. \vspace{-0.15cm}
\] 
For the external forces, we introduce the following notation: 
$f_{\phi}^n:=f_{\phi}(t^n)$, $f_{c_i}^n:=f_{c_i}(t^n)$ and $\bm{f}_{u}^n:=\bm{f}_{u}(t^n)$.
From the discrete forms introduced in \eqref{eqn8} and \eqref{eqn10}, the fully discrete VE formulation reads as: 
For $n=1,2, \ldots N$, find $\{ \phi_h^n, (C_{1,h}^n, C_{2,h}^n), \bm{u}_h^n,p_h^n\} \in \mathring{U}_h^\mathcal{K} \times \bm{M}_h^\mathfrak{J} \times \bm{V}_h^\mathcal{L} \times W_h^\mathcal{L^*}$ such that \vspace{-0.1cm}
\begin{subequations} \label{eqn14}
    \begin{align}
&J_{P,h}(\phi_h^n,\varphi_h) + A_{P,h}(\phi_h^n,\varphi_h) = l_h(C_{1,h}^n,\varphi_h) - l_h(C_{2,h}^n,\varphi_h) + F_{\phi,h}(\varphi_h), \label{eqn14a}\\[-0.1em]
        &M_{C,h}(\delta_t C_{i,h}^n,\chi_{i,h}) + A_{C,h}(C_{i,h}^n,\chi_{i,h}) + z_i S_{C,h}(\phi_h^n; C_{i,h}^n, \chi_{i,h}) - Q_{C,h}^{\mathrm{skew}}(\bm{u}_h^n;C_{i,h}^n, \chi_{i,h}) = F_{c_i,h}(\chi_{i,h}), \label{eqn14b}\\[-0.1em]
        &M_{V,h}(\delta_t \bm{u}_h^n, \bm{v}_h) + A_{V,h}(\bm{u}_h^n, \bm{v}_h) + Q_{V,h}^{\mathrm{skew}}(\bm{u}_h^n;\bm{u}_h^n, \bm{v}_h) - B(p_h^n, \bm{v}_h)  =- S_{V,h}(C_{1,h}^n-C_{2,h}^n;\phi_h^n, \bm{v}_h) + \bm{F}_{u,h}(\bm{v}_h), \label{eqn14c}\\[-0.4em]
        &B_h(q_h,\bm{u}_h^n) = 0, \label{eqn14d}
 \end{align}
\end{subequations}
for all $\{ \varphi_h, (\chi_{1,h}, \chi_{2,h}), \bm{v}_h, q_h\} \in \mathring{U}_h^\mathcal{K} \times \bm{M}_h^\mathfrak{J} \times \bm{V}_h^\mathcal{L} \times W_h^\mathcal{L^*}$ with initial conditions $C_{i,h}^0 = C_{i,I}(0)$ and $\bm{u}_h^0 = \bm{u}_I(0)$, where \(C_{i,I}(0)\) and \(\bm{u}_I(0)\) are suitable approximations of \(C_{i,0}\) and \(\bm{u}_0\). 

Since  
     $
    \text{div}\; \bm{V}_h^\mathcal{L}\subseteq W_h^\mathcal{L^*},$ $\mathcal{L^*}= \mathcal{L}-1,
    $
the pressure equation in \eqref{eqn14}
     implies that $\bm{u}_h^n \in \bm{V}_h^\mathcal{L}$ is exactly divergence-free. More generally,
by introducing the discrete  space~\cite{da2017divergence} \vspace{-0.07cm}
\begin{equation*}
\widetilde{\bm{V}}_h^\mathcal{L}
:= \{\bm{w}_h^n \in \bm{V}_h^\mathcal{L}
\;\text{s.t.}\;
B(q_h,\bm{w}_h^n)=0,
\;\; \forall \;q_h \in W_h^{\mathcal{L}^*} \}, \vspace{-0.07cm}
\end{equation*}
It is straightforward to verify that $\widetilde{\bm{V}}_h^\mathcal{L} \subset \widetilde{\bm{V}}$. 
This yields a weak problem equivalent to ~\eqref{eqn14}:
For $n=1,2, \ldots N$, find $\{ \phi_h^n, (C_{1,h}^n, C_{2,h}^n), \bm{u}_h^n\} \in \mathring{U}_h^\mathcal{K} \times \bm{M}_h^\mathfrak{J} \times \widetilde{\bm{V}}_h^\mathcal{L} $ such that \vspace{-0.07cm}
\begin{subequations} \label{eqn16}
    \begin{align}
        &J_{P,h}(\phi_h^n,\varphi_h) + A_{P,h}(\phi_h^n,\varphi_h) = l_h(C_{1,h}^n;\varphi_h) - l_h(C_{2,h}^n;\varphi_h)+ F_{\phi,h}(\varphi_h), \label{eqn16a}\\[-0.15em]
        &M_{C,h}(\delta_tC_{i,h}^n,\chi_{i,h}) + A_{C,h}(C_{i,h}^n,\chi_{i,h}) + z_i S_{C,h}(\phi_h^n; C_{i,h}^n, \chi_{i,h}) - Q_{C,h}^{\mathrm{skew}}(\bm{u}_h^n;C_{i,h}^n, \chi_{i,h})=F_{c_i,h}(\chi_{i,h})
, \label{eqn16b}\\[-0.15em]
        & M_{V,h}(\delta_t \bm{u}_h^n, \bm{v}_h) + A_{V,h}(\bm{u}_h^n, \bm{v}_h) + Q_{V,h}^{\mathrm{skew}}(\bm{u}_h^n;\bm{u}_h^n, \bm{v}_h) = - S_{V,h}(C_{1,h}^n-C_{2,h}^n;\phi_h^n, \bm{v}_h)+ \bm{F}_{u,h}(\bm{v}_h), \label{eqn16c}
    \end{align}
\end{subequations} 
for all $\{ \varphi_h, (\chi_{1,h}, \chi_{2,h}), \bm{v}_h\} \in \mathring{U}_h^\mathcal{K} \times \bm{M}_h^\mathfrak{J} \times \widetilde{\bm{V}}_h^\mathcal{L} $ with initial conditions $C_{i,h}^0 = C_{i,I}(0)$ and $\bm{u}_h^0 = \bm{u}_I(0)$, where \(C_{i,I}(0)\) and \(\bm{u}_I(0)\) are suitable approximations of \(C_{i,0}\) and \(\bm{u}_0\). \vspace{-0.3cm}
\begin{proposition}\cite{boffi2013mixed, da2018virtual}\label{interpolation_velocity_kernel} 
(Approximation property of the discrete kernel space $\widetilde{\bm{V}}_h^\mathcal{L}$)
Under the mesh regularity assumption \eqref{assump} on $\Omega_h$, the continuous kernel space 
$\widetilde{\bm{V}}$ is approximated by its discrete counterpart $\widetilde{\bm{V}}_h^\mathcal{L}$ 
with the same order of accuracy as the subspace $\bm{V}_h^\mathcal{L}$ approximates $\bm{V}$. 
In particular, for any $\bm{v} \in \widetilde{\bm{V}} \cap \bm{H}^{s+1}(\Omega)$ with $0 < s \le \mathcal{L}$, $\mathcal{L} \ge 2$, we infer \vspace{-0.15cm}
\begin{equation} \label{interpolation_velocity_discret}
 \inf_{\bm{v}_h \in \widetilde{\bm{V}}_h^\mathcal{L}, \;\bm{v}_h \neq 0} \|\bm{v} - \bm{v}_h\|_{0} 
\;+\; h \inf_{\bm{v}_h \in \widetilde{\bm{V}}_h^\mathcal{L},\;\bm{v}_h \neq 0} |\bm{v} - \bm{v}_h|_{1} 
\;\;\le \;\;c \, h^{s+1} |\bm{v}|_{s+1}. \vspace{-0.15cm}
\end{equation}
\end{proposition}
\begin{remark}
An alternative approach to discretizing~\eqref{eqn14} is to replace the discrete skew-symmetric terms  
\(Q_{V,h}^{\mathrm{skew}}(\bm{u}_h^n;\bm{u}_h^n, \bm{v}_h)\) and  
\(Q_{C,h}^{\mathrm{skew}}(\bm{u}_h^n;C_{i,h}^n, \chi_{i,h})\)  
with the \emph{non-skew} discrete trilinear forms  
\(Q_{V,h}(\bm{u}_h^n;\bm{u}_h^n, \bm{v}_h)\) and  
\(Q_{C,h}(\bm{u}_h^n;C_{i,h}^n, \chi_{i,h})\).  
While the corresponding continuous forms  coincide, the non-skew discrete trilinear forms may fail to preserve skew-symmetry due to projection operators affecting the divergence-free condition. In contrast, the skew trilinear forms maintain skew-symmetry by construction and allow for straightforward stability analysis. However, they may lead to a non-mass-conservative scheme~\cite{Ankur_andrea_paper2}. 
\end{remark}
\subsection{Well-Posedness Analysis} \label{well_posed_ness}
In this section, we establish the existence and uniqueness of the solution to the discrete VE problem \eqref{eqn16} by employing the Schauder fixed-point theorem \cite[Chapter~9]{evans2022partial}. 
For simplicity, we set $\bm{X}_h :=  {M}_h^\mathfrak{J} \times {M}_h^\mathfrak{J} \times \widetilde{\bm{V}}_h^\mathcal{L}$ and $\bm{Y}_h :=    \bm{X}_h \times \mathring{U}_h^\mathcal{K}$ equipped with the norms \vspace{-0.25cm}
\begin{align*}
\|\big((\lambda_{1,h}^{n}, \lambda_{2,h}^{n}),\bm{w}_{h}^{n}\big)\|_{\bm{X}} \;\text{or}\;
\|{\big(\bm{\lambda}_{h}^{n},\bm{w}_{h}^{n}\big)}\|_{\bm{X}}
&:= \left(\|\bm{\lambda}_{h}^{n}\|^2_{\bm{M}}+ \|\bm{w}_{h}^{n}\|^2_{\bm{V}}\right)^{\frac{1}{2}}, \\[-0.2em]
\|{\big(\bm{\lambda}_{h}^{n},\bm{w}_{h}^{n},\psi_h^{n}\big)}\|_{\bm{Y}}
&:= \left(\|\bm{\lambda}_{h}^{n}\|^2_{\bm{M}}+ \|\bm{w}_{h}^{n}\|^2_{\bm{V}} + \|\psi_h^{n}\|^2_{U}\right)^{\frac{1}{2}}.
\end{align*}
The proof for well-posedness is developed through a sequence of lemmas and a main theorem. It begins by introducing a mapping $\bm{T}:\bm{X}_h\to\bm{X}_h$ and showing that it maps a closed, convex, bounded set $\mathscr{B}_h\subset \bm{X}_h$ into itself and is continuous. Schauder’s theorem \cite[Chapter~9]{evans2022partial} then yields the existence of a fixed point corresponding to a solution of \eqref{eqn16}. Uniqueness is established separately.

We start by introducing a operator $\bm{T}:\bm{X}_h\to\bm{X}_h$ by \vspace{-0.1cm}
\[
\bm{T}\big((\lambda_{1,h}^{n}, \lambda_{2,h}^{n}),\bm{w}_{h}^{n}\big)
:=\big((\widehat{C}_{1,h}^{n}, \widehat{C}_{2,h}^{n}),\widehat{\bm{u}}_{h}^{n}\big), \vspace{-0.1cm}
\]
where image of $ \bm{T}$ is the last three components of the solution of the following linearized form of problem \eqref{eqn16}: For $n=1,2, \ldots N$, find $\{ \widehat{\phi}_h^n, (\widehat{C}_{1,h}^n, \widehat{C}_{2,h}^n), \widehat{\bm{u}}_h^n\} $ such that  \vspace{-0.15cm}
\begin{subequations} \label{eqn17}
    \begin{align}
        &J_{P,h}(\widehat{\phi}_h^n,\varphi_h) + A_{P,h}(\widehat{\phi}_h^n,\varphi_h) = l_h(\lambda_{1,h}^n;\varphi_h) - l_h(\lambda_{2,h}^n;\varphi_h)+ F_{\phi,h}(\varphi_h),\label{eqn17a}\\[-0.15em]
        &M_{C,h}(\delta_t\widehat{C}_{i,h}^n,\chi_{i,h}) + A_{C,h}(\widehat{C}_{i,h}^n,\chi_{i,h}) + z_i S_{C,h}(\widehat{\phi}_h^n; \lambda_{i,h}^n, \chi_{i,h}) - Q_{C,h}^{\mathrm{skew}}(\bm{w}_h^n;\widehat{C}_{i,h}^n, \chi_{i,h})=F_{c_i,h}(\chi_{i,h}), \label{eqn17b}\\[-0.15em]
        & M_{V,h}(\delta_t \widehat{\bm{u}}_h^n, \bm{v}_h) + A_{V,h}(\widehat{\bm{u}}_h^n, \bm{v}_h) + Q_{V,h}^{\mathrm{skew}}(\bm{w}_h^n;\widehat{\bm{u}}_h^n, \bm{v}_h) + S_{V,h}(\lambda_{1,h}^n-\lambda_{2,h}^n;\widehat{\phi}_h^n, \bm{v}_h)=\bm{F}_{u,h}(\bm{v}_h), \label{eqn17c}
    \end{align}
\end{subequations} 
for all $\{ \varphi_h, (\chi_{1,h}, \chi_{2,h}), \bm{v}_h\} \in \bm{Y}_h $  with initial conditions $C_{i,h}^0 = C_{i,I}(0)$ and $\bm{u}_h^0 = \bm{u}_I(0)$.

The above system can be written equivalently as \vspace{-0.25cm} 
\begin{equation}\label{eqn18} 
\begin{aligned}
\widehat{\bm{A}}_{\bm{\lambda}_{h}^{n},\bm{w}_{h}^{n}}
\big((\widehat{\phi}_{h}^{n}, \widehat{\bm{C}}_{h}^{n}, \widehat{\bm{u}}_{h}^{n}),
(\varphi_h,\bm{\chi}_{h},\bm{v}_{h}) \big)
&= \sum_{i=1}^2 M_{C,h}(\widehat{C}_{i,h}^{n-1},\chi_{i,h})
+  M_{V,h}(\widehat{\bm{u}}_h^{\,n-1}, \bm{v}_h) \\[-0.8em]
&\quad + \tau\sum_{i=1}^2 F_{c_i,h}(\chi_{i,h}) 
+ \tau F_{\phi,h}(\varphi_h)
+ \tau \bm{F}_{u,h}(\bm{v}_h),
\end{aligned} \vspace{-0.25cm}
\end{equation}
where,
$\bm{\lambda}_h^n = (\lambda_{1,h}^n, \lambda_{2,h}^n), \quad 
\widehat{\bm{C}}_{h}^{n} = (\widehat{{C}}_{1,h}^n, \widehat{{C}}_{2,h}^n), \quad
\bm{\chi}_{h} = (\chi_{1,h}, \chi_{2,h}), $ and \vspace{-0.1cm}
\begin{equation} \label{eqn18_new}
\begin{aligned}
\widehat{\bm{A}}_{\bm{\lambda}_{h}^{n},\bm{w}_{h}^{n}}
\left(({\psi}_{h}^{n}, {\bm{D}}_{h}^{n}, {\bm{U}}_{h}^{n}), (\varphi_h,\bm{\chi}_{h},\bm{v}_{h}) \right) = &\bm{A} \left(({\psi}_{h}^{n}, {\bm{D}}_{h}^{n}, {\bm{U}}_{h}^{n}), (\varphi_h,\bm{\chi}_{h},\bm{v}_{h}) \right)  \\[-0.1cm]
& + \tau \bm{B}_{\bm{\lambda}_{h}^{n},\bm{w}_{h}^{n}}
\left(({\psi}_{h}^{n}, {\bm{D}}_{h}^{n}, {\bm{U}}_{h}^{n}), (\varphi_h,\bm{\chi}_{h},\bm{v}_{h})\right), 
\end{aligned} \vspace{-0.15cm}
\end{equation}
 for all $\{ {\psi}_{h}^{n}, {\bm{D}}_{h}^{n}, {\bm{U}}_{h}^{n}\} \in \bm{Y}_h$ such that \vspace{-0.1cm}
 \begin{equation}
 \begin{aligned}
&\bm{A}\left(({\psi}_{h}^{n}, {\bm{D}}_{h}^{n}, {\bm{U}}_{h}^{n}), (\varphi_h,\bm{\chi}_{h},\bm{v}_{h})\right):=M_{C,h}({D}_{1,h}^{n},\chi_{1,h})+M_{C,h}({D}_{2,h}^{n},\chi_{2,h})+M_{V,h}( {\bm{U}}_h^{n}, \bm{v}_h)+ \tau J_{P,h}({\psi}_h^n,\varphi_h)\nonumber\\[-0.1em]
 &\hspace{3.7cm} +\tau A_{P,h}({\psi}_h^n,\varphi_h)+\tau A_{C,h}( {D}_{1,h}^n,\chi_{1,h})  +\tau A_{C,h}( {D}_{2,h}^n,\chi_{2,h})+\tau A_{V,h}( {\bm{U}}_h^n, \bm{v}_h),\\[-0.1em]
& \bm{B}_{\bm{\lambda}_{h}^{n},\bm{w}_{h}^{n}}\left(({\psi}_{h}^{n}, {\bm{D}}_{h}^{n}, {\bm{U}}_{h}^{n}), (\varphi_h,\bm{\chi}_{h},\bm{v}_{h})\right) := l_h(\lambda_{1,h}^n;\varphi_h) - l_h(\lambda_{2,h}^n;\varphi_h) + S_{C,h}( {\psi}_h^n; \lambda_{1,h}^n, \chi_{1,h}) \nonumber\\[-0.1em]
&\hspace{4.20cm} - S_{C,h}( {\psi}_h^n; \lambda_{2,h}^n, \chi_{2,h}) - Q_{C,h}^{\mathrm{skew}}(\bm{w}_h^n;{D}_{1,h}^n, \chi_{1,h}) - Q_{C,h}^{\mathrm{skew}}(\bm{w}_h^n;{D}_{2,h}^n, \chi_{2,h})  \nonumber\\[-0.1em]
&\hspace{4.2cm} + Q_{V,h}^{\mathrm{skew}}(\bm{w}_h^n;{\bm{U}}_h^n, \bm{v}_h)+ S_{V,h}(\lambda_{1,h}^n-\lambda_{2,h}^n;{\psi}_h^n, \bm{v}_h).
\end{aligned}\vspace{-0.2cm}
\end{equation}
\begin{lemma} \label{lemma4}
\label{l_INFSUP} 
There exist constants $\widehat{\alpha}, \widehat{\beta} > 0$ such that for given 
$(\bm{\lambda}_{h}^{n}, \bm{w}_{h}^{n})  \in \bm{X}_h$ satisfying
  $ \|{(\bm{\lambda}_{h}^{n}, \bm{w}_{h}^{n})}\|_{\bm{X}} \allowbreak \leq \frac{\widehat{\alpha}}{2(\widetilde{\beta}_{S_C}+\widetilde{\beta}_{S_V}+ \widetilde{\beta}_{Q_C}+ \widetilde{\beta}_{Q_V})}$,
 the following
discrete inf-sup condition holds \vspace{-0.4cm}
\begin{equation}\label{INF_Ah}
\begin{aligned}
\sup_{\substack{(\varphi_h,\bm{\chi}_{h},\bm{v}_{h}) \in \bm{Y}_h \\ (\varphi_h,\bm{\chi}_{h},\bm{v}_{h}) \neq \mathbf{0}}}
\frac{\widehat{\bm{A}}_{\bm{\lambda}_{h}^{n},\bm{w}_{h}^{n}}
\left(({\psi}_{h}^{n}, {\bm{D}}_{h}^{n}, {\bm{U}}_{h}^{n}), (\varphi_h,\bm{\chi}_{h},\bm{v}_{h}) \right)}{\left\|(\varphi_h,\bm{\chi}_{h},\bm{v}_{h})\right\|_{\bm{Y}}} &\geq \widehat{\beta}\widetilde{\alpha}_{M_C}\Vert \bm{D}_{h}^{n}\Vert_{0}+\widehat{\beta}\widetilde{\alpha}_{M_V}\Vert\bm{U}_{h}^n\Vert_{0}\\[-1.2em]
&\;\;\;+\tau\dfrac{\widehat{\alpha}}{2}
\left\|({\psi}_{h}^{n}, {\bm{D}}_{h}^{n}, {\bm{U}}_{h}^{n})\right\|_{\bm{Y}}.
\end{aligned}\vspace{-0.15cm}
\end{equation}
Consequently, the problem \eqref{eqn17} admits a unique solution.
\end{lemma} \vspace{-0.4cm}
 \begin{proof}\!\!.\;
From Lemma~\ref{lemma1}, we have \( 0 <  \widetilde{\alpha}_{M_C}, \widetilde{\alpha}_{M_V} \leq 1 \). Let us define $ \widetilde{\alpha}  := \min\{1,\widetilde{\alpha}_{J_P},\widetilde{\alpha}_{A_P},\widetilde{\alpha}_{A_C}, \widetilde{\alpha}_{A_V}\}$. Then, for any sufficiently small $\tau$ such that $0 < \widetilde{\alpha}\,\tau \leq 1$, we have \vspace{-0.07cm}
\begin{equation*}
\begin{aligned}
 {\bm{A}}
\left(({\psi}_{h}^{n}, {\bm{D}}_{h}^{n}, {\bm{U}}_{h}^{n}), ({\psi}_{h}^{n}, {\bm{D}}_{h}^{n}, {\bm{U}}_{h}^{n}) \right)  &\geq \widetilde{\alpha}_{M_C}\Vert \bm{D}_{h}^{n}\Vert_{0}^2+\widetilde{\alpha}_{M_V}\Vert\bm{U}_{h}^n\Vert_{0}^2+\tau\widetilde{\alpha}
\left\|({\psi}_{h}^{n}, {\bm{D}}_{h}^{n}, {\bm{U}}_{h}^{n})\right\|_{\bm{Y}}^2 \nonumber\\[-0.1em]
&  \geq \widetilde{\alpha}_{M_C}^2\Vert \bm{D}_{h}^{n}\Vert_{0}^2+\widetilde{\alpha}_{M_V}^2\Vert\bm{U}_{h}^n\Vert_{0}^2+(\tau\widetilde{\alpha})^2
\left\|({\psi}_{h}^{n}, {\bm{D}}_{h}^{n}, {\bm{U}}_{h}^{n})\right\|_{\bm{Y}}^2 \nonumber\\[-0.1em]
&\geq \frac{1}{3}\big(\widetilde{\alpha}_{M_C}\Vert \bm{D}_{h}^{n}\Vert_{0}+\widetilde{\alpha}_{M_V}\Vert\bm{U}_{h}^n\Vert_{0}+\tau\widetilde{\alpha}
\left\|({\psi}_{h}^{n}, {\bm{D}}_{h}^{n}, {\bm{U}}_{h}^{n})\right\|_{\bm{Y}}\big)^2. \nonumber
\end{aligned} \vspace{-0.15cm}
\end{equation*}
Multiplying and dividing the RHS by \( c_4>0 \) such that \( \tau \widetilde{\alpha} c_4 \geq 1 \), we get  
\begin{equation}
\begin{aligned}
 {\bm{A}}
\big(\underbrace{({\psi}_{h}^{n}, {\bm{D}}_{h}^{n}, {\bm{U}}_{h}^{n})}_{F}, ({\psi}_{h}^{n}, {\bm{D}}_{h}^{n}, {\bm{U}}_{h}^{n}) \big)  &\geq \frac{1}{3c_4^2}\big(\underbrace{c_4\widetilde{\alpha}_{M_C}\Vert \bm{D}_{h}^{n}\Vert_{0}}_{G}+\underbrace{c_4\widetilde{\alpha}_{M_V}\Vert\bm{U}_{h}^n\Vert_{0}}_{H} +\tau\widetilde{\alpha}c_4
\left\|({\psi}_{h}^{n}, {\bm{D}}_{h}^{n}, {\bm{U}}_{h}^{n})\right\|_{\bm{Y}}\big)^2. \nonumber
\end{aligned} 
\end{equation}
Using the fact that 
\[
\sup_{\substack{F' \in \bm{Y}_h \\ F' \neq \mathbf{0}}} \frac{\bm{A}\left(F, F' \right)}{\left\|F'\right\|_{\bm{Y}}}
\geq \frac{A(F, F)}{\|F\|_{\bm{Y}}} \geq \frac{A(F, F)}{(G + H + \tau\widetilde{\alpha}c_4\|F\|_{\bm{Y}})}
\geq \frac{1}{3c_4^2} (G + H + \tau\widetilde{\alpha}c_4\|F\|_{\bm{Y}}),
\] 
we obtain 
\begin{equation} \label{eqn20}
    \begin{aligned} 
\sup_{\substack{
(\varphi_h,\bm{\chi}_{h},\bm{v}_{h}) \in \bm{Y}_h \\
(\varphi_h,\bm{\chi}_{h},\bm{v}_{h}) \neq \mathbf{0}
}} \frac{{\bm{A}}
\left(({\psi}_{h}^{n}, {\bm{D}}_{h}^{n}, {\bm{U}}_{h}^{n}), (\varphi_h,\bm{\chi}_{h},\bm{v}_{h}) \right)}{\left\|(\varphi_h,\bm{\chi}_{h},\bm{v}_{h})\right\|_{\bm{Y}}} &\geq \widehat{\beta}\widetilde{\alpha}_{M_C}\Vert \bm{D}_{h}^{n}\Vert_{0}+\widehat{\beta}\widetilde{\alpha}_{M_V}\Vert\bm{U}_{h}^n\Vert_{0}\\[-1.0em]
& \;\;\;+\tau\widehat{\alpha}
\left\|({\psi}_{h}^{n}, {\bm{D}}_{h}^{n}, {\bm{U}}_{h}^{n})\right\|_{\bm{Y}},
\end{aligned}
\end{equation}
where, $\widehat{\alpha}=\frac{\widetilde{\alpha}}{3c_4}$  and $\widehat{\beta}=\frac{1}{3c_4}$. Similarly, by applying Lemma \ref{lemma2} along with Young's inequality, we can derive the following estimate \vspace{-0.05cm}
\begin{equation}  \label{eqn21}
\sup_{\substack{
(\varphi_h,\bm{\chi}_{h},\bm{v}_{h}) \in \bm{Y}_h \\
(\varphi_h,\bm{\chi}_{h},\bm{v}_{h}) \neq \mathbf{0}
}}
\frac{{\tau \; \bm{B}}_{\bm{\lambda}_{h}^{n},\bm{w}_{h}^{n}}
\left(({\psi}_{h}^{n}, {\bm{D}}_{h}^{n}, {\bm{U}}_{h}^{n}), (\varphi_h,\bm{\chi}_{h},\bm{v}_{h}) \right)}
{\left\|(\varphi_h,\bm{\chi}_{h},\bm{v}_{h})\right\|_{\bm{Y}}} 
\geq \tau\big[-(\beta_{\lambda}+\beta_{\bm{w}}) \|{(\bm{\lambda}_{h}^{n}, \bm{w}_{h}^{n})}\|_{\bm{X}}\big]
\left\|({\psi}_{h}^{n}, {\bm{D}}_{h}^{n}, {\bm{U}}_{h}^{n})\right\|_{\bm{Y}}, \vspace{-0.1cm}
\end{equation}
where, $\beta_{\lambda} := \widetilde{\beta}_{S_C}+\widetilde{\beta}_{S_V}$ and $ \beta_{\bm{w}} := \widetilde{\beta}_{Q_C}+\widetilde{\beta}_{Q_V}.$ 
Thus, the conclusion follows directly from the definition \eqref{eqn18_new} and the given hypotheses, together with  the Babu\v{s}ka--Brezzi theory. 
\end{proof}
\vspace{-0.3cm}
\begin{lemma} \label{lemma5}
    (Well-definedness of fixed-point operator $\bm{T}$ and stability estimates) Under the assumptions of Lemma \ref{lemma4}, there exists a positive constant $\widehat{c}_{\mathtt{data}}$ such that the following property holds\vspace{-0.1cm}
\begin{align}\label{E_Stab}
\| \bm{T}(\bm{\lambda}_{h}^{n},\bm{w}_{h}^{n})\|_0+\tau\sum_{j=1}^{n}\|{ \bm{T} (\bm{\lambda}_{h}^{j},\bm{w}_{h}^{j})}\|_{\bm{X}}
&\leq \Vert \widehat{\bm{C}}_{h}^{n}\Vert_{0}+\Vert\widehat{\bm{u}}_{h}^{n}\Vert_{0}+\tau\sum_{j=1}^{n}\left\|(\widehat{\phi}_{h}^{j}, \widehat{\bm{C}}_{h}^{j},\widehat{\bm{u}}_{h}^{j})\right\|_{\bm{Y}} \leq \widehat{c}_{\mathtt{data}}\Big(\Vert \bm{C}_{0}\Vert_{0}+\Vert\bm{u}_{0}\Vert_{0}\nonumber\\[-0.8em]
& + \tau\sum_{j=1}^{n}\|(f_{c_1}^n, f_{c_2}^n, f_{\phi}^n, \bm{f}_{u}^n)\|_{L^2(\Omega)\times L^2(\Omega)\times L^2(\Omega)\times \bm{L}^2(\Omega)}\Big) :=\widehat{\delta}_0.
\end{align} 
\end{lemma} \vspace{-0.5cm}
\begin{proof}\!\!.\;
    From equation \eqref{eqn18} and inf-sup condition \eqref{INF_Ah}, we obtain  \vspace{-0.1cm}
    \begin{equation}
    \begin{aligned}
\widehat{\beta}\widetilde{\alpha}_{M_C}\Vert \widehat{\bm{C}}_{h}^{n}\Vert_{0}
+\widehat{\beta}\widetilde{\alpha}_{M_V}\Vert\widehat{\bm{u}}_{h}^{n}\Vert_{0}
+\tau\dfrac{\widehat{\alpha}}{2}
\left\|(\widehat{\phi}_{h}^{n}, \widehat{\bm{C}}_{h}^{n}, \widehat{\bm{u}}_{h}^{n})\right\|_{\bm{Y}}
\leq\; & \widetilde{\beta}_{M_C}\|\widehat{\bm{C}}_{h}^{n-1}\|_0
+ \widetilde{\beta}_{M_V}\|\widehat{\bm{u}}_{h}^{n-1}\|_0 + \tau \Big( \widetilde{\beta}_{\bm{F}_u} \|\bm{f}_u^n\|_0
\notag \\
& +   \sum_{i=1}^2 \widetilde{\beta}_{F_{c_i}} \|f_{c_i}^n\|_0
+  \widetilde{\beta}_{F_\phi} \|f_\phi^n\|_0\Big).\notag
\end{aligned}     \vspace{-0.1cm}
    \end{equation}
Summing over $j=1,\dots,n$, we obtain a discrete inequality in which the intermediate time levels accumulate in a telescopic manner. Consequently, the estimate depends only on the initial data at level $j=0$ and the accumulated forcing contributions over the time interval.
\end{proof}
\vspace{-0.2cm}
\begin{lemma}\label{l_S2}
Consider the following closed ball in $\bm{X}_h$ \vspace{-0.1cm}
 \[
\mathscr{B}_h:=\big\{ (\bm{\lambda}_{h}^{n},\bm{w}_{h}^{n})\in \bm{X}_h:\;\;  \|{(\bm{\lambda}_{h}^{n}, \bm{w}_{h}^{n})}\|_{\bm{X}}\leq \rho_1 \big\}, \qquad  \rho_1 := \frac{\widehat{\alpha}}{2(\widetilde{\beta}_{S_C}+\widetilde{\beta}_{S_V}+ \widetilde{\beta}_{Q_C}+ \widetilde{\beta}_{Q_V})},\vspace{-0.1cm}
\]
and suppose that the initial data in above Lemma satisfy
$
\widehat{\delta}_0 \leq \rho$ with $\rho := \rho_1\,\tau.
$ Then, $\textbf{T}(\mathscr{B}_h)\subset \mathscr{B}_h$.
\end{lemma}\vspace{-0.1cm}
\begin{proof}\!\!.\;
The proof follows directly from Lemma \ref{lemma5}.
\end{proof}
\begin{lemma}[Continuity of  operator $\bm{T}$]\label{l_S2a}
For any $1\leq n\leq N$, there exists a positive constant  ${{c}_{\mathtt{cnt}}=\frac{4}{\widehat{\alpha}}\max\left\{\widetilde{\beta}_{l}+\frac{(\widetilde{\beta}_{S_C} +\widetilde{\beta}_{S_V})\widetilde{c}_{\ast}}{\min\{\widetilde{\alpha}_{A_P}, \widetilde{\alpha}_{J_P}\}},\; \widetilde{\beta}_{Q_C}\rho_1+\widetilde{\beta}_{Q_V}\rho_1\right\}}$ such that
\begin{align}
\|{ \bm{T}(\bm{\lambda}_{h}^{n},\bm{w}_{h}^{n})-\bm{T}(\widetilde{\bm{\lambda}}_{h}^{n},\widetilde{\bm{w}}_{h}^{n})}\|_{\bm{X}}&\leq {c}_{\mathtt{cnt}}\|{( {\bm{\lambda}}_{h}^{n} ,\bm{w}_{h}^{n})- (\widetilde{\bm{\lambda}}_{h}^{n},  \widetilde{\bm{w}}_{h}^{n})}\|_{\bm{X}}.\nonumber
\end{align}
\end{lemma} \vspace{-0.2cm}
\begin{proof}\!\!.\;
We prove that the operator $\bm{T}$ is Lipschitz continuous. For given pairs $(\bm{\lambda}_{h}^{n},\bm{w}_{h}^{n})$ and $(\widetilde{\bm{\lambda}}_{h}^{n},\widetilde{\bm{w}}_{h}^{n})$ in $\mathscr{B}_h$, we assume that the linearized scheme \eqref{eqn17} admits unique solutions $\big(\widehat{\phi}_h^n, \widehat{\bm{C}}_{h}^{n},\widehat{\bm{u}}_{h}^{n}\big)$ and $\big(\widetilde{\phi}_h^n, \widetilde{\bm{C}}_{h}^{n},\allowbreak \widetilde{\bm{u}}_{h}^{n}\big)$, respectively. These solutions satisfy  \vspace{-0.1cm}
\[
\bm{T}(\bm{\lambda}_{h}^{n},\bm{w}_{h}^{n})=\big(\widehat{\bm{C}}_{h}^{n},\widehat{\bm{u}}_{h}^{n}\big)\in \mathscr{B}_h, \quad  
\bm{T}(\widetilde{\bm{\lambda}}_{h}^{n},\widetilde{\bm{w}}_{h}^{n})=\big(\widetilde{\bm{C}}_{h}^{n},\widetilde{\bm{u}}_{h}^{n}\big)\in \mathscr{B}_h. \vspace{-0.1cm}
\] 
It is easy to see that \vspace{-0.05cm}
\begin{equation*}
\widehat{\bm{A}}_{\bm{\lambda}_{h}^{n},\bm{w}_{h}^{n}}\left((\widehat{\phi}_h^n, \widehat{\bm{C}}_{h}^{n},\widehat{\bm{u}}_{h}^{n})-(\widetilde{\phi}_h^n, \widetilde{\bm{C}}_{h}^{n},\widetilde{\bm{u}}_{h}^{n}), (\varphi_h,\bm{\chi}_{h},\bm{v}_{h})\right) =\tau\bm{B}_{{\widetilde{\bm{\lambda}}}_{h}^{n}-\bm{\lambda}_{h}^{n},{\widetilde{\bm{w}}}_{h}^{n}-\bm{w}_{h}^{n}}\left((\widetilde{\phi}_h^n, \widetilde{\bm{C}}_{h}^{n},\widetilde{\bm{u}}_{h}^{n}), (\varphi_h,\bm{\chi}_{h},\bm{v}_{h}) \right). \vspace{-0.05cm}
\end{equation*}
Using the inf-sup condition from Lemma \ref{l_INFSUP} and Lemma \ref{lemma2}, we obtain \vspace{-0.1cm}
\begin{equation}\label{eqn23}
\begin{aligned} 
&\tau\dfrac{\widehat{\alpha}}{2}\left\|(\widehat{\phi}_h^n, \widehat{\bm{C}}_{h}^{n},\widehat{\bm{u}}_{h}^{n})-(\widetilde{\phi}_h^n, \widetilde{\bm{C}}_{h}^{n},\widetilde{\bm{u}}_{h}^{n})\right\|_{ \bm{Y}} \leq \sup _{
(\varphi_h,\bm{\chi}_{h},\bm{v}_{h}) \neq \mathbf{0}} \frac{\tau\bm{B}_{{\widetilde{\bm{\lambda}}}_{h}^{n}-\bm{\lambda}_{h}^{n},{\widetilde{\bm{w}}}_{h}^{n}-\bm{w}_{h}^{n}}\left((\widetilde{\phi}_h^n, \widetilde{\bm{C}}_{h}^{n},\widetilde{\bm{u}}_{h}^{n}), (\varphi_h,\bm{\chi}_{h},\bm{v}_{h}) \right)}{\left\|(\varphi_h,\bm{\chi}_{h},\bm{v}_{h})\right\|_{ \bm{Y}}}\\[-0.3em]
&\hspace{2cm}\leq \tau\bigg\{\widetilde{\beta}_{l}\|{{\widetilde{\bm{\lambda}}}}_{h}^{n}-{\bm{\lambda}}_{h}^{n}\|_{\bm{M}}+\widetilde{\beta}_{S_C}\|{{\widetilde{\bm{\lambda}}}}_{h}^{n}-{\bm{\lambda}}_{h}^{n}\|_{\bm{M}}\|\widetilde{\phi}_h^n\|_2+\widetilde{\beta}_{Q_C}\Vert{{\widetilde{\bm{w}}}}_{h}^{n}-{\bm{w}}_{h}^{n}\Vert_{\bm{V}}\Vert\widetilde{\bm{C}}_{h}^{n}\Vert_{\bm{M}}\\[-0.3em]
&\hspace{2cm}\qquad +\widetilde{\beta}_{Q_V}\Vert{{\widetilde{\bm{w}}}}_{h}^{n}-{\bm{w}}_{h}^{n}\Vert_{\bm{V}}\Vert\widetilde{\bm{u}}_{h}^{n}\Vert_{\bm{V}} +\widetilde{\beta}_{S_V}\|{{\widetilde{\bm{\lambda}}}}_{h}^{n}-{\bm{\lambda}}_{h}^{n}\|_{\bm{M}}\|\widetilde{\phi}_h^n\|_2\bigg\}\\[-0.3em]
&\hspace{2cm}\leq \tau\left(\widetilde{\beta}_{l}+(\widetilde{\beta}_{S_C}+\widetilde{\beta}_{S_V}) \|\widetilde{\phi}_h^n\|_2\right)\|{{\widetilde{\bm{\lambda}}}}_{h}^{n}-{\bm{\lambda}}_{h}^{n}\|_{\bm{M}}+
\tau\left(\widetilde{\beta}_{Q_C}\rho_1+\widetilde{\beta}_{Q_V}\rho_1\right)\Vert{{\widetilde{\bm{w}}}}_{h}^{n}-{\bm{w}}_{h}^{n}\Vert_{\bm{V}}. 
\end{aligned} \vspace{-0.1cm}
\end{equation}
Using $\varphi_h=\widetilde{\phi}_h^n$ in potential equation of \eqref{eqn17} along with Lemmas \ref{lemma1}, we have \vspace{-0.1cm}
\begin{equation} \label{eqn24}
\min\{\widetilde{\alpha}_{A_P}, \widetilde{\alpha}_{J_P}\}\|\widetilde{\phi}_h^n\|_2 \leq \widetilde{\beta}_{l}\|\bm{\lambda}_h^n\|_{\bm{M}} + \widetilde{\beta}_{F_{\phi}} \,\|f_{\phi}\|_{0} \leq \widetilde{\beta}_{l} \rho_1 + \widetilde{\beta}_{F_{\phi}} \,\|f_{\phi}\|_{0} := \widetilde{c}_{\ast}. \vspace{-0.1cm}
\end{equation}
Substituting \eqref{eqn24} into \eqref{eqn23} yields the desired result. 
\end{proof} \vspace{-0.1cm}
 \begin{theorem} \label{Thm1}
(Existence) Assume that the initial data satisfy the condition of Lemma~\ref{l_S2}. Then, the fully discrete scheme \eqref{eqn16} admits a  solution \( \{ \phi_h^n, \bm{C}_h^n, \bm{u}_h^n \} \in \bm{Y}_h \) with \( (\bm{C}_h^n, \bm{u}_h^n) \in \mathscr{B}_h \), and the following estimate holds for \( 1 \leq n \leq N \): \vspace{-0.1cm}
\begin{equation*}
\Vert {\bm{C}}_{h}^{n}\Vert_{0}+\Vert{\bm{u}}_{h}^{n}\Vert_{0}+\tau\sum_{j=1}^{n}\left\|({\phi}_{h}^{j}, {\bm{C}}_{h}^{j},{\bm{u}}_{h}^{j})\right\|_{\bm{Y}}\leq \widehat{c}_{\mathtt{data}}\Big(\Vert \bm{C}_{0}\Vert_{0}+\Vert\bm{u}_{0}\Vert_{0} + \tau\sum_{j=1}^{n}\|(f_{c_1}^n, f_{c_2}^n, f_{\phi}^n, \bm{f}_{u}^n)\|_{[L^2(\Omega)]^3 \times \bm{L}^2(\Omega)}\Big) :=\widehat{\delta}_0. \nonumber \vspace{-0.2cm}
\end{equation*} 
 \end{theorem}
 \begin{lemma} \label{lemma_phi_h^n}
    Under the assumptions of Theorem~\ref{Thm1}, for  $1 \leq n \leq N$, there holds 
     $   \|\phi_h^n\|_2 \leq c \widehat{\delta}_0$.  
\end{lemma}
\begin{proof}\!\!.\;
 The result follows from the potential equation of scheme \eqref{eqn16} with the test function $\varphi_h = \phi_h^n$, combined with Theorem~\ref{Thm1}. (proceeding similar to equation \eqref{eqn24}).
\end{proof}
\begin{theorem}(Uniqueness) \label{Uniqueness}
Suppose that the following condition holds
\[
\widehat{\alpha} - \max\{c_a, c_b, c_c\} > 0,
\]
where, $c_a = 1 + \widetilde{\beta}_{S_C} \widehat{\delta}_0 + \widetilde{\beta}_{S_V} \widehat{\delta}_0$, $c_b = \frac{1 + 3\widetilde{\beta}_{S_C} \widehat{\delta}_0 + \widetilde{\beta}_{Q_C} \widehat{\delta}_0 + \widetilde{\beta}_{S_V} \widehat{\delta}_0}{2}$ and $c_c = \widetilde{\beta}_{Q_C} \widehat{\delta}_0 + 2\widetilde{\beta}_{S_V} \widehat{\delta}_0 + \widetilde{\beta}_{Q_V} \widehat{\delta}_0$.
The constants $c_a$, $c_b$, and $c_c$ depend on the continuity bounds from
Lemma~\ref{lemma2}, $\widehat{\delta}_0$ is defined in
Theorem~\ref{Thm1}, and $\widehat{\alpha}$ is the constant from
Lemma~\ref{lemma4}. Then, for every $1 \le n \le N$, the solution of the fully
discrete scheme~\eqref{eqn16} is unique.
\end{theorem} \vspace{-0.4cm}
\begin{proof}\!\!.\;
Let $\{ \phi_h^n, (C_{1,h}^n, C_{2,h}^n), \bm{u}_h^n \}$ and
$\{ \Phi_h^n, (D_{1,h}^n, D_{2,h}^n), \bm{w}_h^n \}$ be two solutions of the fully
discrete scheme~\eqref{eqn16}. For simplicity, we omit the subscript $h$ and define the difference between the solutions as \vspace{-0.15cm}
\[
e_{\phi}^n := \phi^n - \Phi^n, \quad 
\bm{e}_{C}^n = ({e}_{C_1}^n, {e}_{C_2}^n) := (C_1^n - D_1^n,\; C_2^n - D_2^n), \quad 
\bm{e}_u^n := \bm{u}^n - \bm{w}^n. \vspace{-0.15cm}
\]
Testing equation \eqref{eqn16} with $\{ e_{\phi}^n, \bm{e}_C^n, \bm{e}_u^n \}$ and using  equation \eqref{eqn18}, we obtain \vspace{-0.15cm}
\begin{equation} 
\widehat{\bm{A}}_{\bm{C}^n, \bm{u}^n}\left(( \phi^n, \bm{C}^n, \bm{u}^n), (e_{\phi}^n, \bm{e}_C^n, \bm{e}_u^n) \right)- \widehat{\bm{A}}_{\bm{D}^n, \bm{w}^n}\left( (\Phi^n, \bm{D}^n, \bm{w}^n), (e_{\phi}^n, \bm{e}_C^n, \bm{e}_u^n) \right)=0, \nonumber \vspace{-0.15cm}
\end{equation}
which implies (by definition from \eqref{eqn18}) \vspace{-0.15cm}
\begin{equation}\label{uni1}
\begin{aligned} 
    \bm{A} \left((e_{\phi}^n, \bm{e}_C^n, \bm{e}_u^n), (e_{\phi}^n, \bm{e}_C^n, \bm{e}_u^n) \right) 
    = -\tau \Big[\bm{B}_{\bm{C}^n, \bm{u}^n}
    \left((\phi^n, \bm{C}^n, \bm{u}^n), (e_{\phi}^n, \bm{e}_C^n, \bm{e}_u^n)\right)\\[-0.3em]
    -  \bm{B}_{\bm{D}^n, \bm{w}^n}
    \left((\Phi^n, \bm{D}^n, \bm{w}^n), (e_{\phi}^n, \bm{e}_C^n, \bm{e}_u^n)\right)\Big],
\end{aligned} \vspace{-0.20cm}
\end{equation}
where \vspace{-0.15cm}
\begin{equation*}
   \begin{aligned}
&\bm{B}_{\bm{C}^n, \bm{u}^n}
\left((\phi^n, \bm{C}^n, \bm{u}^n), (e_{\phi}^n, \bm{e}_C^n, \bm{e}_u^n)\right)
- \bm{B}_{\bm{D}^n, \bm{w}^n}
\left((\Phi^n, \bm{D}^n, \bm{w}^n), (e_{\phi}^n, \bm{e}_C^n, \bm{e}_u^n)\right)  \\[-0.3em]
& =\left[l_h(C_1^n; e_{\phi}^n) - l_h(D_1^n; e_{\phi}^n)\right]
- \left[l_h(C_2^n; e_{\phi}^n) - l_h(D_2^n; e_{\phi}^n)\right] + \big[S_{C,h}(\phi^n; C_1^n, e_{C_1}^n) \\[-0.3em]
&  - S_{C,h}(\Phi^n; D_1^n, e_{C_1}^n)\big]
+ \left[S_{C,h}(\phi^n; C_2^n, e_{C_2}^n) - S_{C,h}(\Phi^n; D_2^n, e_{C_2}^n)\right]  - \big[Q_{C,h}^{\mathrm{skew}}(\bm{u}^n; C_1^n, e_{C_1}^n) \\[-0.3em]
& - Q_{C,h}^{\mathrm{skew}}(\bm{w}^n; D_1^n, e_{C_1}^n)\big] 
\; - \left[Q_{C,h}^{\mathrm{skew}}(\bm{u}^n; C_2^n, e_{C_2}^n) - Q_{C,h}^{\mathrm{skew}}(\bm{w}^n; D_2^n, e_{C_2}^n)\right] + \big[Q_{V,h}^{\mathrm{skew}}(\bm{u}^n;  \\[-0.3em]
&  \bm{u}^n, \bm{e}_u^n) - Q_{V,h}^{\mathrm{skew}}(\bm{w}^n; \bm{w}^n, \bm{e}_u^n)\big] 
\; + \left[S_{V,h}(C_1^n - C_2^n; \phi^n, \bm{e}_u^n) - S_{V,h}(D_1^n - D_2^n; \Phi^n, \bm{e}_u^n)\right].
\end{aligned} \vspace{-0.15cm}
\end{equation*}
For each trilinear form, we add and subtract appropriate intermediate terms (similar to equation \eqref{addsub1}) and, using the continuity properties from Lemma~\ref{lemma2}  together with equation \eqref{eqn20}, we obtain  \vspace{-0.15cm}
\begin{equation*}
\begin{aligned}
\widehat{\alpha} \|(e_{\phi}^n, \bm{e}_C^n, \bm{e}_u^n)\|^2_{U \times \bm{M} \times \bm{V}} 
&\le \|e_{\phi}^n\|_U (\|e_{C_1}^n\|_M + \|e_{C_2}^n\|_M)                     
+ \widetilde{\beta}_{S_C} \widehat{\delta}_0 (\|e_{\phi}^n\|_U (\|e_{C_1}^n\|_M + \|e_{C_2}^n\|_M)\\[-0.2em]
&
+ \|e_{C_1}^n\|_M^2 + \|e_{C_2}^n\|_M^2)  
+ \widetilde{\beta}_{Q_C} \widehat{\delta}_0 \|\bm{e}_u^n\|_V (\|e_{C_1}^n\|_M + \|e_{C_2}^n\|_M)                                     
\\[-0.2em]
&+ \widetilde{\beta}_{Q_V} \widehat{\delta}_0 \|\bm{e}_u^n\|_V^2                               
+ \widetilde{\beta}_{S_V} \widehat{\delta}_0 (\|e_{\phi}^n\|_U + \|e_{C_1}^n\|_M + \|e_{C_2}^n\|_M) \|\bm{e}_u^n\|_V.                    
\end{aligned} \vspace{-0.15cm} 
\end{equation*}
 Young’s inequality, along with the given assumptions concludes the proof.
\end{proof} 
\section{Convergence analysis}
In this section, we establish an a priori error estimate for the fully discrete VE scheme \eqref{eqn16}.
We start with establishing three technical lemmas that quantify the  errors between the continuous trilinear forms \( S_V, S_C \), and \( Q_C^{\mathrm{skew}} \) and their discrete counterparts \( S_{V,h}, S_{C,h} \), and \( Q_{C,h}^{\mathrm{skew}} \), respectively. These estimates will play a key role in deriving the error estimates. \vspace{-0.1cm}
\begin{lemma} \label{lemma_S_V}
Let \(\varphi_1 \in \mathring{U} \cap H^{\gamma+2}(\Omega)\) and \(\xi_1 \in M \cap H^{\gamma+1}(\Omega)\) with \(0 \leq \gamma \leq \min\{\mathfrak{J}, \mathcal{K}-1, \mathcal{L}\}\). Then, for any \(\varphi_2 \in \mathring{U}\), \(\xi_2 \in M\), and \(\bm{v} \in \bm{V}\), it holds that   \vspace{-0.1cm}
\begin{align}
    |S_{V}(\xi_1; \varphi_1, \bm{v}) - S_{V,h}(\xi_2; \varphi_2, \bm{v})|  &\leq c h^\gamma \big( \|\xi_1\|_{\gamma+\frac{1}{2}} \|\varphi_1\|_{\gamma+\frac{3}{2}} 
    + \|\xi_1\|_{\gamma+1} \|\varphi_1\|_{1} 
     \notag \\
    & \;\;+ \|\xi_1\|_{1} \|\varphi_1\|_{\gamma+1} \big) \|\bm{v}\|_{1} +  S_{V,h}(\xi_1;\varphi_1- \varphi_2, \bm{v}) + S_{V,h}(\xi_1-\xi_2; \varphi_2, \bm{v}). \nonumber
\end{align}
\end{lemma}
\begin{proof}\!\!.\;
We decompose the error term as \vspace{-0.1cm}
\begin{align}
\mathscr{R}_{S, V} &:=  S_{V}(\xi_1;\varphi_1, \bm{v}) - S_{V,h}(\xi_2;\varphi_2, \bm{v}) \nonumber\\[-0.2em]
&= S_{V}(\xi_1;\varphi_1, \bm{v}) - S_{V,h}(\xi_1;\varphi_1, \bm{v}) + S_{V,h}(\xi_1;\varphi_1, \bm{v}) - S_{V,h}(\xi_2;\varphi_2, \bm{v}) \nonumber\\[-0.2em]
&= S_{V}(\xi_1;\varphi_1, \bm{v}) - S_{V,h}(\xi_1;\varphi_1, \bm{v}) \nonumber\\[-0.2em]
&\hspace{2cm} + S_{V,h}(\xi_1;\varphi_1- \varphi_2, \bm{v}) + S_{V,h}(\xi_1-\xi_2; \varphi_2, \bm{v}) \qquad \left(\text{using } \pm S_{V,h}(\xi_1;\varphi_2, \bm{v})\right) \label{addsub1}\\[-0.2em]
&:= \mathscr{R}_{S, V, 1} + \mathscr{R}_{S, V, 2}. \nonumber 
\end{align}
We split the error term $\mathscr{R}_{S,V,1}$ into local contributions $\mathscr{R}_{S,V,1}^E$ over each element $E$. Using the definitions of the continuous and discrete forms, we have \vspace{-0.2cm}
\begin{equation*}
 \begin{aligned}
\mathscr{R}_{S,V,1}^E &=: \int_E \xi_1 \nabla \varphi_1 \cdot \bm{v} \,dx - \int_E \left[({\Pi}^{0,E}_{\mathfrak{J}} \xi_1) (\bm{\Pi}^{0,E}_{\mathcal{K}-1}\nabla \varphi_1) \right] \cdot \left[\bm{\Pi}^{0,E}_{\mathcal{L}} \bm{v}\right] \,dx \\[-0.2em]
&= \int_E \xi_1 \nabla\varphi_1 \cdot (\bm{I} - \bm{\Pi}^{0,E}_{\mathcal{L}}) \bm{v} \,dx + \int_E \left[\xi_1 - {\Pi}^{0,E}_{\mathfrak{J}} \xi_1 \right] \nabla \varphi_1 \cdot \bm{\Pi}^{0,E}_{\mathcal{L}} \bm{v} \,dx \\[-0.2em]
&\quad + \int_E {\Pi}^{0,E}_{\mathfrak{J}} \xi_1 (\nabla\varphi_1 - \bm{\Pi}^{0,E}_{\mathcal{K}-1}\nabla \varphi_1) \cdot \bm{\Pi}^{0,E}_{\mathcal{L}} \bm{v} \,dx := \sum_{i=1}^3\mathscr{R}_{S,V,1, i}^E\;.
\end{aligned} \vspace{-0.25cm}
\end{equation*}
For $\mathscr{R}_{S,V,1,1}^E$, using Sobolev embeddings 
\(H^{\frac{1}{2}+\gamma}(E) \hookrightarrow L^4(E)\) and 
\(H^{\frac{3}{2}+\gamma}(E) \hookrightarrow W^{1,4}(E)\), we have \vspace{-0.15cm}
\begin{equation*}
    \begin{aligned}
\text{For}\;\;0 \leq \gamma \leq 1:\quad 
\mathscr{R}_{S,V,1,1}^E
&\;\leq \; \|\xi_1\|_{0,4,E} \|\nabla\varphi_1\|_{0,4,E}\; h_E |\bm{v}|_{1,E} \;\leq \; c\, h_E \|\xi_1\|_{\frac{1}{2}+\gamma,E} \|\varphi_1\|_{\frac{3}{2}+\gamma,E} \|\bm{v}\|_{1,E}.\\[-0.2em]
\text{For}\;\;1 < \gamma \leq \mathcal{L}:\quad
\mathscr{R}_{S,V,1,1}^E
&= \int_E (\bm{I} - \bm{\Pi}^{0,E}_{\mathcal{L}}) (\xi_1 \nabla\varphi_1) 
     \cdot (\bm{I} - \bm{\Pi}^{0,E}_{\mathcal{L}}) \bm{v} \,dx \\[-0.2em]
&\leq \|(\bm{I} - \bm{\Pi}^{0,E}_{\mathcal{L}}) (\xi_1 \nabla\varphi_1)\|_{0,E} 
     \|(\bm{I} - \bm{\Pi}^{0,E}_{\mathcal{L}}) \bm{v} \|_{0,E} \\[-0.2em]
&\leq c\, h_E^{\gamma} \|\xi_1\|_{\gamma,E} \|\varphi_1\|_{1+\gamma,E} \|\bm{v}\|_{1,E}.
\end{aligned} \vspace{-0.15cm}
\end{equation*}
For the second and third terms, we again use Sobolev embeddings and continuity of projections with respect to the $L^2$ and $L^4$ norms to get  \vspace{-0.15cm}
\begin{equation*}
    \begin{aligned}
\mathscr{R}_{S,V,1,2}^E
&\leq \;\|\xi_1 - {\Pi}^{0,E}_{\mathfrak{J}} \xi_1 \|_{0,4,E}
     \|\nabla\varphi_1\|_{0,E}
     \|\bm{\Pi}^{0,E}_{\mathcal{L}} \bm{v} \|_{0,4,E} \; \leq \; c\, h_E^\gamma \|\xi_1\|_{\gamma+1,E} \|\varphi_1\|_{1,E} \|\bm{v}\|_{1,E}.\\[-0.2em]
\mathscr{R}_{S,V,1,3}^E
&\leq \; \|{\Pi}^{0,E}_{\mathfrak{J}} \xi_1\|_{0,4,E}
     \|\nabla\varphi_1 - \bm{\Pi}^{0,E}_{\mathcal{K}-1}\nabla \varphi_1 \|_{0,E}
     \|{\Pi}^{0,E}_{\mathcal{L}} \bm{v} \|_{0,4,E} \\[-0.2em]
&\leq c\, h_E^\gamma \|\xi_1\|_{1,E} \|\varphi_1\|_{\gamma+1,E} \|\bm{v}\|_{1,E}.
\end{aligned} \vspace{-0.15cm}
\end{equation*}
Summing the contribution from all three terms of $\mathscr{R}_{S,V,1}$ over all elements yields  \vspace{-0.1cm}
\[
\mathscr{R}_{S,V,1} \leq ch^\gamma(\|\xi_1\|_{\gamma+\frac{1}{2}} \|\varphi_1\|_{\gamma+\frac{3}{2}}+ \|\xi_1\|_{\gamma+1} \|\varphi_1\|_{1}+ \|\xi_1\|_{1} \|\varphi_1\|_{\gamma+1})|\bm{v}|_{1}. \vspace{-0.1cm}
\]
Combining estimates for $\mathscr{R}_{S,V,1}$ and $\mathscr{R}_{S,V,2}$ using the triangle inequality concludes the proof. 
\end{proof} 
\begin{lemma} \label{lemma_S_C} 
Let $\phi_1 \in \mathring{U} \cap H^{\gamma+2}(\Omega)$, $\xi_1 \in M \cap H^{\gamma+1}(\Omega)$ with $0 \leq \gamma \leq \min\{\mathfrak{J}, \mathcal{K}-1\}$. Then, for any $\phi_2 \in \mathring{U}$ and  $\xi_2, \xi \in M$, it holds that \vspace{-0.15cm}
\begin{align}
    |S_C(\phi_1; \xi_1, \xi) - S_{C,h}(\phi_2; \xi_2, \xi)|  &\leq c h^\gamma \big( \|\xi_1\|_{\gamma+1} \|\phi_1\|_{\gamma+2} 
    + \|\xi_1\|_{\gamma+1} \|\phi_1\|_{2}
     \notag \\
    &\;\; + \|\xi_1\|_{1} \|\phi_1\|_{\gamma+2} \big)  \|\xi\|_{1} +  S_{C,h}(\phi_1-\phi_2; \xi_1, \xi) + S_{C,h}(\phi_2; \xi_1-\xi_2, \xi). \nonumber
\end{align}
\end{lemma}
\vspace{-0.35cm}
\begin{proof}\!\!.\;
We proceed as in the previous lemma 
\begin{align*}
{\mathscr{R}}_{S, C} &:= S_C(\phi_1; \xi_1, \xi) - S_{C,h}(\phi_2; \xi_2, \xi) \\
&= S_C(\phi_1; \xi_1, \xi) - S_{C,h}(\phi_1; \xi_1, \xi)   + S_{C,h}(\phi_1-\phi_2; \xi_1, \xi) + S_{C,h}(\phi_2; \xi_1-\xi_2, \xi) \\
&:= \mathscr{R}_{S, C, 1} + \mathscr{R}_{S, C, 2}. 
\end{align*}
We split the error term $\mathscr{R}_{S,C,1}$ into local contributions $\mathscr{R}_{S,C,1}^E$ over each element $E$. Using the definitions of the continuous and discrete forms, we have \vspace{-0.15cm}
\begin{equation*}
\begin{aligned}
\mathscr{R}_{S,C,1}^E &=:  \int_E \xi_1 \nabla \phi_1 \cdot (\bm{I} - \bm{\Pi}^{0,E}_{\mathfrak{J}-1}) \nabla \xi \,dx + \int_E \left[\xi_1 - {\Pi}^{0,E}_{\mathfrak{J}} \xi_1 \right] \nabla \phi_1 \cdot \bm{\Pi}^{0,E}_{\mathfrak{J}-1} \nabla \xi \,dx \\[-0.3em]
&\quad + \int_E {\Pi}^{0,E}_{\mathfrak{J}} \xi_1 (\nabla \phi_1 - \bm{\Pi}^{0,E}_{\mathcal{K}-1}\nabla \phi_1) \cdot \bm{\Pi}^{0,E}_{\mathfrak{J}-1} \nabla \xi \,dx :=  \sum_{i=1}^3\mathscr{R}_{S,C,1, i}^E\;.
\end{aligned} \vspace{-0.25cm}
\end{equation*}
For first term $\mathscr{R}_{S,C,1,1}^E$, we use the definition of $\bm{\Pi}^{0,E}_{\mathfrak{J}-1}$, Cauchy--Schwarz inequality, and the approximation properties to get \vspace{-0.15cm}
\begin{equation*}
 \begin{aligned}
\mathscr{R}_{S,C,1,1}^E &=: \int_E (\bm{I} - \bm{\Pi}^{0,E}_{\mathfrak{J}-1}) \left(\xi_1 \nabla \phi_1 \right) \cdot (\bm{I}- \bm{\Pi}^{0,E}_{\mathfrak{J}-1}) \nabla \xi \,dx \\[-0.3em]
&\leq \left\|(\bm{I} - \bm{\Pi}^{0,E}_{\mathfrak{J}-1}) (\xi_1 \nabla \phi_1) \right\|_{0,E} \left\|\nabla \xi\right\|_{0,E} \\[-0.3em]
&\leq c h_E^{\gamma} \|\xi_1 \nabla \phi_1\|_{\gamma,E} \|\xi\|_{1,E} \;\;\leq \;\;c h_E^{\gamma } \|\xi_1\|_{\gamma+1,E} \|\phi_1\|_{\gamma+2,E} \|\xi\|_{1,E}.
\end{aligned}   \vspace{-0.15cm}
\end{equation*}
Other two terms can be bounded using Hölder's inequality, continuity of projections with respect to $L^2$ and $L^4$ norms, Sobolev embeddings and the approximation properties such that \vspace{-0.15cm}
\begin{equation*}
  \begin{aligned}
\mathscr{R}_{S,C,1,2}^E
&\leq \left\|(I- {\Pi}^{0,E}_{\mathfrak{J}}) \xi_1\right\|_{0,4,E} \left\| \nabla \phi_1\right\|_{0,4,E} \left\|\bm{\Pi}^{0,E}_{\mathfrak{J}-1} \nabla \xi\right\|_{0,E} \leq c h_E^\gamma \|\xi_1\|_{\gamma+1,E} \|\phi_1\|_{2,E} \|\xi\|_{1,E}.\\[-0.3em]
\mathscr{R}_{S,C,1,3}^E
&\leq \|\Pi^{0,E}_\mathfrak{J} \xi_1\|_{0,4,E} \|(\bm{I}-\bm{\Pi}^{0,E}_{\mathcal{K}-1} )\nabla \phi_1 \|_{0,4,E} \|\bm{\Pi}^{0,E}_{\mathfrak{J}-1} \nabla \xi\|_{0,E}\; \leq c h_E^\gamma \|\xi_1\|_{1,E} \| \phi_1\|_{\gamma+2,E} \|\xi\|_{1,E}.
\end{aligned}   \vspace{-0.05cm}
\end{equation*}
Combining all element-wise estimates and the triangle inequality yields the result.
\end{proof} 
\begin{lemma}\label{lemma_Q_C}
Let $\bm{u}_1 \in \bm{V} \cap \bm{H}^{\gamma+1}(\Omega)$ and $\xi_1 \in M \cap H^{\gamma+1}(\Omega)$, with $\frac{1}{2} \leq \gamma \leq \min\{\mathfrak{J},\mathcal{L}\}$. For any $\xi_2 \in M$, define the decomposition \vspace{-0.1cm}
\[
\xi_1 - \xi_2 = (\xi_1 - \xi_{1,I}) + (\xi_{1,I} - \xi_2) := \omega_I + \omega_h, \vspace{-0.1cm}
\]
where $\xi_{1,I}$ is the interpolant of $\xi_1$ as defined in Proposition~\ref{interpolation_concentration}. Then for any $\bm{u}_2 \in \bm{V}$, the following estimate holds: \vspace{-0.1cm}
\begin{equation}
  \begin{aligned}
   \big| Q_C^{\mathrm{skew}}(\bm{u}_1; \xi_1, \omega_h) &- Q_{C,h}^{\mathrm{skew}}(\bm{u}_2; \xi_2, \omega_h) \big| \leq ch^\gamma (\|\bm{u}_1\|_{\gamma+1}\|\xi_1\|_{\gamma+1} + \|\bm{u}_1\|_{\gamma}\|\xi_1\|_{\gamma+1}   + \|\bm{u}_1\|_{\gamma+1}\|\xi_1\|_{1}  \notag \\[-0.2em]
    &\qquad \qquad+ \|\bm{u}_1\|_{1}\|\xi_1\|_{\gamma+1} )\|\omega_h\|_{1} + Q_{C,h}^{\mathrm{skew}}(\bm{u}_2; \omega_I, \omega_h)  +  Q_{C,h}^{\mathrm{skew}}(\bm{u}_1 - \bm{u}_2; \xi_1, \omega_h).\notag \vspace{-0.25cm}
\end{aligned}   
\end{equation}
\end{lemma}\vspace{-0.3cm}
\begin{proof}\!\!.\;
The results follows by a similar argument to those used for the previous two results.
\end{proof}
We are now in a position to derive an a priori error bound, which is proven under the following assumptions on the regularity of the exact solutions.\vspace{-0.15cm}
\begin{assumption} \label{assumption-fully_discrete}
 The exact solution tuple \(\{ \phi, (C_1, C_2), \bm{u}\}\) and forces $\{f_{\phi}, f_{c_1}, f_{c_2}, \bm{f}_{u}\} $ satisfies \vspace{-0.2cm}
    \[
    \begin{aligned}
        C_i &\in    L^{\infty}(0,T; H^{1}(\Omega)) \cap L^{4}(0,T; H^{j+1}(\Omega)), && \partial_tC_i \in L^{1}(0,T; H^{j}(\Omega)), &&& \partial_{tt}C_i \in L^{1}(0,T; L^2(\Omega)), \\[-0.3em]
        &\bm{u} \in L^{\infty}(0,T; \bm{H}^{1}(\Omega)) \cap L^{4}(0,T; \bm{H}^{r+1}(\Omega)), 
     &&\partial_t\bm{u} \in L^{1}(0,T; \bm{H}^{r}(\Omega)) ,  &&&\partial_{tt}\bm{u} \in L^{1}(0,T; \bm{L}^{2}(\Omega)), \\[-0.3em]
    \phi &\in L^{\infty}(0,T; H^{2}(\Omega)) \cap L^{4}(0,T; H^{k+2}(\Omega)), &&  &&& \\[-0.3em]
    f_{c_i} &\in L^{2}(0,T; {H}^{j}(\Omega)), && f_{\phi} \in L^{2}(0,T; {H}^{k}(\Omega)), &&& \bm{f}_{u} \in L^{2}(0,T; \bm{H}^{r}(\Omega)),
    \end{aligned} \vspace{-0.2cm}
    \]
    for \(i=1,2, \;\;\frac{1}{2}\leq  j \leq \mathfrak{J}, \;\;0 \leq  k \leq \mathcal{K}-1\)\; and \;\( \frac{1}{2}\leq  r \leq \mathcal{L}\).
\end{assumption} 
\begin{theorem} \label{thm_fully_error}
Let \(\{ \phi^n, (C_1^n, C_2^n), \bm{u}^n\}\in \mathring{U} \times \bm{M} \times \widetilde{\bm{V}}\)  be the solution of \eqref{eqn6*} under Assumption \ref{assumption-fully_discrete}. Also, assume that \(\{ \phi_h^n, (C_{1,h}^n, C_{2,h}^n), \bm{u}_h^n\} \in \mathring{U}_h^\mathcal{K} \times \bm{M}_h^\mathfrak{J} \times \widetilde{\bm{V}}_h^\mathcal{L}\) is the solution of fully discrete scheme  \eqref{eqn16},
then the following error estimate holds \vspace{-0.15cm}
\begin{equation}
\begin{aligned}
    \|\phi^n-\phi_h^n\|_2^2+ \|\bm{u}^n-\bm{u}_h^n\|_0^2 &+\|C_{1}^n-C_{1,h}^n\|_0^2+ \|C_{2}^n-C_{2,h}^n\|_0^2 \nonumber\\[-0.4em]
    &+ \tau\sum_{i=1}^n\left(\|\bm{u}^i-\bm{u}_h^i\|_1^2 + \|C_{1}^i-C_{1,h}^i\|_1^2+ \|C_{2}^i-C_{2,h}^i\|_1^2\right) \leq c(h^{2\min\{j,k, r\}}+\tau^2),\nonumber 
\end{aligned} \vspace{-0.2cm}
\end{equation}
for $1\leq n\leq N$, where \(\tfrac12 \le j \le \mathfrak{J}, 0 \le k \le \mathcal{K}-1\),
\(\tfrac12 \le r \le \mathcal{L}\), and \(c>0\) depends on the mesh regularity
and the solution regularity, but is independent of \(h\) and \(\tau\).
\end{theorem} \vspace{-0.35cm}
\begin{proof}\!\!. \;
     For all $ \bm{v}_h\in \widetilde{\bm{V}}_h^\mathcal{L}, C_h \in {M}_h^\mathfrak{J}$, we define the following norms \vspace{-0.20cm}
     \begin{equation}  \label{equivalent_norm}
    \begin{aligned} 
        &\triplenorm{\bm{v}_h^n}_{v} := \left(M_{V,h}( \bm{v}_h^n, \bm{v}_h^n)\right)^{\frac{1}{2}},  &\triplenorm{{C}_h^n}_{C} := \left(M_{C,h}( C_h^n, C_h^n)\right)^{\frac{1}{2}}. \vspace{-0.2cm}
    \end{aligned} 
    \end{equation}
    These norms are equivalent on the respective discrete spaces. We proceed by splitting the proof into the following steps:
\vspace{0.1cm}

    \noindent \textbf{Step 1. Interpolation error estimates:}     
We split the errors as follows \vspace{-0.2cm}
\begin{alignat}{3}
    \bm{u}_h^n - \bm{u}^n      &\;:=\; \bm{u}_h^n - \bm{u}_I^n    &&\;+\; \bm{u}_I^n - \bm{u}^n       &&\;=:\; \bm{\theta}_h^n + \bm{\theta}_I^n, \\[-0.2em]
    C_{i,h}^n - C_i^n &\;:=\; C_{i,h}^n - C_{i,I}^n &&\;+\; C_{i,I}^n - C_i^n &&\;=:\; \xi_{i,h}^n + \xi_{i,I}^n, \quad i=1,2, \\[-0.2em]
    \phi_h^n - \phi^n &\;:=\; \phi_h^n - \phi_I^n &&\;+\; \phi_I^n - \phi^n &&\;=:\; \upnu_h^n + \upnu_I^n, 
\end{alignat}
where $ C_{i,I}^n, \phi_I^n$ and $\bm{u}_I^n$ are the discrete approximations of $ C_{i}^n, \phi^n$ and $\bm{u}^n$ defined in propositions $\ref{interpolation_concentration}, \ref{interpolation_potential}$ and $\ref{interpolation_velocity_kernel}$ respectively. Therefore, we have 
\begin{align}
\label{known_interpolant_bound_fully}
    \|\xi_{i,I}^n\|_{0} &\leq c h^{j+1} \|C_i^n\|_{j+1},  \quad 
    \|\xi_{i,I}^n\|_{1} \leq c h^j \|C_i^n\|_{j+1}, \quad i=1,2, \nonumber\\
    \|\upnu_I^n\|_{2} &\leq c h^k \|\phi^n\|_{k+2},  \\
    \|\bm{\theta}_I^n\|_{0} &\leq c h^{r+1} \|\bm{u}^n\|_{r+1}, \quad 
    \|\bm{\theta}_I^n\|_{1} \leq c h^r \|\bm{u}^n\|_{r+1}. \nonumber
\end{align}
\noindent \textbf{Step 2. Error equation for the velocity:} 
Using the fully discrete scheme \eqref{eqn16c} for velocity field and continuous weak form  \eqref{eqn6*c} while taking into account that $\bm{\theta}_h^n \in \widetilde{\bm{V}}_h^\mathcal{L}$, we obtained \vspace{-0.2cm}
\begin{align} \label{velocity_fully}
  M_{V,h}(\delta_t \bm{\theta}_h^n, \bm{\theta}_h^n) + A_{V,h}( \bm{\theta}_h^n, \bm{\theta}_h^n) 
 & = \left(M_{V}(\partial_t \bm{u}^n, \bm{\theta}_h^n) - M_{V,h}(\delta_t \bm{u}_I^n, \bm{\theta}_h^n)\right) 
   + \big(A_{V}( \bm{u}^n, \bm{\theta}_h^n) - A_{V,h}( \bm{u}_I^n, \bm{\theta}_h^n) \big) \nonumber \\
  &\quad  + \left(Q_{V}^{\mathrm{skew}}(\bm{u}^n;\bm{u}^n, \bm{\theta}_h^n) - Q_{V,h}^{\mathrm{skew}}(\bm{u}_h^n;\bm{u}_h^n, \bm{\theta}_h^n) \right) + \big(S_{V}(C_{1}^n-C_{2}^n;\phi^n, \bm{\theta}_h^n)\nonumber \\
  &\quad  - S_{V,h}(C_{1,h}^n-C_{2,h}^n;\phi_h^n, \bm{\theta}_h^n) \big) + \left( \bm{F}_{u,h}(\bm{\theta}_h^n)- \bm{F}_{u}(\bm{\theta}_h^n)\right)\nonumber \\
  &\phantom{=}:= \Upsilon_{M, V}^n + \Upsilon_{A, V}^n + \Upsilon_{Q, V}^n + \Upsilon_{S, V}^n + \Upsilon_{\bm{F}, V}^n.
\end{align} 
Now, we bound the terms in right hand side one by one. 
For the term $\Upsilon_{M, V}^n$, using  polynomial consistency property, we get \vspace{-0.1cm}
\begin{align}  \label{estimate_vel_1_fully}
\Upsilon_{M, V}^n 
&=: M_{V}(\partial_t \bm{u}^n, \bm{\theta}_h^n) - M_{V,h}(\delta_t \bm{u}_I^n, \bm{\theta}_h^n) \nonumber\\[-0.1em]
& =  M_{V}\big(\partial_t \bm{u}^n- \delta_t \bm{u}^n, \bm{\theta}_h^n\big) + \sum_{E\in \Omega_h} M_{V}^E\big( \delta_t \bm{u}^n-\delta_t\bm{u}_{\pi}^n, \bm{\theta}_h^n\big)+ M_{V,h}^E\big( \delta_t\bm{u}_{\pi}^n-\delta_t\bm{u}_I^n, \bm{\theta}_h^n\big) \nonumber\\[-0.1em]
&:= \Upsilon_{M, V, 1}^n+ \Upsilon_{M, V, 2}^n+ \Upsilon_{M, V, 3}^n. 
\end{align}
To estimate the term $\Upsilon_{M,V,1}^n$, we use Lemma~\ref{lemma1} and apply Taylor’s theorem  to obtain \vspace{-0.2cm}
\begin{equation*}
    \begin{aligned}
    \Upsilon_{M, V, 1}^n &\leq \frac{1}{\tau} \left\|\tau \bm{u}^n_t- ({\bm{u}^n-\bm{u}^{n-1}})\right\|_0 \|\bm{\theta}_h^n\|_0 = \frac{1}{\tau} \left\| \int_{t^{n-1}}^{t^n}(s-t^{n-1})u_{tt}(\cdot, s)ds\right\|_0 \|\bm{\theta}_h^n\|_0  \nonumber\\[-0.3em]
    & \leq   \left\|\partial_{tt}u\right\|_{L^1(t^{n-1},\, t^n; L^2(\Omega))}\|\bm{\theta}_h^n\|_0. \nonumber
\end{aligned} \vspace{-0.15cm}
\end{equation*}
To estimate the terms $\Upsilon_{M,V,2}^n$ and $\Upsilon_{M,V,3}^n$, using Lemma~\ref{lemma1} and prepositions  \ref{brenner2008mathematica} and \ref{interpolation_velocity_kernel}, we have \vspace{-0.15cm}
\begin{equation*}
\begin{aligned}
    \sum_{i=2,3}\Upsilon_{M, V, i}^n &\leq \sum_{E\in \Omega_h} \left(\|\delta_t \bm{u}^n-\delta_t\bm{u}_{\pi}^n\|_{0,E} + \|\delta_t \bm{u}^n_{\pi}-\delta_t\bm{u}_{I}^n\|_{0,E} \right)\| \bm{\theta}_h^n\|_{0,E}  \nonumber\\[-0.5em]
    & \leq \frac{1}{\tau}\sum_{E\in \Omega_h} \left(2\|(\bm{u}^n-\bm{u}^{n-1})-(\bm{u}_{\pi}^n-\bm{u}_{\pi}^{n-1})\|_{0,E}  + \|(\bm{u}^n-\bm{u}^{n-1})-(\bm{u}_{I}^n-\bm{u}_{I}^{n-1})\|_{0,E} \right)\| \bm{\theta}_h^n\|_{0,E}\nonumber\\[-0.5em]
    & \leq \frac{1}{\tau}\left(ch^r \|\partial_tu\|_{L^1(t^{n-1},\, t^n; H^r(\Omega))} \| \bm{\theta}_h^n\|_{0} \right). \nonumber 
\end{aligned} \vspace{-0.2cm}
\end{equation*}
For $\Upsilon_{A, V}^n$, using  polynomial consistency, continuity of $A_{V,h}$ and prepositions  \ref{brenner2008mathematica} and \ref{interpolation_velocity_kernel}, we get \vspace{-0.2cm}
\begin{align}  \label{estimate_vel_2_fully}
\Upsilon_{A, V}^n    &=: A_{V}( \bm{u}^n, \bm{\theta}_h^n) - A_{V,h}( \bm{u}_I^n, \bm{\theta}_h^n)  = \sum_{E\in \Omega_h}  A_{V}^E\big( \bm{u}^n- \bm{u}_{\pi}^n, \bm{\theta}_h^n\big) + A_{V,h}^E\big( \bm{u}_{\pi}^n-\bm{u}_I^n, \bm{\theta}_h^n\big) \nonumber\\[-0.4em]
&\leq  \sum_{E\in \Omega_h} \Big(2 \| \bm{u}^n- \bm{u}_{\pi}^n\|_{1,E}+ \| \bm{u}^n- \bm{u}_I^n\|_{1,E}\Big) \|\bm{\theta}_h^n\|_{1,E} \nonumber\\[-0.4em]
& \leq ch^r \|\bm{u}^n\|_{r+1} \|\bm{\theta}_h^n\|_{1} \; \leq \;ch^{2r} \|\bm{u}^n\|_{r+1}^2 + \frac{\widetilde{\alpha}_{A_V}}{24}{\|\bm{\theta}_h^n\|_{1}^2}. 
\end{align}
The bound for the skew-symmetric nonlinear term $\Upsilon_{Q, V}$ has already been established in Lemma $4.3$ of~\cite{da2018virtual}. In particular, using Lemma \ref{lemma2}, preposition  \ref{interpolation_velocity_kernel} with assumptions \ref{assumption-fully_discrete} yields \vspace{-0.2cm}
\begin{align} \label{estimate_vel_3_fully}
\Upsilon_{Q, V}^n &= \left(Q_{V}^{\mathrm{skew}}(\bm{u}^n;\bm{u}^n, \bm{\theta}_h^n)-Q_{V,h}^{\mathrm{skew}}(\bm{u}^n;\bm{u}^n, \bm{\theta}_h^n) \right)+ \left(Q_{V,h}^{\mathrm{skew}}(\bm{u}^n;\bm{u}^n, \bm{\theta}_h^n) - Q_{V,h}^{\mathrm{skew}}(\bm{u}_h^n;\bm{u}_h^n, \bm{\theta}_h^n) \right) \nonumber\\[-0.1em]
& \leq ch^r(\|\bm{u}^n\|_r+ \|\bm{u}^n\|_1+ \|\bm{u}^n\|_{r+1})\|\bm{u}^n\|_{r+1} \|\bm{\theta}_h^n\|_{1} - Q_{V,h}^{\mathrm{skew}}(\bm{u}^n;\bm{\theta}_I^n, \bm{\theta}_h^n) - Q_{V,h}^{\mathrm{skew}}(\bm{u}^n;\bm{\theta}_h^n, \bm{\theta}_h^n) \nonumber \\[-0.1em]
&\qquad \qquad\ -Q_{V,h}^{\mathrm{skew}}(\bm{\theta}_I^n;\bm{u}_h^n, \bm{\theta}_h^n) -Q_{V,h}^{\mathrm{skew}}(\bm{\theta}_h^n;\bm{u}_h^n, \bm{\theta}_h^n)\nonumber\\[-0.1em]
& \leq ch^{2r} \|\bm{u}^n\|_{r+1}^4 + \frac{2\widetilde{\alpha}_{A_V}}{24} \|\bm{\theta}_h^n\|_{1}^2 + c \|\bm{u}^n\|_1^2\|\bm{\theta}_I^n\|_1^2 - Q_{V,h}^{\mathrm{skew}}(\bm{\theta}_I^n;\bm{u}_h^n, \bm{\theta}_h^n)- Q_{V,h}^{\mathrm{skew}}(\bm{\theta}_h^n;\bm{u}_h^n, \bm{\theta}_h^n)\nonumber\\[-0.3em]
&\leq ch^{2r} (\|\bm{u}^n\|_{r+1}^4+ \|\bm{u}^n\|_{r+1}^2) + \frac{\widetilde{\alpha}_{A_V}}{12} \|\bm{\theta}_h^n\|_{1}^2  - Q_{V,h}^{\mathrm{skew}}(\bm{\theta}_I^n;\bm{u}_h^n, \bm{\theta}_h^n)- Q_{V,h}^{\mathrm{skew}}(\bm{\theta}_h^n;\bm{u}_h^n, \bm{\theta}_h^n)\nonumber\\[-0.3em]
&:= ch^{2r} (\|\bm{u}^n\|_{r+1}^4+ \|\bm{u}^n\|_{r+1}^2) + \frac{\widetilde{\alpha}_{A_V}}{12} \|\bm{\theta}_h^n\|_{1}^2   -\left( \Upsilon_{Q, V,1}^n+ \Upsilon_{Q, V,2}^n\right).
\end{align}
Using the same hypothesis as above, we have   \vspace{-0.1cm} 
\begin{equation*}
\begin{aligned}
    -\Upsilon_{Q, V,1}^n &=: Q_{V,h}^{\mathrm{skew}}(\bm{\theta}_I^n;\bm{u}^n-\bm{u}_h^n-\bm{u}^n, \bm{\theta}_h^n) = Q_{V,h}^{\mathrm{skew}}(\bm{\theta}_I^n; -\bm{\theta}_I^n, \bm{\theta}_h^n)+ Q_{V,h}^{\mathrm{skew}}(\bm{\theta}_I^n;-\bm{u}^n, \bm{\theta}_h^n)\nonumber\\[-0.4em]
    & \leq \widetilde{\beta}_{Q_V}  \|\bm{\theta}_I^n\|_1\left(\|\bm{\theta}_I^n\|_1+\|\bm{u}^n\|_1\right) \|\bm{\theta}_h^n\|_1  \leq ch^{2r} \|\bm{u}^n\|_{r+1}^2\left(\|\bm{\theta}_I^n\|_1^2+\|\bm{u}^n\|_1^2\right) + \frac{\widetilde{\alpha}_{A_V}}{24} \|\bm{\theta}_h^n\|_{1}^2 \nonumber\\[-0.4em]
    & \leq ch^{2r} \|\bm{u}^n\|_{r+1}^2 + \frac{\widetilde{\alpha}_{A_V}}{24} \|\bm{\theta}_h^n\|_{1}^2. \vspace{-0.1cm}
\end{aligned}
\end{equation*}
For $\Upsilon_{Q, V,2}^n$, Sobolev inequality \eqref{sobolev_ine} and two successive applications of Young’s inequality yields \vspace{-0.1cm}
\begin{equation*}
   \begin{aligned}
    -\Upsilon_{Q, V,2}^n &=: Q_{V,h}^{\mathrm{skew}}(\bm{\theta}_h^n;\bm{u}^n-\bm{u}_h^n-\bm{u}^n, \bm{\theta}_h^n) = Q_{V,h}^{\mathrm{skew}}(\bm{\theta}_h^n; -\bm{\theta}_I^n, \bm{\theta}_h^n)+ Q_{V,h}^{\mathrm{skew}}(\bm{\theta}_h^n;-\bm{u}^n, \bm{\theta}_h^n)\nonumber\\[-0.3em]
    & \leq \widetilde{\beta}_{Q_V} \|\bm{\theta}_h^n\|_1\left(\|\bm{\theta}_I^n\|_1+\|\bm{u}^n\|_1\right)  \|\bm{\theta}_h^n\|_1^{\frac{1}{2}} \|\bm{\theta}_h^n\|_0^{\frac{1}{2}}\nonumber\\[-0.3em]
    & \leq \frac{\widetilde{\alpha}_{A_V}}{24} \|\bm{\theta}_h^n\|_{1}^2 + c \left(\|\bm{\theta}_I^n\|_1^2+\|\bm{u}^n\|_1^2\right)  \|\bm{\theta}_h^n\|_1 \|\bm{\theta}_h^n\|_0 \leq \frac{\widetilde{\alpha}_{A_V}}{12} \|\bm{\theta}_h^n\|_{1}^2 + c \|\bm{u}^n\|_1^4   \|\bm{\theta}_h^n\|_0^2. 
\end{aligned} \vspace{-0.1cm}
\end{equation*}
To derive an upper bound for $\Upsilon_{S,V}^n$, we apply Lemma~\ref{lemma_S_V} with 
$\gamma=\min\{j,k,r\}$ and the stability estimate 
$\|\phi_h^n\|_{L^{\infty}(0,T; H^2(\Omega))}$ from Lemma~\ref{lemma_phi_h^n}, yielding \vspace{-0.1cm}
\begin{align} \label{estimate_vel_4_fully}
\Upsilon_{S,V}^n 
& \leq c h^{\min\{j,k, r\}}  
     \left(\|C_{1}^n-C_{2}^n\|_{j+1}  
    \|\phi^n\|_{k+2} \right) \|\bm{\theta}_h^n\|_{1} \notag \\[-0.2em]
    &\qquad\qquad +  S_{V,h}\left(C_{1}^n-C_{2}^n;\phi^n-\phi_h^n, \bm{\theta}_h^n\right) + S_{V,h}\left((C_{1}^n-C_{2}^n)-(C_{1,h}^n-C_{2,h}^n); \phi_h^n, \bm{\theta}_h^n\right)  \notag\\[-0.2em]
    & \leq c h^{2\min\{j,k, r\}}(\|\phi^n\|_{k+2}^4+\|C_1^n\|_{j+1}^4+\|C_2^n\|_{j+1}^4) + \frac{\widetilde{\alpha}_{A_V}}{24} \|\bm{\theta}_h^n\|_{1}^2 \nonumber\\[-0.2em]
& \qquad +  S_{V,h}\left(C_{1}^n-C_{2}^n;-\upnu_h^n-\upnu_I^n, \bm{\theta}_h^n\right)  + S_{V,h}\left((\xi_{2,h}^n- \xi_{1,h}^n)+(\xi_{2,I}^n- \xi_{1,I}^n); \phi_h^n, \bm{\theta}_h^n\right) \nonumber\\[-0.2em]
& \leq c h^{2\min\{j,k, r\}}\left(\|\phi^n\|_{k+2}^4+\|C_1^n\|_{j+1}^4 +\|C_2^n\|_{j+1}^4\right) + \frac{\widetilde{\alpha}_{A_V}}{24} \|\bm{\theta}_h^n\|_{1}^2 + \frac{\widetilde{\alpha}_{A_V}}{24} \|\bm{\theta}_h^n\|_{1}^2+ c(\|C_1^n\|^2_1+\|C_2^n\|^2_1)\|\upnu_I^n\|_2^2 \nonumber\\[-0.2em]
&\quad + \frac{\widetilde{\alpha}_{A_V}}{24} \|\bm{\theta}_h^n\|_{1}^2  +   c(\|\xi_{1,I}^n\|^2_1+\|\xi_{2,I}^n\|^2_1)\|\phi_h^n\|_2^2+S_{V,h}\left(C_{1}^n-C_{2}^n;-\upnu_h^n, \bm{\theta}_h^n\right) + S_{V,h}\left(\xi_{2,h}^n-\xi_{1,h}^n; \phi_h^n, \bm{\theta}_h^n\right) \nonumber\\[-0.2em]
 &\leq  c h^{2\min\{j,k, r\}}(\|\phi^n\|_{k+2}^4+\|C_1^n\|_{j+1}^4+\|C_2^n\|_{j+1}^4+ \|\phi^n\|_{k+2}^2+\|C_1^n\|_{j+1}^2+\|C_2^n\|_{j+1}^2) 
\nonumber\\[-0.2em]
& \qquad+ \frac{\widetilde{\alpha}_{A_V}}{8} \|\bm{\theta}_h^n\|_{1}^2  
  +\underbrace{{S_{V,h}\left(C_{1}^n-C_{2}^n;-\upnu_h^n, \bm{\theta}_h^n\right)}}_{\Upsilon_{S,V, 1}^n} + \underbrace{{S_{V,h}\left(\xi_{2,h}^n-\xi_{1,h}^n; \phi_h^n, \bm{\theta}_h^n\right)}}_{\Upsilon_{S,V, 2}^n}. 
\end{align}
Denoting  $\widetilde{\alpha}^{*}_{A_P}= \min\{\widetilde{\alpha}_{A_P}, \widetilde{\alpha}_{J_P}\}$ and  employing Sobolev inequality \eqref{sobolev_ine} with two successive applications of Young's inequality yields 
\begin{align*}
 \Upsilon_{S, V,1}^n &+ \Upsilon_{S, V,2}^n
    \leq \widetilde{\beta}_{S_V}\left( \|C_{1}^n-C_{2}^n\|_1 \|\upnu_h^n\|_2   + \|\xi_{2,h}^n-\xi_{1,h}^n\|_1 \|\phi_h^n\|_2  \right)\|\bm{\theta}_h^n\|_0^{\frac{1}{2}} \|\bm{\theta}_h^n\|_1^{\frac{1}{2}}\nonumber\\[-0.2em]
    & \leq \frac{\widetilde{\alpha}^{*}_{A_P}}{2} \|\upnu_h^n\|_2^2 + c\left(\|C_{1}^n-C_{2}^n\|_1^2 + \|\phi_h^n\|_2^2\right) \|\bm{\theta}_h^n\|_0\|\bm{\theta}_h^n\|_1+ \frac{\widetilde{\alpha}_{A_C}}{26} (\|\xi_{1,h}^n\|_1^2+\|\xi_{2,h}^n\|_1^2) \nonumber\\[-0.2em]
    & \leq \frac{\widetilde{\alpha}^{*}_{A_P}}{2} \|\upnu_h^n\|_2^2 + c\left(\|C_{1}^n-C_{2}^n\|_1^4+ \|\phi_h^n\|_2^4\right)  \|\bm{\theta}_h^n\|_0^2 + \frac{\widetilde{\alpha}_{A_V}}{12} \|\bm{\theta}_h^n\|_{1}^2  + \frac{\widetilde{\alpha}_{A_C}}{26} (\|\xi_{1,h}^n\|_1^2+\|\xi_{2,h}^n\|_1^2). 
\end{align*}
For the last term, using the definition of $\bm{F}_{u,h}(\cdot)$ and $\bm{F}_{u}(\cdot)$ together with Cauchy-Schwarz and Young's inequalities, we get  \vspace{-0.05cm}
\begin{equation}  \label{estimate_force_1_fully}
\Upsilon_{\bm{F}, V}^n
 \leq  \frac{c}{2\epsilon}h^{2r} \|\bm{f}_{u}^n\|_{r}^2 + \frac{\epsilon}{2}{\|\bm{\theta}_h^n\|_{0}^2}. \vspace{-0.05cm}
\end{equation}
Now, from equation~\eqref{velocity_fully} and the estimates~\eqref{estimate_vel_1_fully}--\eqref{estimate_force_1_fully},  together with equivalent norm  \eqref{equivalent_norm}  implies \vspace{-0.1cm}
\begin{align} \label{estimate_vel_fnl_fully}
\frac{1}{2\tau}\left(\triplenorm{\bm{\theta}_h^n}_v^2-\triplenorm{\bm{\theta}_h^{n-1}}_v^2\right) &+ \frac{13\widetilde{\alpha}_{A_V}}{24}\|\bm{\theta}_h^n\|_1^2  \leq \frac{\widetilde{\alpha}^*_{A_P}}{2}\|\upnu_h^n\|_2^2 + \frac{\widetilde{\alpha}_{A_C}}{26}\Big(\|\xi_{1,h}^n\|_1^2 + \|\xi_{2,h}^n\|_1^2\Big)  \nonumber \\[-0.2em]
&+ \frac{1}{\tau}\left(ch^r \|\partial_tu\|_{L^1(t^{n-1},\, t^n; H^r(\Omega))} \triplenorm{\bm{\theta}_h^n}_{v}\right) + \left\|\partial_{tt}u\right\|_{L^1(t^{n-1},\, t^n; L^2(\Omega))}\triplenorm{\bm{\theta}_h^n}_{v} \nonumber\\[-0.2em]
&\quad + ch^{2\min\{j,k,r\}}\Big(\|\bm{u}^n\|_{r+1}^2 + \|\phi^n\|_{k+2}^2 + \|C_1^n\|_{j+1}^2 + \|C_2^n\|_{j+1}^2 \nonumber\\[-0.2em]
&\quad + \|\bm{u}^n\|_{r+1}^4 + \|\phi^n\|_{k+2}^4 + \|C_1^n\|_{j+1}^4 + \|C_2^n\|_{j+1}^4\Big) + ch^{2r} \|\bm{f}_{u}^n\|_{r}^2\nonumber \\[-0.2em]
&\quad   + c\left( 1+ \|\bm{u}^n\|_1^4 + \|C_1^n-C_2^n\|_1^4 + \|\phi_h^n\|_2^4\right)\triplenorm{\bm{\theta}_h^n}_{v}^2.
\end{align}

\noindent \textbf{Step 3. Error equation for potential:} 
Using the semi-discrete scheme \eqref{eqn16a} and continuous weak form  \eqref{eqn6*a} for potential while taking into account that $\upnu_h^n \in \mathring{U}{}_{h}^{k}$, we obtained \vspace{-0.1cm}
\begin{align} \label{pot-fully}
  J_{P,h}( \upnu_h^n, \upnu_h^n) + A_{P,h}( \upnu_h^n, \upnu_h^n) 
  &= \left(J_{P}( \phi^n, \upnu_h^n) - J_{P,h}( {\phi}_I^n, \upnu_h^n)\right) + \big(A_{P}( \phi^n, \upnu_h^n) - A_{P,h}( {\phi}_I^n, \upnu_h^n)\big)  \nonumber\\[-0.2em]
  &\qquad+ \left( F_{\phi,h}(\upnu_h^n)- F_{\phi}(\upnu_h^n)\right) + \left(l_h(C_{1,h}^n - C_{2,h}^n; \upnu_h^n) \right) - \left(l(C_{1}^n - C_{2}^n; \upnu_h^n) \right) \nonumber\\[-0.2em]
  &\phantom{=}:= \Upsilon_{J, P}^n + \Upsilon_{A, P}^n + \Upsilon_{F, P}^n + \Upsilon_{l, P}^n. 
\end{align}
The estimate of first three terms follows the same argument as for the bound of $\Upsilon_{A, V}^n$ and $\Upsilon_{\bm{F}, V}^n$ \vspace{-0.2cm}
    \begin{equation}
    \Upsilon_{J, P}^n  + \Upsilon_{A, P}^n + \Upsilon_{F, P}^n  \leq ch^k \left(\|\phi^n\|_{k+2}  + \|f_{\phi}^n\|_{k} \right) \|{\upnu}_h^n\|_{2} \leq ch^{2k} \left(\|\phi^n\|_{k+2}^2 + \|f_{\phi}^n\|_{k}^2 \right)+ \frac{\widetilde{\alpha}^{*}_{A_P}}{4} \|\upnu_h^n\|_2^2,   \nonumber \vspace{-0.1cm}
\end{equation}
where $\widetilde{\alpha}^{*}_{A_P}= \min\{\widetilde{\alpha}_{A_P}, \widetilde{\alpha}_{J_P}\}$. 
For third term, we split the last term $\Upsilon_{l, P}^n$ into local contributions $\Upsilon_{l, P}^{n,E}$. Using definitions \eqref{eqn7} and \eqref{eqn10} while applying Cauchy-Schwarz inequality and continuity of projections $\bm{\Pi}^{0,E}_{\mathfrak{J}}$ and $\bm{\Pi}^{0,E}_{\mathcal{K}}$ with respect to $L^2$ norm  \vspace{-0.20cm}
\begin{align*}
\Upsilon_{l, P}^{n,E}  
&=  l_h\big((C_{1,h}^n - C_{2,h}^n)-(C_{1}^n - C_{2}^n); \upnu_h^n\big)+    l_h^E(C_{1}^n - C_{2}^n; \upnu_h^n)   - l^E(C_{1}^n - C_{2}^n; \upnu_h^n)     \nonumber\\[-0.25em]
& =  l_h\big((C_{1,h}^n - C_{2,h}^n)-(C_{1}^n - C_{2}^n); \upnu_h^n\big)  + c  \int_E \big[ \Pi^{0,E}_{\mathfrak{J}} (C_{1}^n - C_{2}^n) - (C_1^n - C_2^n) \big] \, \Pi^{0,E}_{\mathcal{K}} \upnu_h^n \nonumber\\[-0.25em]
&\quad -c  \int_E(C_1^n - C_2^n) \,  (I - \Pi^{0,E}_{\mathcal{K}}) \upnu_h^n \, dx    \nonumber\\[-0.25em]
& \leq  c(\|C_{1}^n- C_{1,h}^n\|_{0, E}+ \|C_{2}^n-C_{2,h}^n\|_{0, E})\|\upnu_h^n\|_{0, E}  + c  \int_E \big[ \Pi^{0,E}_{\mathfrak{J}} (C_{1}^n - C_{2}^n) - (C_1^n - C_2^n) \big] \, \Pi^{0,E}_{\mathcal{K}} \upnu_h^n \nonumber\\[-0.25em]
&\quad -c  \int_E  (I - \Pi^{0,E}_{\mathcal{K}}) (C_1^n - C_2^n) \,  (I - \Pi^{0,E}_{\mathcal{K}}) \upnu_h^n \, dx    \nonumber\\[-0.25em]
&\leq c\Big[ \big(\|\xi_{1,h}^n\|_{0,E}+ \|\xi_{1,I}^n\|_{0,E} + \|\xi_{2,h}^n\|_{0,E}+ \|\xi_{2,I}^n\|_{0,E}\big) + h^{\min\{j,k\}} \|C_1^n\|_{j, E}+h^{\min\{j,k\}}\|C_2^n\|_{j, E}\Big]\|\upnu_h^n\|_{2, E} \nonumber\\[-0.25em]
&\leq c(\|\xi_{1,h}^n\|_{0, E}^2+\|\xi_{2,h}^n\|_{0, E}^2) + ch^{2\min\{j, k\}} (\|C_1^n\|_{j, E}^2+ \|C_2^n\|_{j, E}^2) +  \frac{\widetilde{\alpha}^{*}_{A_P}}{4} \|\upnu_h^n\|_{2, E}^2.
\end{align*}
Therefore, summing the local contributions of the above over all elements and collecting the contribution from first two terms as well, we obtain the final bound for \eqref{pot-fully} \vspace{-0.15cm}
\begin{equation} \label{estimate_poten_fnl_fully}
    \frac{\widetilde{\alpha}^{*}_{A_P}}{2}\|\upnu_h^n\|_2^2 \leq ch^{2\min\{j, k\}}(\|C_1^n\|_{j}^2 +
    \|C_2^n\|_{j}^2+ \|\phi^n\|_{k+2}^2) +  ch^{2k}  \|f_{\phi}^n\|_{k}^2+ c\|\xi_{1,h}^n\|_0^2 + c\|\xi_{2,h}^n\|_0^2. \vspace{-0.15cm}
\end{equation}
\noindent \textbf{Step 4. Error equation for the concentrations:}  Using the semi-discrete scheme \eqref{eqn16b} for concentrations  $C_i, i=1,2$ and continuous weak form  \eqref{eqn6*b}, we obtained \vspace{-0.15cm}
\begin{align} \label{concentration_fully}
M_{C,h}(\partial_t \xi_{i,h}^n, \xi_{i,h}^n) + A_{C,h}(\xi_{i,h}^n, \xi_{i,h}^n) &= \left(M_C(\partial_t C_i^n, \xi_{i,h}^n) - M_{C,h}(\partial_t C_{i,I}^n, \xi_{i,h}^n)\right)  + \big(A_C(C_i^n, \xi_{i,h}^n) \nonumber \\[-0.2em]
& \quad - A_{C,h}(C_{i,I}^n, \xi_{i,h}^n)\big) + z_i \big(S_C(\phi^n; C_i^n, \xi_{i,h}^n) - S_{C,h}(\phi_h^n; C_{i,h}^n, \xi_{i,h}^n)\big) \nonumber \\[-0.2em]
&\quad  - \left(  Q_C^{\mathrm{skew}}(\bm{u}^n; C_i^n, \xi_{i,h}^n) - Q_{C,h}^{\mathrm{skew}}(\bm{u}_h^n; C_{i,h}^n, \xi_{i,h}^n)\right) + \left( F_{c_i,h}(\xi_{i,h}^n)- F_{c_i}(\xi_{i,h}^n)\right)\nonumber \\[-0.2em]
&:= \Upsilon_{M,C}^n + \Upsilon_{A,C}^n + z_i \Upsilon_{S,C}^n + \Upsilon_{Q,C}^n + \Upsilon_{F,C}^n.
\end{align}
We now proceed to bound the terms on the right-hand side individually. The estimates for the last term and first two terms follow the same steps as in Step~2, yielding \vspace{-0.2cm}
\begin{subequations}
    \begin{align}
    &\Upsilon_{M,C}^n \leq  \frac{1}{\tau} \Big(c h^{j}\|\partial_t C_{i}\|_{L^1(t^{n-1},t^n;H^j(\Omega))} {\|\xi_{i,h}^n\|_{0}}\Big)+ \|\partial_{tt}C_{i}\|_{L^1(t^{n-1},t^n;L^2(\Omega))}   {\|\xi_{i,h}^n\|_{0}}, \label{estimate_conc_1_fully}\\[-0.2em]
    & \Upsilon_{A, C}^n + \Upsilon_{F,C}^n  \leq ch^{2j}\| C_i^n\|_{j+1}^2 + \frac{\widetilde{\alpha}_{A_C}}{26} \|\xi_{i,h}^n\|_{1}^2 + \frac{c}{2\epsilon}h^{2j}\| f_{c_i}^n\|_{j}^2 + \frac{\epsilon}{2} \|\xi_{i,h}^n\|_{0}^2. \label{estimate_conc_11_fully} 
\end{align}
\end{subequations}
To estimate the skew-symmetric nonlinear term $\Upsilon_{Q,C}^n$, we apply Lemma~\ref{lemma_Q_C} with $\gamma=\min\{ j, r\}$, which yields the following bound
\begin{align} \label{estimate_conc_2_fully}\vspace{-0.15cm}
\Upsilon_{Q, C}^{n} 
& \leq c h^{\min\{ j, r\}} \left( \|\bm{u}^{n}\|_{r+1} + \|\bm{u}^{n}\|_{r} + \|\bm{u}^{n}\|_{r+1} + \|\bm{u}^{n}\|_{1} \right) \|C_i^{n}\|_{j+1} \|\xi_{i,h}^{n}\|_{1} + Q_{C,h}^{\mathrm{skew}}(\bm{u}_h^{n}; \xi_{i,I}^{n}, \xi_{i,h}^{n}) \nonumber \\[-0.2em]
& \quad\quad  + Q_{C,h}^{\mathrm{skew}}(\bm{u}^{n} - \bm{u}_h^{n}; C_i^{n}, \xi_{i,h}^{n}) \nonumber\\[-0.2em]
& \leq c h^{2\min\{ j, r\}} \left(\|C_i^{n}\|_{j+1}^4 + \|\bm{u}^{n}\|_{r+1}^4\right)+ \frac{\widetilde{\alpha}_{A_C}}{26} \|\xi_{i,h}^{n}\|_{1}^2 + Q_{C,h}^{\mathrm{skew}}(\bm{u}_h^{n}; \xi_{i,I}^{n}, \xi_{i,h}^{n}) + Q_{C,h}^{\mathrm{skew}}(\bm{u}^{n} - \bm{u}_h^{n}; C_i^{n}, \xi_{i,h}^{n}) \nonumber\\[-0.2em]
& := c h^{2\min\{ j, r\}} \left(\|C_i^{n}\|_{j+1}^4 + \|\bm{u}^{n}\|_{r+1}^4\right)+ \frac{\widetilde{\alpha}_{A_C}}{26} \|\xi_{i,h}^{n}\|_{1}^2 + \Upsilon_{Q, C, 1}^{n} + \Upsilon_{Q, C, 2}^{n}.
\end{align}
We now estimate the terms $\Upsilon_{Q, C,1}^n$ and $\Upsilon_{Q, C,2}^n$ using Sobolev inequality \eqref{sobolev_ine}, followed by two successive applications of Young’s inequality, to obtain \vspace{-0.15cm}
\begin{align*}
    \Upsilon_{Q, C,1}^{n}+ \Upsilon_{Q, C,2}^{n}
    &=  Q_{C,h}^{\mathrm{skew}}(\bm{u}_h^{n}-\bm{u}^{n}+\bm{u}^{n}; \xi_{i,I}^{n}, \xi_{i,h}^{n})
    -  Q_{C,h}^{\mathrm{skew}}(\bm{\theta}_h^{n} + \bm{\theta}_I^{n}; C_i^{n}, \xi_{i,h}^{n}) \nonumber\\[-0.2em]
    &=Q_{C,h}^{\mathrm{skew}}(\bm{u}^{n}; \xi_{i,I}^{n}, \xi_{i,h}^{n})+ Q_{C,h}^{\mathrm{skew}}(\bm{\theta}_h^{n}+ \bm{\theta}_I^{n}; \xi_{i,I}^{n}, \xi_{i,h}^{n}) -  Q_{C,h}^{\mathrm{skew}}(\bm{\theta}_h^{n} + \bm{\theta}_I^{n}; C_i^{n}, \xi_{i,h}^{n})  \nonumber\\[-0.2em]
    & \leq c\|\bm{u}^{n}\|_1^2\|\xi_{i,I}^{n}\|_1^2 +  \frac{\widetilde{\alpha}_{A_C}}{26} \|\xi_{i,h}^{n}\|_{1}^2 + c\|\bm{\theta}_I^{n}\|_1^2\|\xi_{i,I}^{n}\|_1^2 +  \frac{\widetilde{\alpha}_{A_C}}{26} \|\xi_{i,h}^{n}\|_{1}^2 + c\|\bm{\theta}_I^{n}\|_1^2\|C_i^{n}\|_1^2 \nonumber\\[-0.2em]
    & \qquad \qquad +  \frac{\widetilde{\alpha}_{A_C}}{26} \|\xi_{i,h}^{n}\|_{1}^2 + Q_{C,h}^{\mathrm{skew}}(\bm{\theta}_h^{n}; \xi_{i,I}^{n}, \xi_{i,h}^{n}) -  Q_{C,h}^{\mathrm{skew}}(\bm{\theta}_h^{n}; C_i^{n}, \xi_{i,h}^{n}) \nonumber\\[-0.2em]
    & \leq c h^{2\min\{ j, r\}} (\|C_i^{n}\|_{j+1}^2+\|\bm{u}^{n}\|_{r+1}^2)+ \frac{3\widetilde{\alpha}_{A_C}}{26} \|\xi_{i,h}^{n}\|_{1}^2   \nonumber\\[-0.2em]
    & \qquad\qquad + \widetilde{\beta}_{Q_C} (\|\xi_{i,I}^{n}\|_1+\|C_{i}^{n}\|_1)\|\bm{\theta}_h^{n}\|_1  \|\xi_{i,h}^{n}\|_0^{\frac{1}{2}} \|\xi_{i,h}^{n}\|_1^{\frac{1}{2}} \nonumber\\[-0.2em]
    & \leq c h^{2\min\{ j, r\}}(\|C_i^{n}\|_{j+1}^2+\|\bm{u}^{n}\|_{r+1}^2) + \frac{3\widetilde{\alpha}_{A_C}}{26} \|\xi_{i,h}^{n}\|_{1}^2   + \frac{\widetilde{\alpha}_{A_V}}{24} \|\bm{\theta}_h^{n}\|_1^2\nonumber\\[-0.2em]
    & \qquad\qquad  + c (\|\xi_{i,I}^{n}\|_1^2 +\|C_{i}^{n}\|_1^2)  \|\xi_{i,h}^{n}\|_0 \|\xi_{i,h}^{n}\|_1\nonumber\\[-0.2em]
    & \leq c h^{2\min\{ j, r\}}(\|C_i^{n}\|_{j+1}^2+\|\bm{u}^{n}\|_{r+1}^2) + \frac{4\widetilde{\alpha}_{A_C}}{26} \|\xi_{i,h}^{n}\|_{1}^2 + \frac{\widetilde{\alpha}_{A_V}}{24} \|\bm{\theta}_h^{n}\|_1^2 + c \|C_{i}^{n}\|_1^4  \|\xi_{i,h}^{n}\|_0^2.  
\end{align*}
To derive an upper bound for \(\Upsilon_{S,C}^n\), we proceed similar to \(\Upsilon_{S,V}^n\). In particular, we apply Lemma~\ref{lemma_S_C} with $\gamma=\min\{j,k\}$ to get \vspace{-0.25cm}
\begin{align} \label{estimate_conc_3_fully}
\Upsilon_{S,C}^n 
& \leq c h^{2\min\{j,k\}}(\|C_{i}^n\|_{j+1}^4+\|\phi^n\|_{k+2}^4) + \frac{\widetilde{\alpha}_{A_C}}{26} \|\xi_{i,h}^n\|_1^2 +  S_{C,h}\left(-\upnu_h^n-\upnu_I^n; C_i^n, \xi_{i,h}^n\right) \nonumber\\[-0.2em]
&\quad - S_{C,h}(\phi_h^n; \xi_{i,I}^n+\xi_{i,h}^n, \xi_{i,h}^n) \nonumber\\[-0.2em]
& \leq c h^{2\min\{j,k\}}(\|C_{i}^n\|_{j+1}^4+\|\phi^n\|_{k+2}^4)  + c\|C_i^n\|^2_1\|\upnu_I^n\|_2^2 + c\|\phi_h^n\|^2_2\|\xi_{i,I}^n\|_1^2\nonumber\\[-0.2em]
& \quad + \frac{\widetilde{3\alpha}_{A_C}\|\xi_{i,h}^n\|_1^2}{26}  +   \widetilde{\beta}_{S_C} \|C_{i}^n\|_1\|\upnu_h^n\|_2  \|\xi_{i,h}^n\|_0^{\frac{1}{2}} \|\xi_{i,h}^n\|_1^{\frac{1}{2}} +  \widetilde{\beta}_{S_C} \|\phi_h^n\|_2\|\xi_{i,h}^n\|_1  \|\xi_{i,h}^n\|_0^{\frac{1}{2}} \|\xi_{i,h}^n\|_1^{\frac{1}{2}}\nonumber\\[-0.2cm]
&\leq  c h^{2\min\{j,k\}} (\|C_{i}^n\|_{j+1}^4+ \|\phi^n\|_{k+2}^4+ \|C_{i}^n\|_{j+1}^2+ \|\phi^n\|_{k+2}^2)+ \frac{6\widetilde{\alpha}_{A_C}}{26} \|\xi_{i,h}^n\|_1^2  \nonumber\\[-0.2cm]
&\quad +  c \|C_{i}^n\|_1^4  \|\xi_{i,h}^n\|_0^2 +\frac{\widetilde{\alpha}^{*}_{A_P}}{2}\|\upnu_h^n\|_2^2
+  c \|\phi_h^n\|_2^4  \|\xi_{i,h}^n\|_0^2. 
\end{align}
Now, from equation~\eqref{concentration_fully} and the estimates~\eqref{estimate_conc_1_fully}-\eqref{estimate_conc_3_fully}, together with equivalent norm  \eqref{equivalent_norm}, we obtain  \vspace{-0.2cm}
\begin{align} \label{estimate_conc_fnl_fully}
\frac{1}{2\tau}\bigg(\triplenorm{\xi_{i,h}^n}_C^2-&\triplenorm{\xi_{i,h}^{n-1}}_C^2\bigg) +\widetilde{\alpha}_{A_C}\|\xi_{i,h}^n\|_1^2 
\leq  c\left(  1+  \|C_i^n\|_1^4 + \|\phi_h^n\|_2^4  \right) \triplenorm{\xi_{i,h}^n}_C^2 + ch^{2j}\| f_{c_i}^n\|_{j}^2\nonumber\\[-0.2cm]
&+c h^{2\min\{j,k, r\}} \Big(  \|C_i^n\|_{j+1}^2 + \|\phi^n\|_{k+2}^2 + \|\bm{u}^n\|_{r+1}^2 + \|C_i^n\|_{j+1}^4  \nonumber\\[-0.2cm]
&+ \|\phi^n\|_{k+2}^4 + \|\bm{u}^n\|_{r+1}^4\Big) + \frac{12\widetilde{\alpha}_{A_C}}{26} \|\xi_{i,h}^n\|_1^2  + \frac{\widetilde{\alpha}^*_{A_P}}{2} \|\upnu_h^n\|_2^2 
 + \frac{\widetilde{\alpha}_{A_V}}{24} \|\bm{\theta}_h^n\|_1^2  \nonumber\\[-0.2cm]
& + \frac{1}{\tau} \Big(c h^{j}\|\partial_t C_{i}\|_{L^1(t^{n-1},t^n;H^j(\Omega))} {\triplenorm{\xi_{i,h}^n}_{C}}\Big)+ \|\partial_{tt}C_{i}\|_{L^1(t^{n-1},t^n;L^2(\Omega))}   {\triplenorm{\xi_{i,h}^n}_{C}}. 
\end{align}
At this stage, we begin by substituting the estimate \eqref{estimate_poten_fnl_fully} into the estimates \eqref{estimate_vel_fnl_fully} and \eqref{estimate_conc_fnl_fully}. We then combine the resulting expressions and iterate over $i=1,2 \ldots n$ to obtain the following inequality \vspace{-0.15cm}
\begin{align*}
\triplenorm{\bm{\theta}_h^n}_v^2 &+\triplenorm{\xi_{1,h}^n}_C^2+ \triplenorm{\xi_{2,h}^n}_C^2 + \tau\sum_{i=0}^n\left(\|\bm{\theta}_h^i\|_1^2 + \|\xi_{1,h}^i\|_1^2+ \|\xi_{2,h}^i\|_1^2\right)  \\[-0.2cm]
&\leq  \tau\sum_{i=0}^n\mu_i\left(\triplenorm{\bm{\theta}_h^i}_v^2 + \triplenorm{\xi_{1,h}^i}_C^2+ \triplenorm{\xi_{2,h}^i}_C^2\right) + c\left(\|\bm{\theta}_h^0\|_0^2 + \|\xi_{1,h}^0\|_0^2 + + \|\xi_{2,h}^0\|_0^2 \right)\\[-0.1cm]
&\quad + c h^{2\min\{j, r\}}\left(\|\partial_t\bm{u}\|_{L^1(0,t^n;H^r(\Omega))}^2 + \|\partial_t C_1\|_{L^1(0,t^n;H^j(\Omega))}^2+ \|\partial_t C_2\|_{L^1(0,t^n;H^j(\Omega))}^2\right) \\[-0.1cm]
&\quad + c\tau^2\left(\|\bm{u}_{tt}\|_{L^1(0,t^n;L^2(\Omega))}^2 + \|\partial_{tt}C_1\|_{L^1(0,t^n;L^2(\Omega))}^2+ \|\partial_{tt}C_2\|_{L^1(0,t^n;L^2(\Omega))}^2\right) \\[-0.1cm]
&\quad + c h^{2\min\{j,k, r\}}\tau\sum_{i=0}^n \left(\|\bm{u}^i\|_{r+1}^2 + \|\phi^i\|_{k+2}^2 + \|C_1^i\|_{j+1}^2+ \|C_2^i\|_{j+1}^2\right)\\[-0.1cm]
&\quad + c h^{2\min\{j,k, r\}}\tau\sum_{i=0}^n \left(\|\bm{u}^i\|_{r+1}^4 + \|\phi^i\|_{k+2}^4 + \|C_1^i\|_{j+1}^4+ \|C_2^i\|_{j+1}^4\right)\\[-0.2cm]
&\quad + c h^{2\min\{j,k, r\}}\tau\sum_{i=0}^n \left(\|\bm{f}_{u}^i\|_{r+1}^2 + \|f_{\phi}^i\|_{k+2}^2 + \|f_{c_1}^i\|_{j+1}^2+ \|f_{c_2}^i\|_{j+1}^2\right).
\end{align*}
where \(\mu_i = c\left( 1+ \|\bm{u}^i\|_1^4 + \|C_{1}^i\|_1^4 + \|C_{2}^i\|_1^4+ \|\phi_h^i\|_2^4\right)\). Assuming \( \tau \mu_i < 1 \) and defining \( \rho_i := (1 - \tau \mu_i)^{-1} \), we apply discrete Gronwall inequality \cite[Lemma 5.1]{heywood1990finite} with the initial condition \( \left(\bm{u}_h^0, C_{1,h}^0, C_{2,h}^0\right) = \left(\bm{u}_I(0), C_{1,I}(0), C_{2,I}(0)\right) \) to get \vspace{-0.15cm}
\begin{equation}
    \|\bm{\theta}_h^n\|_0^2 +\|\xi_{1,h}^n\|_0^2+ \|\xi_{2,h}^n\|_0^2 + \tau\sum_{i=1}^n\left(\|\bm{\theta}_h^i\|_1^2 + \|\xi_{1,h}^i\|_1^2+ \|\xi_{2,h}^i\|_1^2\right) \leq c( h^{2\min\{j,k, r\}}+\tau^2). \nonumber \vspace{-0.15cm}
\end{equation}

Note that an estimate for \( \|\upnu_h^n\|_2 \) can be obtained from the equation \eqref{estimate_poten_fnl_fully} by employing the above estimate for the concentration in the \( L^2 \)-norm. The desired result then follows by using the approximation properties given in \eqref{known_interpolant_bound_fully}, and the application of the triangle inequality.
\end{proof} 
\vspace{-0.5cm}
\begin{figure}[H]
\centering
\setlength{\tabcolsep}{1pt}
\includegraphics[width=0.23\textwidth]{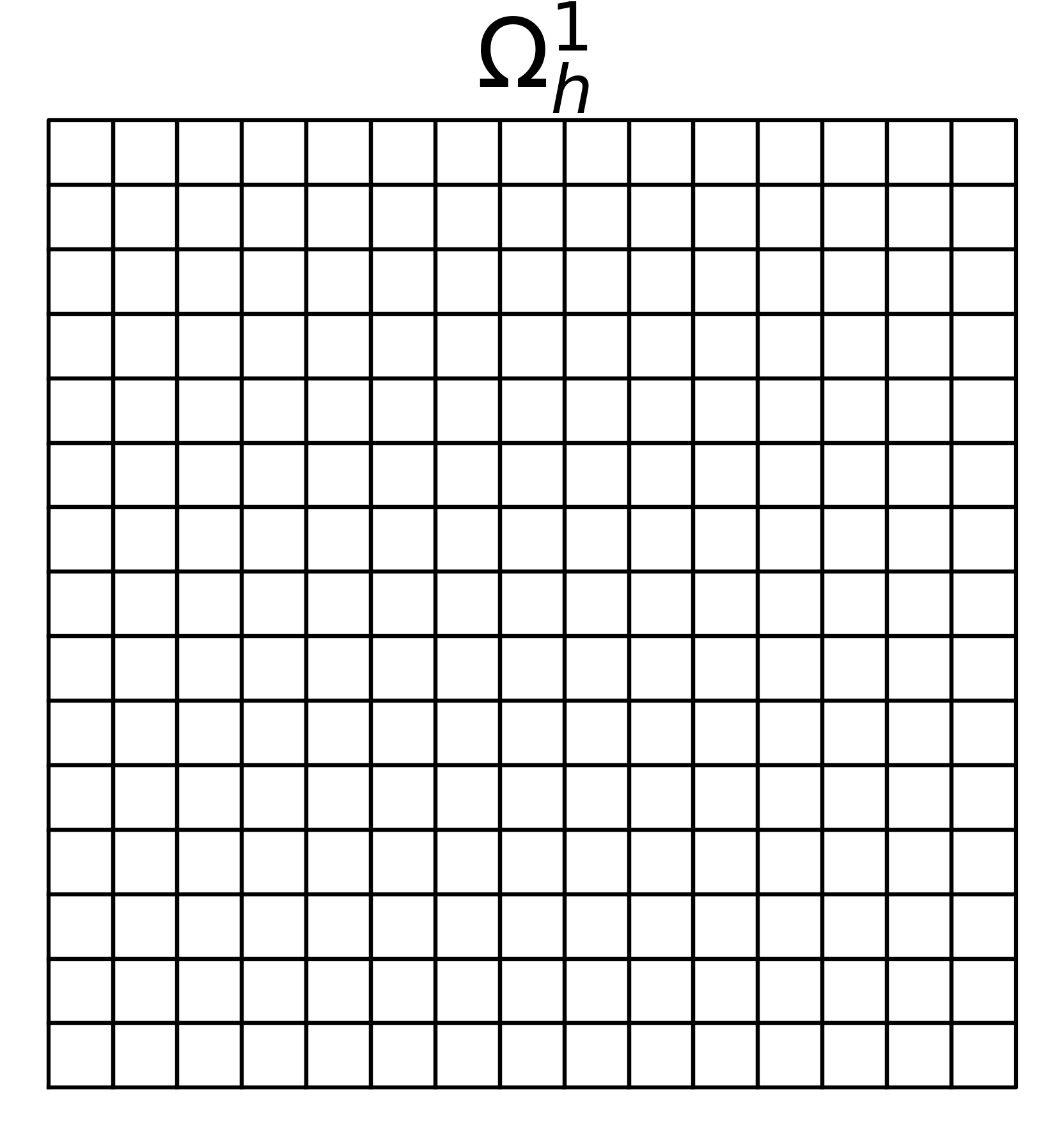}
\includegraphics[width=0.23\textwidth]{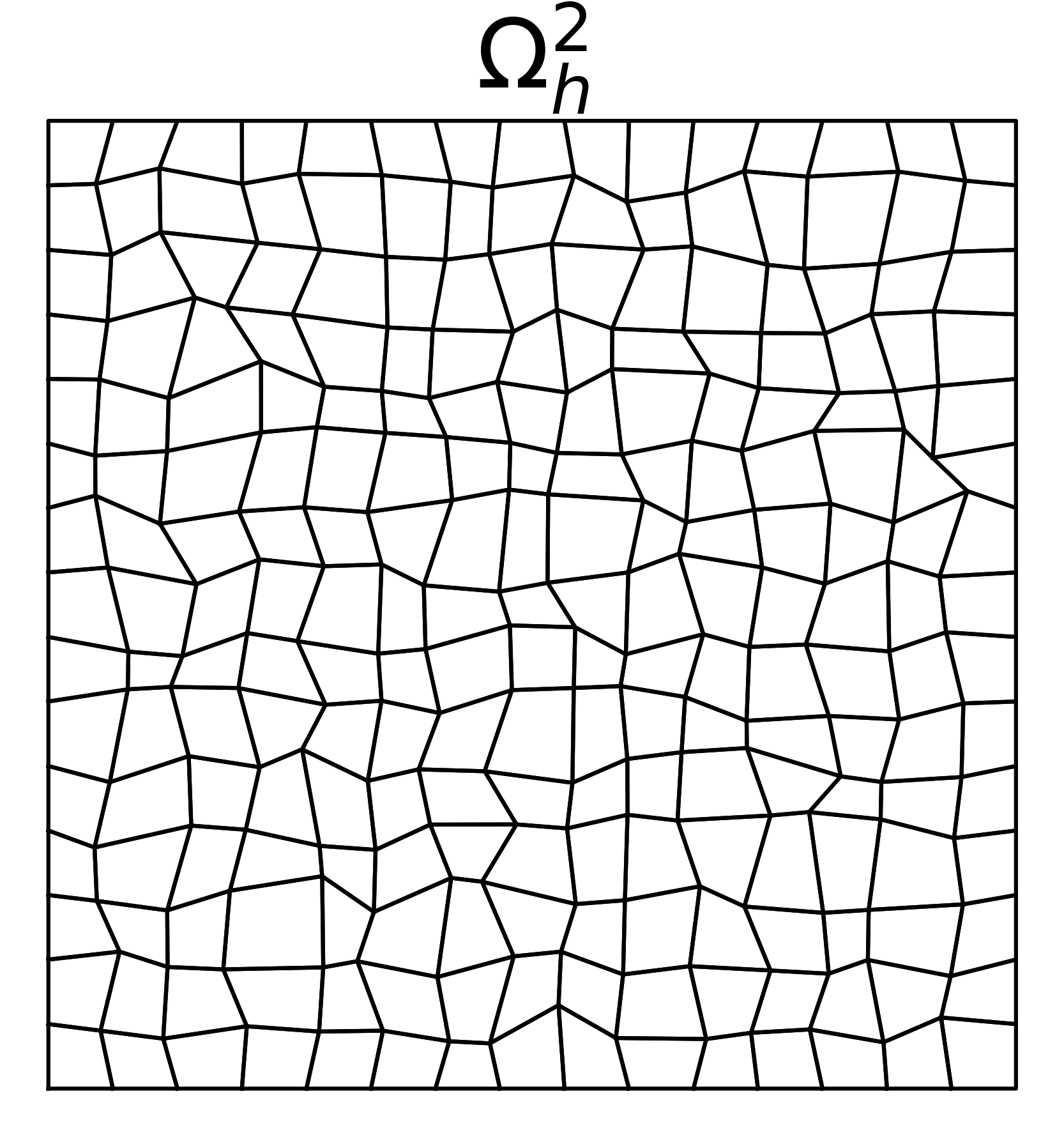}
\includegraphics[width=0.23\textwidth]{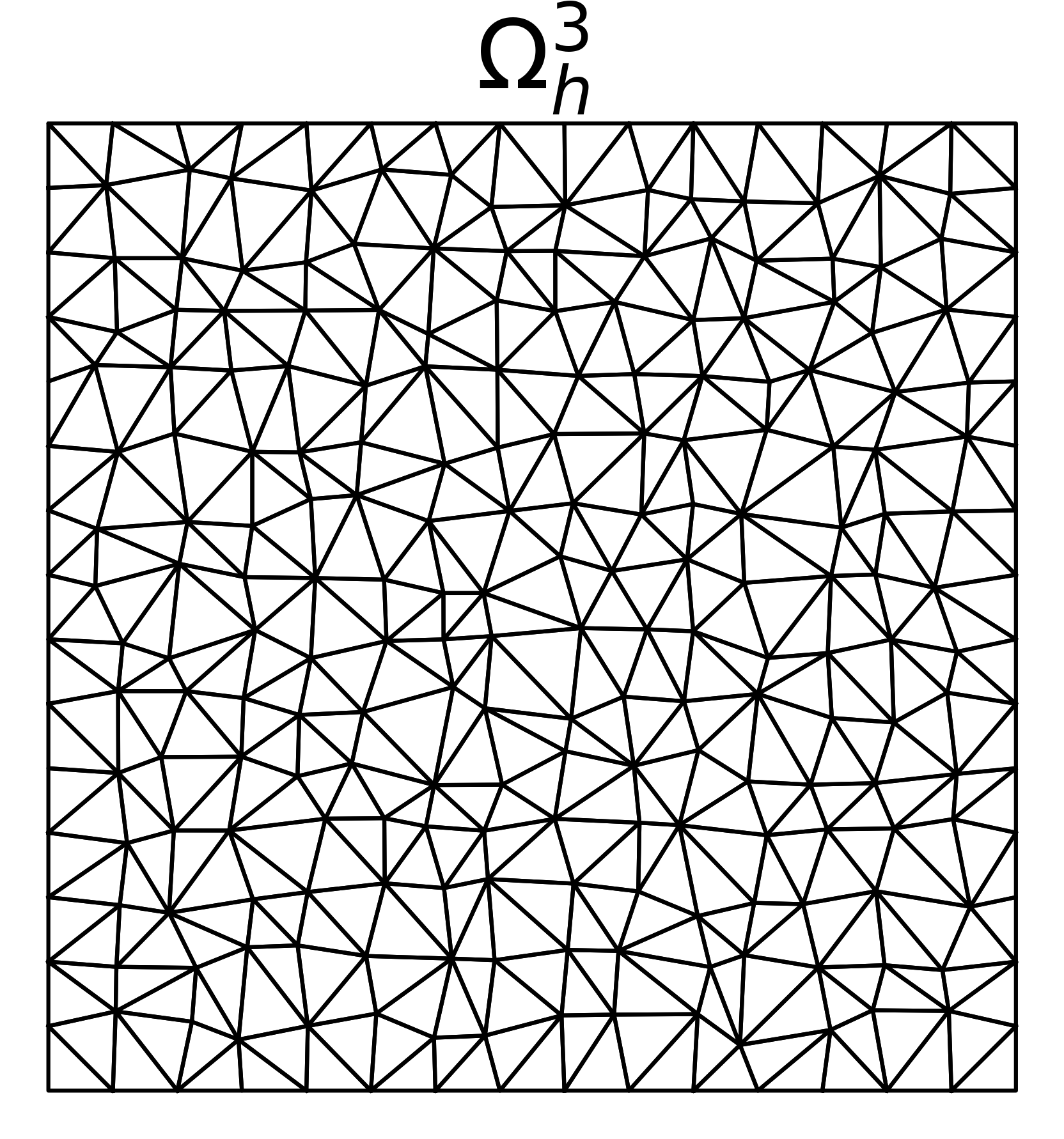}
\includegraphics[width=0.23\textwidth]{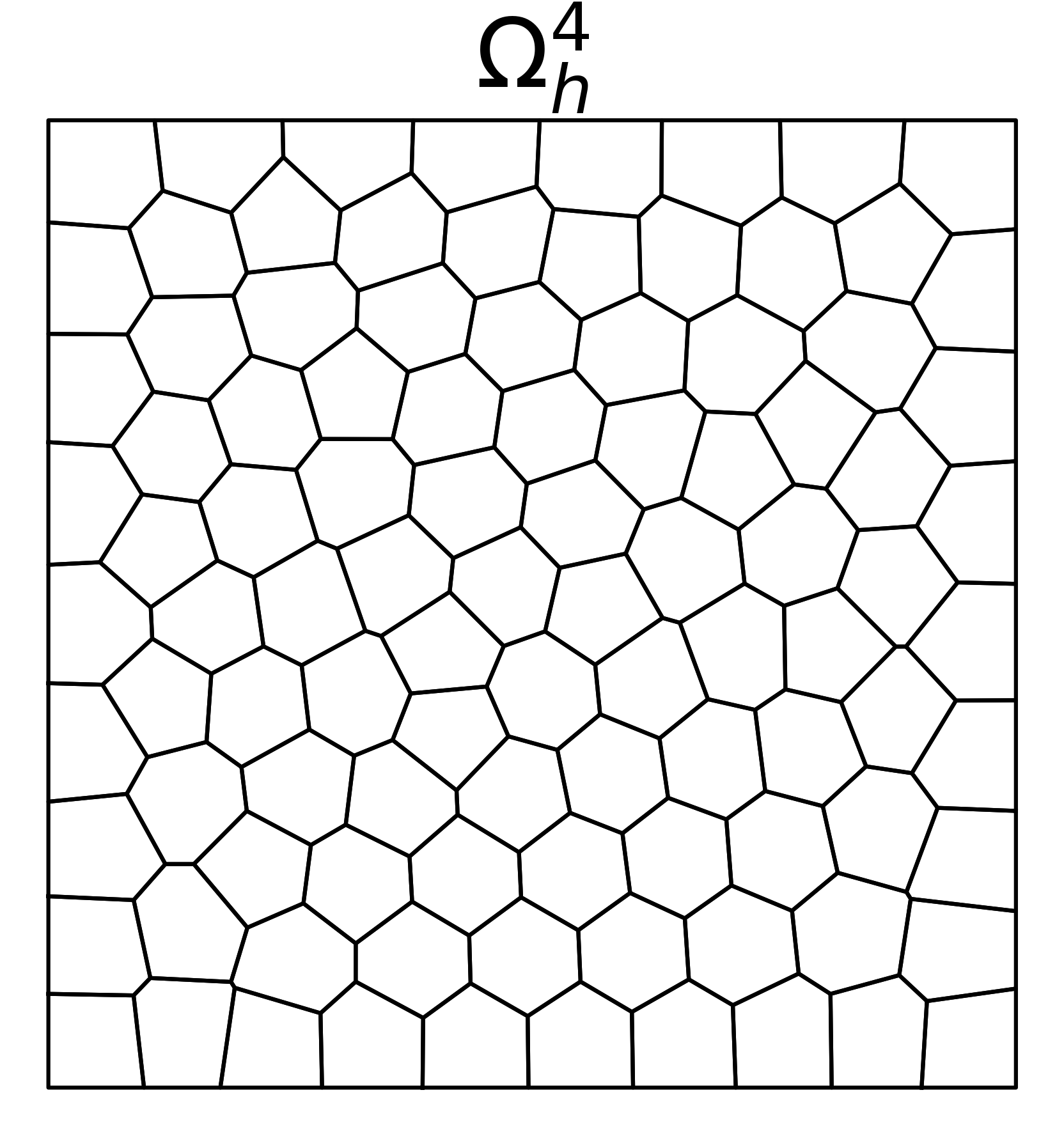}
\caption{Examples of polygonal meshes.}
\label{fig:meshes}
\end{figure}
\vspace{-0.9cm}
\section{Numerical experiments}
In this section, we present numerical experiments validating the theoretical analysis of the fully discrete VE scheme~\eqref{eqn16}. All simulations are carried out using the open-source finite element library DUNE \cite{bastian2008generic, dedner2010generic}. 
Unless otherwise stated, we use polynomial degree two for the discrete spaces $\bm{M}_h^{\mathfrak{J}}$, $\bm{V}_h^{\mathcal{L}}$, and $\mathring{U}_h^{\mathcal{K}}$, while $W_h^\mathcal{L^*} $ is consistently taken to be of degree zero. This choice for $W_h^\mathcal{L^*} $ reduces the number of degrees of freedom in the Navier–Stokes subsystem, as discussed in \cite[Section 5]{da2017divergence}.  The computational domain is discretized using the four polygonal mesh families exemplified in  Fig.~\ref{fig:meshes}. The error measures are defined as (see Section~\ref{VE_formation} for the  definitions of projections): \vspace{-0.15cm}
\begin{equation} \label{error_formula_1}
\begin{aligned}
\mathtt{E}(\bm{u}, L^2, H^1) 
&= \left( \tau \sum_{n=1}^{N} \left| \bm{u}^n - \bm{\Pi}_k^{\nabla} \bm{u}_h^n \right|_1^2 \right)^{\frac{1}{2}}, \; \mathtt{E}(p, L^2, L^2) 
= \left( \tau \sum_{n=1}^{N} \left\| p^n - \Pi_k^{0} p_h^n \right\|_0^2 \right)^{\frac{1}{2}}, \\[-0.2em]
\mathtt{E}(C_i, L^2, H^1) 
&= \left( \tau \sum_{n=1}^{N} \left| C_i^n - \Pi_k^{\nabla} C_{i,h}^n \right|_1^2 \right)^{\frac{1}{2}}, 
\quad i=1,2. 
\end{aligned} \vspace{-0.15cm}
\end{equation}
\begin{equation} \label{error_formula_2}
\begin{aligned}
&\mathtt{E}(\bm{u}, L^\infty, L^2) 
= \left\| \bm{u}(T) - \bm{\Pi}_k^{\nabla} \bm{u}_h^N \right\|_0,\qquad \qquad  
\mathtt{E}(\phi, L^\infty, H^2) 
&= \left| \phi(T) - \Pi_k^{\Delta} \phi_h^N \right|_2,\\[-0.2em]
&\mathtt{E}(C_i, L^\infty, L^2) 
= \left\| C_i(T) - \Pi_k^{\nabla} C_{i,h}^N \right\|_0,
\quad i=1,2.
\end{aligned}
\end{equation}
\begin{table}
\centering
\setlength{\tabcolsep}{8pt} 

\resizebox{\linewidth}{!}{
\begin{tabular}{llcccccc}
\hline
dofs & \diagbox[innerwidth=0.8cm]{$h$}{$\tau$} & $1/4$ &$1/8$& $1/16$&$1/32$&$1/64$\\
\hline
\\[-0.9em]
\multicolumn{7}{c}{$\mathtt{E}(C_2, L^2, H^1)$} \\
\\[-0.9em]
\hline
383   & $1/4$  & \fbox{$1.2291$} & $7.7841 \times 10^{-1}$ & $6.0098 \times 10^{-1}$ & $5.4206 \times 10^{-1}$ & \fbox{$5.2625 \times 10^{-1}$} \\[-0.42em] 
1335  & $1/8$  & $1.0920$ & \fbox{$5.6443 \times 10^{-1}$} & $3.0753 \times 10^{-1}$ & $1.9399 \times 10^{-1}$ & \fbox{$1.5187 \times 10^{-1}$} \\[-0.42em]  
4967  & $1/16$ & $1.0811$ & $5.4565 \times 10^{-1}$ & \fbox{$2.7506\times 10^{-1}$} & $1.4078\times 10^{-1}$ & \fbox{$7.6295\times 10^{-2}$} \\[-0.42em]  
19143 & $1/32$ & $1.0803$ & $5.4440 \times 10^{-1}$ & $2.7283\times 10^{-1}$ & \fbox{$1.3669\times 10^{-1}$} & {$6.8749\times 10^{-2}$} \\[-0.42em]  
75143 & $1/64$ & \fbox{$1.0803$} & \fbox{$5.4432 \times 10^{-1}$} & \fbox{$2.7269\times 10^{-1}$} & \fbox{$1.3642\times 10^{-1}$} & \fbox{$6.8247\times 10^{-2}$} \\
\hline
\\[-0.9em]
\multicolumn{7}{c}{$\mathtt{E}(\bm{u}, L^2, H^1)$} \\
\\[-0.9em]
\hline
383   & $1/4$  & \fbox{$9.1935 \times 10^{-1}$} & $6.5161 \times 10^{-1}$ & $5.5815 \times 10^{-1}$ & $5.3452 \times 10^{-1}$ & \fbox{$5.4535 \times 10^{-1}$} \\[-0.42em]   
1335  & $1/8$  & $7.3854 \times 10^{-1}$ & \fbox{$3.8993 \times 10^{-1}$} & $2.2883 \times 10^{-1}$ & $1.6438 \times 10^{-1}$ & \fbox{$1.4365 \times 10^{-1}$} \\[-0.42em]   
4967  & $1/16$ & $7.2309 \times 10^{-1}$ & $3.6355 \times 10^{-1}$ & \fbox{$1.8436\times 10^{-1}$} & $9.6878\times 10^{-2}$ & \fbox{$5.6714\times 10^{-2}$} \\[-0.42em]   
19143 & $1/32$ & $7.2203 \times 10^{-1}$ & $3.6173 \times 10^{-1}$ & $1.8110\times 10^{-1}$ & \fbox{$9.0863\times 10^{-2}$} & {$4.6025\times 10^{-2}$} \\[-0.42em]   
75143 & $1/64$ & \fbox{$7.2196 \times 10^{-1}$} & \fbox{$3.6162 \times 10^{-1}$} & \fbox{$1.8089\times 10^{-1}$} & \fbox{$9.0473\times 10^{-2}$} & \fbox{$4.5274\times 10^{-2}$} \\
\hline
\\[-0.9em]
\multicolumn{7}{c}{$\mathtt{E}(p, L^2, L^2)$} \\
\\[-0.9em]
\hline
383   & $1/4$  & \fbox{$1.0500 \times 10^{-1}$} & $8.0241 \times 10^{-2}$ & $7.7429 \times 10^{-2}$ & $8.4669 \times 10^{-2}$ & \fbox{$9.7159 \times 10^{-2}$} \\[-0.42em]  
1335  & $1/8$  & $9.5272 \times 10^{-2}$ & \fbox{$5.4356 \times 10^{-2}$} & $3.7218 \times 10^{-2}$ & $3.1440 \times 10^{-2}$ & \fbox{$2.9887 \times 10^{-2}$} \\[-0.42em]  
4967  & $1/16$ & $9.3093 \times 10^{-2}$ & $4.8576 \times 10^{-2}$ & \fbox{$2.7412\times 10^{-2}$} & $1.8560\times 10^{-2}$ & \fbox{$1.5529\times 10^{-2}$} \\[-0.42em]  
19143 & $1/32$ & $9.2493 \times 10^{-2}$ & $4.6996 \times 10^{-2}$ & $2.4379\times 10^{-2}$ & \fbox{$1.3710\times 10^{-2}$} & {$9.2617\times 10^{-3}$} \\[-0.42em]  
75143 & $1/64$ & \fbox{$9.2337 \times 10^{-2}$} & \fbox{$4.6588 \times 10^{-2}$} & \fbox{$2.3556\times 10^{-2}$} & \fbox{$1.2195\times 10^{-2}$} & \fbox{$6.8480\times 10^{-3}$} \\ 
\hline
\end{tabular}
}
\vspace{0.1cm}
\caption{Example~\ref{manufactured}. Error norms \eqref{error_formula_1} using the VE scheme \eqref{eqn16} with polynomial degrees $\mathfrak{J} = \mathcal{K} = \mathcal{L} = 2$ on the mesh $\Omega_h^1$.}
\label{tabla1-1}
\end{table}
\subsection{Accuracy test} \label{manufactured} Here, we employ the method of manufactured solutions  and consider the following prescribed smooth functions on the spatial domain $\Omega = (0,1)^2$ and the time interval $[0,1]$: \vspace{-0.15cm}
\begin{equation}
\begin{aligned} 
    &C_1 = e^{-2t}\sin(4\pi x)\sin(4\pi y), 
    &&C_2 = e^{-2t}\sin(2\pi x)\sin(2\pi y), \\[-0.2em]
    &\phi = e^{-t}\cos(2\pi x)\cos(2\pi y), 
    &&p = e^{-t}\bigl(\sin x - \sin y\bigr), \\[-0.2em]
    &\bm{u} = \bigl(e^{-t}\sin^2(\pi x)\sin(2\pi y),\;
    -e^{-t}\sin(2\pi x)\sin^2(\pi y)\bigr).\nonumber
\end{aligned} \vspace{-0.15cm}
\end{equation}
We set $\upnu=1$ and augment system~\eqref{eqn3}  with suitable forcing terms so that  the above set of functions forms an exact solution. 

We begin by fixing $\mathfrak{J} = \mathcal{K} = \mathcal{L} = 2$ and performing simultaneous refinements in both space and time, with $h$ and $\tau$ taking the values $1/4$, $1/8$, $1/16$, $1/32$, and $1/64$. The results are reported in Table~\ref{tabla1-1}, where the columns (left to right) correspond to time step refinements $\tau$  and the rows (top to bottom) correspond to spatial mesh refinements $h$ . Boxed entries highlight the convergence trends observed on the mesh $\Omega_h^1$ for the velocity, pressure, and concentration errors defined in \eqref{error_formula_1}.  
Figure~\ref{fig:manufactured_errors} (top row) presents the corresponding Log-Log plots for the case $h = \tau = 1/4, 1/8, 1/16, 1/32, 1/64$ on all the meshes depicted in Figure~\ref{fig:meshes}, illustrating that the errors follow the expected linear behavior i.e. $h^{\min\{\mathfrak{J},\mathcal{K}-1,\mathcal{L}\}}=h$ (see Theorem~\ref{thm_fully_error}). All results are reported for the concentration $C_2$, since $C_1$ exhibits the same convergence behavior.

To further verify whether the observed convergence rate for $\mathcal{K} = 2$ is intrinsic or coincidental, we increased $\mathcal{K}$ to 3 while keeping the other polynomial degrees unchanged and fixing $\tau=h^2$. 
In agreement with Theorem~\ref{thm_fully_error}, the concentration and velocity errors measured in the norms~\eqref{error_formula_1} converge quadratically, 
while the pressure error remains linear (pressure robust scheme), as observed in Fig.~\ref{fig:manufactured_errors} (middle row).

We then investigate the error norms listed in \eqref{error_formula_2}, using the time step  $\tau=h^2$,  \( \mathfrak{J} = \mathcal{L} =2\) and $\mathcal{K}=3$.  Since standard interpolation estimates yield third-order accuracy in $h$ for the concentrations and velocity, and second-order accuracy for the electrostatic potential in the norms appearing in \eqref{error_formula_2}, it is reasonable to expect that the fully discrete solution may attain similar convergence behavior. The results, obtained across all four meshes, as shown in the bottom row of Fig.~\ref{fig:manufactured_errors}, confirm that the concentrations and velocity indeed achieve third-order convergence with respect to \( h \).
\begin{figure}
\centering
\setlength{\tabcolsep}{2pt}
\renewcommand{\arraystretch}{1}

\begin{tabular}{cccc}
\includegraphics[width=0.33\textwidth]{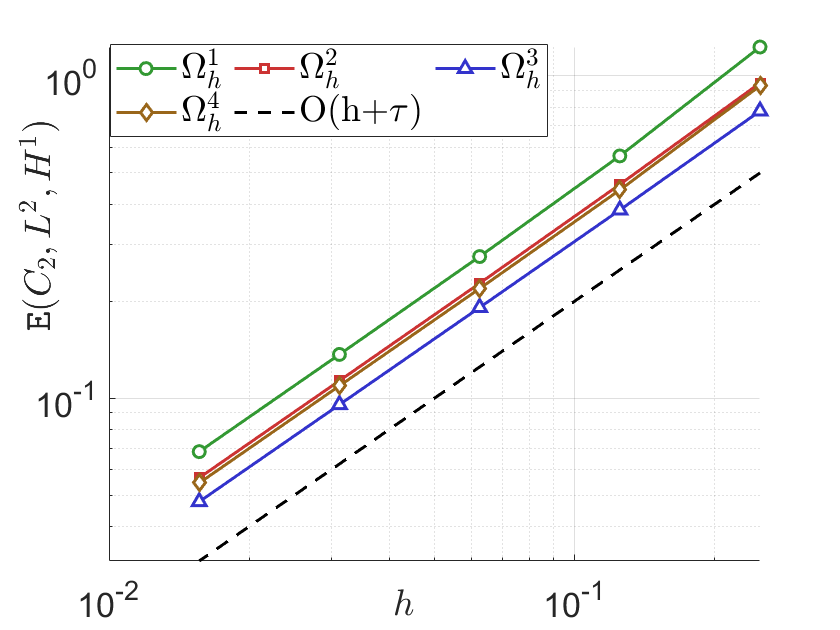} &
\includegraphics[width=0.33\textwidth]{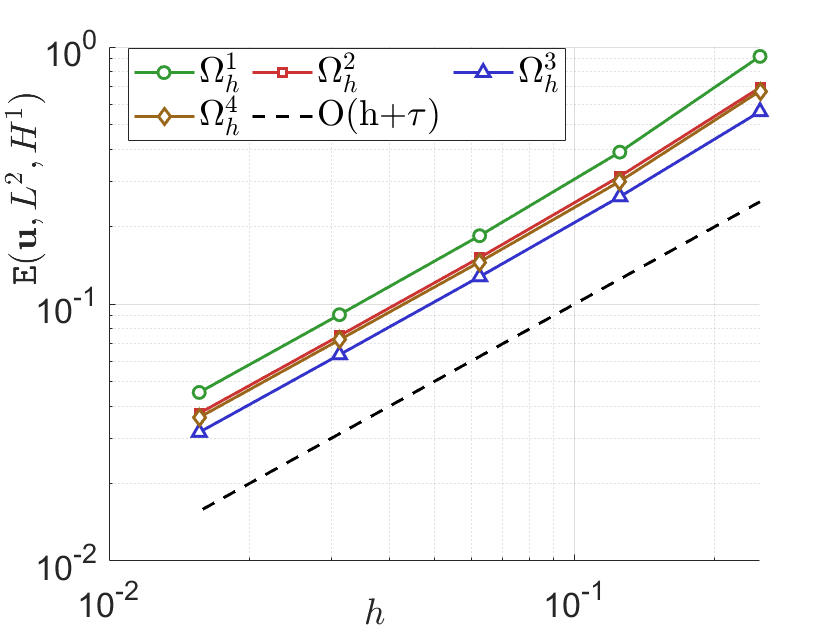} &
\includegraphics[width=0.33\textwidth]{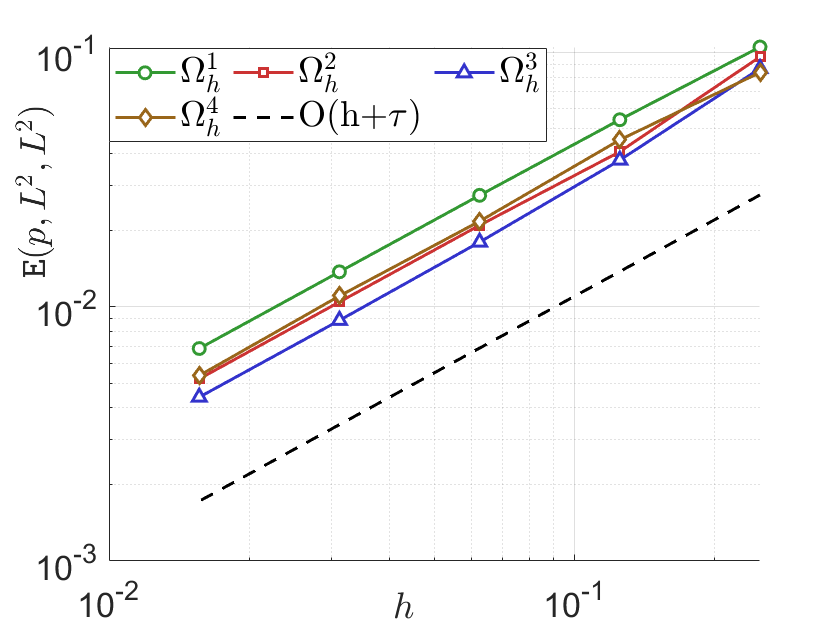} \\
\includegraphics[width=0.33\textwidth]{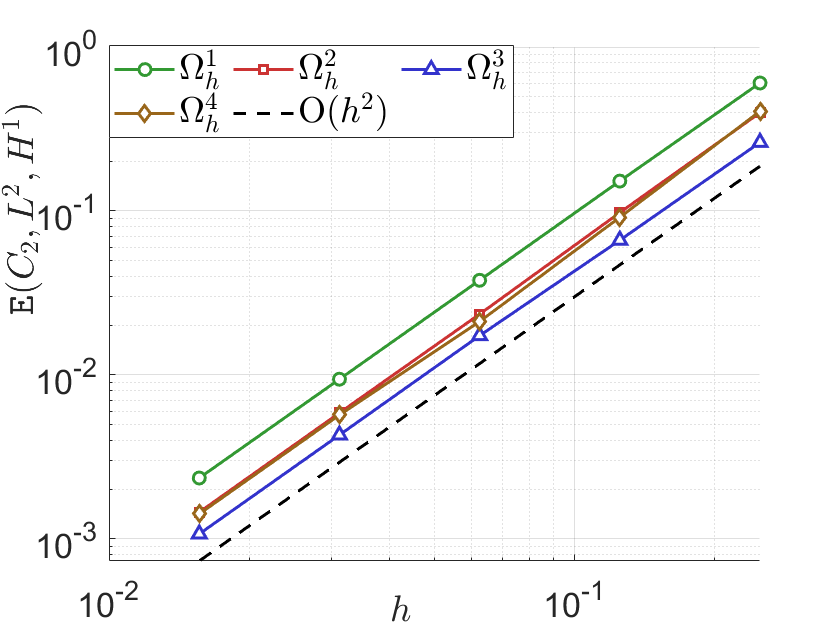} &
\includegraphics[width=0.33\textwidth]{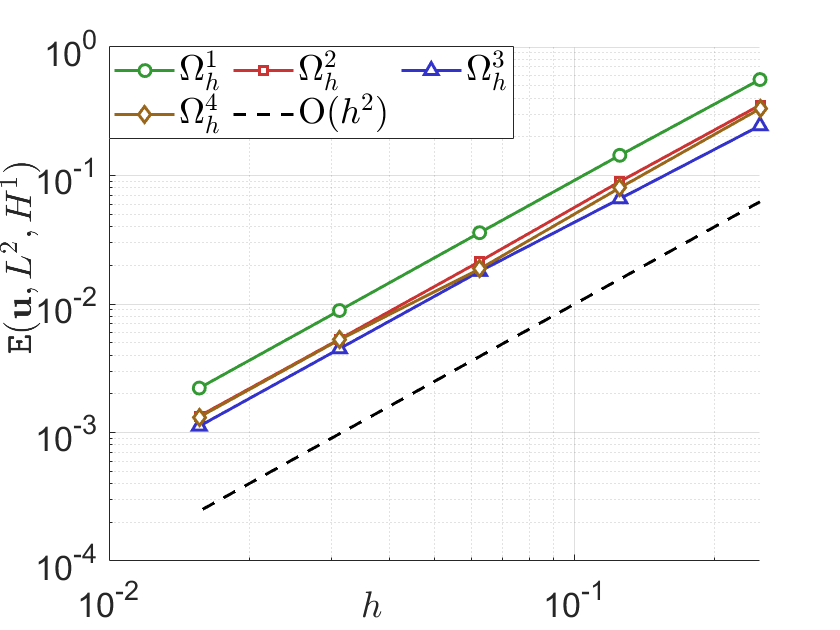} &
\includegraphics[width=0.33\textwidth]{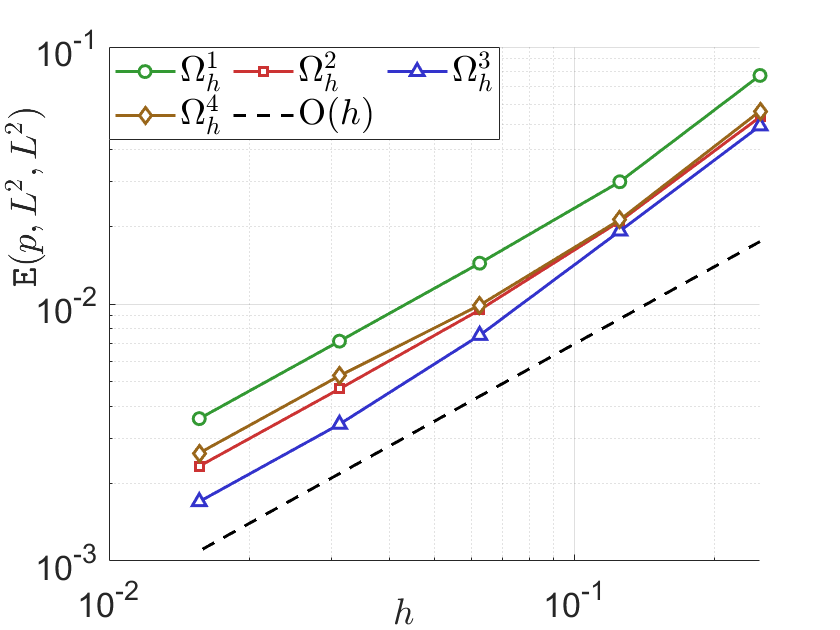} \\
\includegraphics[width=0.33\textwidth]{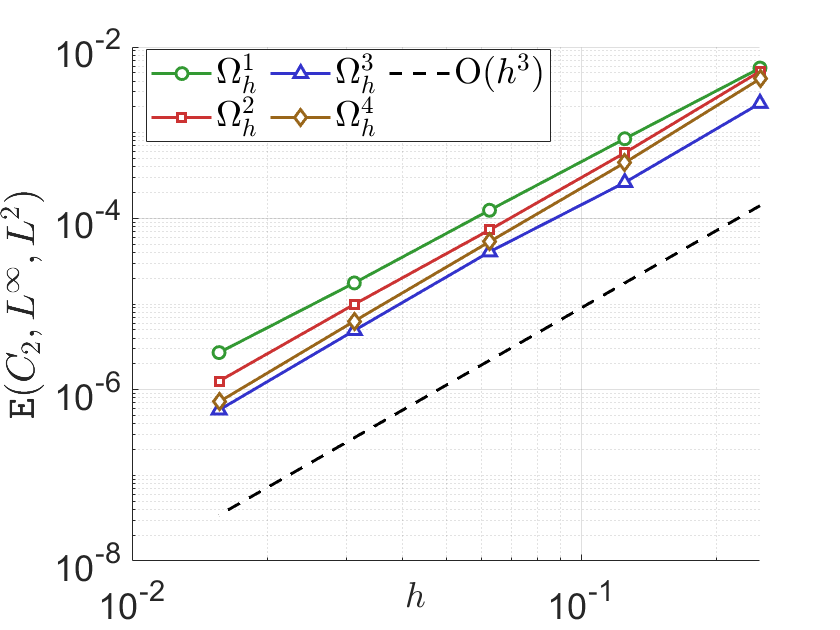} &
\includegraphics[width=0.33\textwidth]{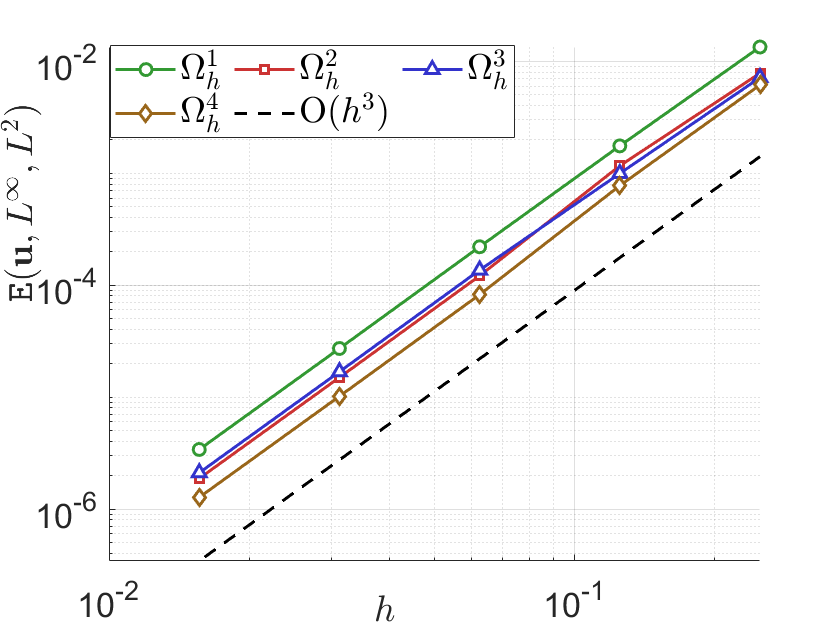} &
\includegraphics[width=0.33\textwidth]{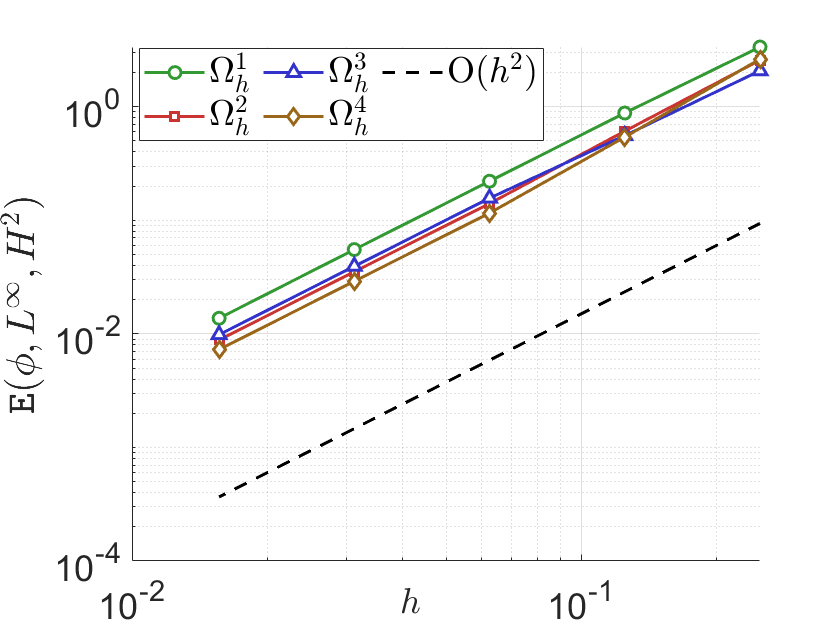}
\end{tabular}
\caption{Example~\ref{manufactured}. Error norms obtained on all four meshes depicted in Fig.~\ref{fig:meshes}. 
Top row: \eqref{error_formula_1} with time step $\tau=h$ and polynomial degrees $\mathfrak{J}=\mathcal{K}=\mathcal{L}=2$. 
Middle row: \eqref{error_formula_1} with time step $\tau=h^2$ and polynomial degrees $\mathfrak{J}=\mathcal{L}=2$, $\mathcal{K}=3$. 
Bottom row: \eqref{error_formula_2} with time step $\tau=h^2$ and polynomial degrees $\mathfrak{J}=\mathcal{L}=2$, $\mathcal{K}=3$.}
\label{fig:manufactured_errors}
\end{figure}
\subsection{Small viscosity test}\label{ex_viscosity}
We test the robustness of scheme \eqref{eqn16} for low viscosities by setting $\upnu = 10^{-3}, 10^{-2}, 10^{-1}$, and $10^{0}$, using the setup from the Accuracy Test (\ref{manufactured}). Figure~\ref{viscosity} shows the error norms \eqref{error_formula_1} and \eqref{error_formula_2} of the velocity variable as functions of the viscosity $\upnu$. The spatial meshes considered are $h = 1/4, 1/8, 1/16, 1/32, 1/64$, with time step choices $\tau = h$ for the left plot and $\tau = h^2$ for the right plot.

The results demonstrate that the proposed VE scheme maintains stability and delivers accurate approximations even for very small viscosity values (see Figure~\ref{viscosity}). As expected, the errors increase slightly as $\upnu$ decreases, reflecting the higher numerical challenge in resolving sharper gradients at lower viscosity values. Moreover, the computed velocity field is (numerically) divergence-free, with $\|\nabla\cdot u_h\|_{L^2(\Omega)} \approx 5.0\times 10^{-7}$ for the smallest viscosity considered, confirming the incompressibility-preserving nature of the method. These observations are consistent with the general observation that exactly divergence-free Galerkin methods are more robust with respect to small viscosity parameters; see, for instance, \cite{schroeder2017pressure}, and also \cite{da2018virtual} in the VEM context.
\begin{figure}
\centering
\includegraphics[width=0.55\linewidth]{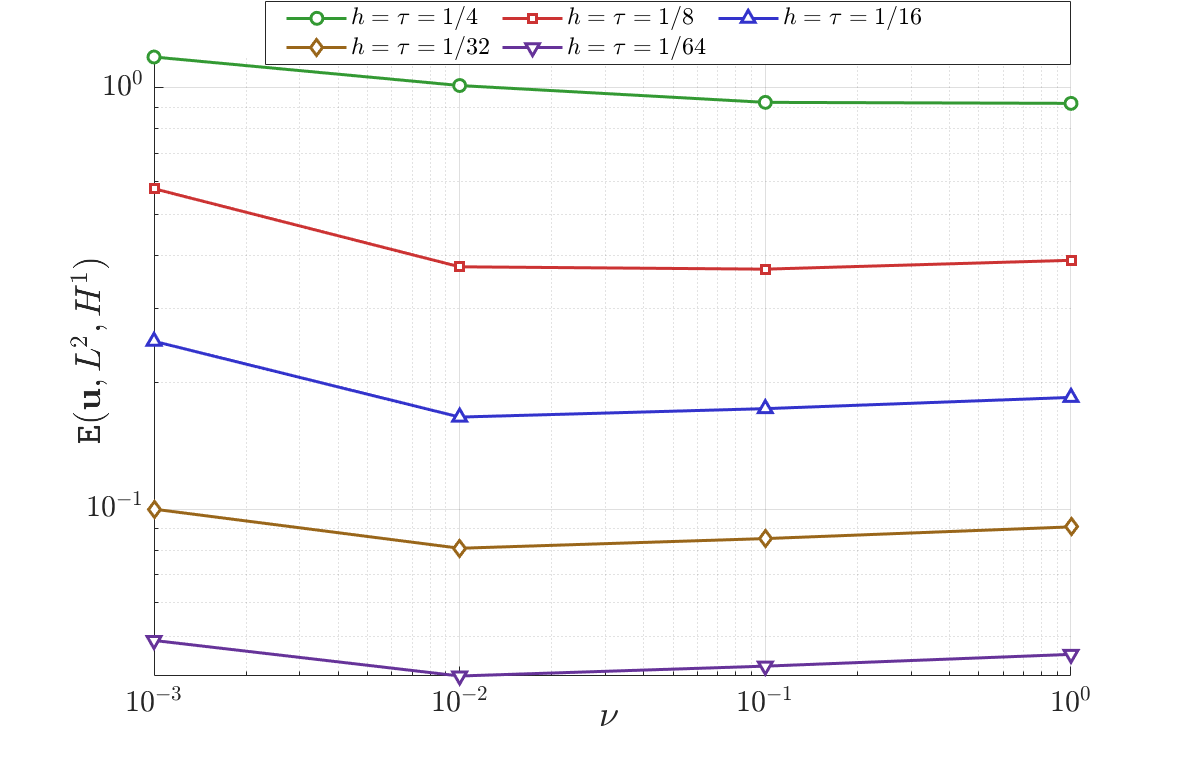}\hspace{-7mm}
\includegraphics[width=0.48\linewidth]{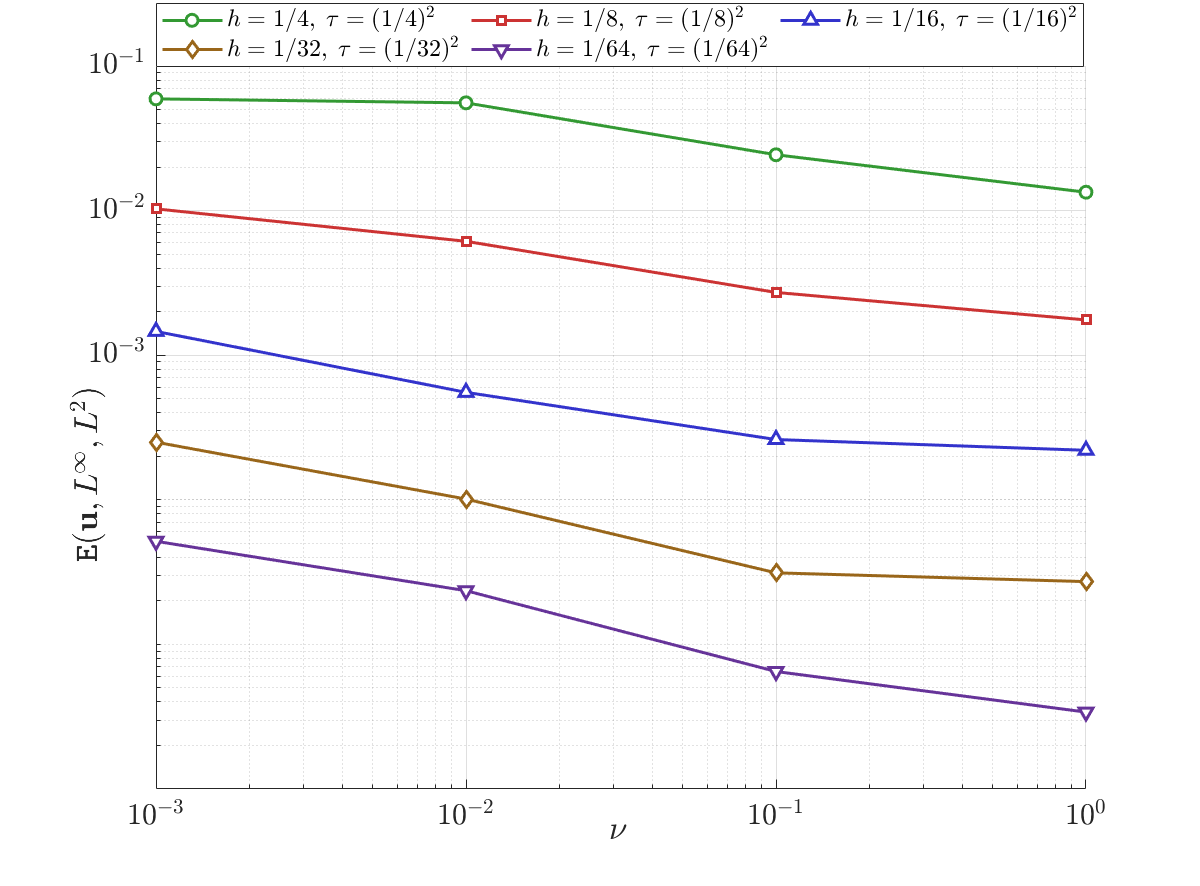}
\caption{Example~\ref{ex_viscosity}. Velocity error norms \eqref{error_formula_1} and \eqref{error_formula_2} for different viscosities~$\upnu$ using $\tau=h$ (left) and $\tau=h^2$ (right), with $\mathfrak{J}=\mathcal{K}=\mathcal{L}=2$ on $\Omega_h^1$.}
\label{viscosity}
\end{figure}
\subsection{Dynamics under Initial Discontinuous
Concentrations} \label{Disconti}
Here, we study the effect of discontinuous initial data on the evolution of the numerical scheme \eqref{eqn16} posed on the unit square with all forcing terms set to zero. We consider initial concentrations \cite{prohl2009convergent} \vspace{-0.2cm}
\[
c_{1,0}(x,y) =
\begin{cases}
1, & (x,y) \in (0,1)^2 \setminus \Big( (0,0.75)\times(0,1)\;\cup\;(0.75,1)\times\bigl(0,\tfrac{1}{2}\bigr) \Big), \\[-0.2em]
10^{-6}, & \text{otherwise},
\end{cases}\vspace{-0.1cm}
\]
\[
c_{2,0}(x,y) =
\begin{cases}
1, & (x,y) \in (0,1)^2 \setminus \Big( (0,0.75)\times(0,1)\;\cup\;(0.75,1)\times\bigl(\tfrac{1}{2},1\bigr) \Big), \\[-0.2em]
10^{-6}, & \text{otherwise}.
\end{cases} \vspace{-0.1cm}
\]
and initial velocity $\bm{u}_0(x,y)=\bm{0}$. The discontinuity of the initial concentrations represents an interface between
the electrolyte and the solid surfaces where electroosmosis is expected to occur. The simulations are performed with mesh size $h=1/64$, time step $\tau=10^{-3}$, and final time $T=0.25$. The temporal evolution of the concentrations, electric potential, velocity field, and pressure is illustrated in Figure~\ref{Fig-Dis2}.
\begin{figure}
			\centering
			\includegraphics[width=0.32\textwidth]{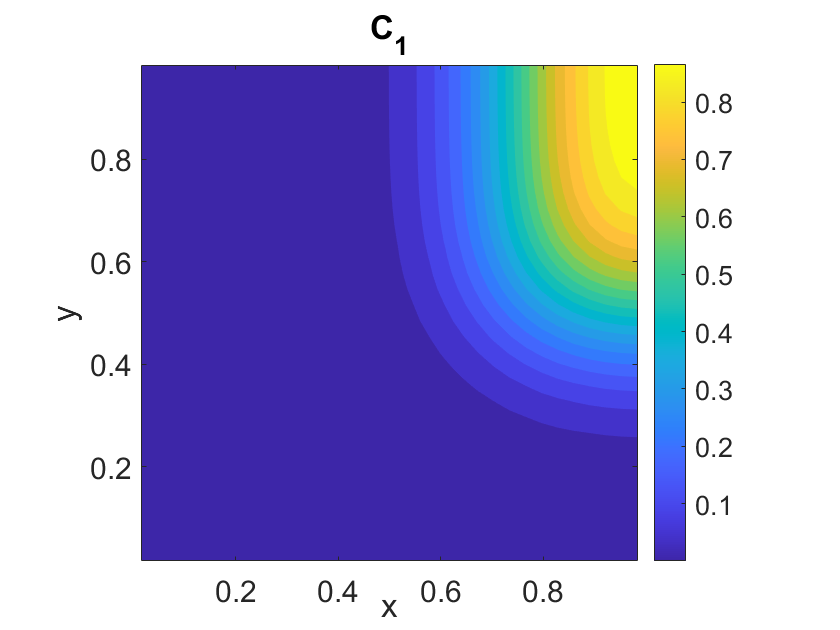}\;
            \includegraphics[width=0.32\textwidth]{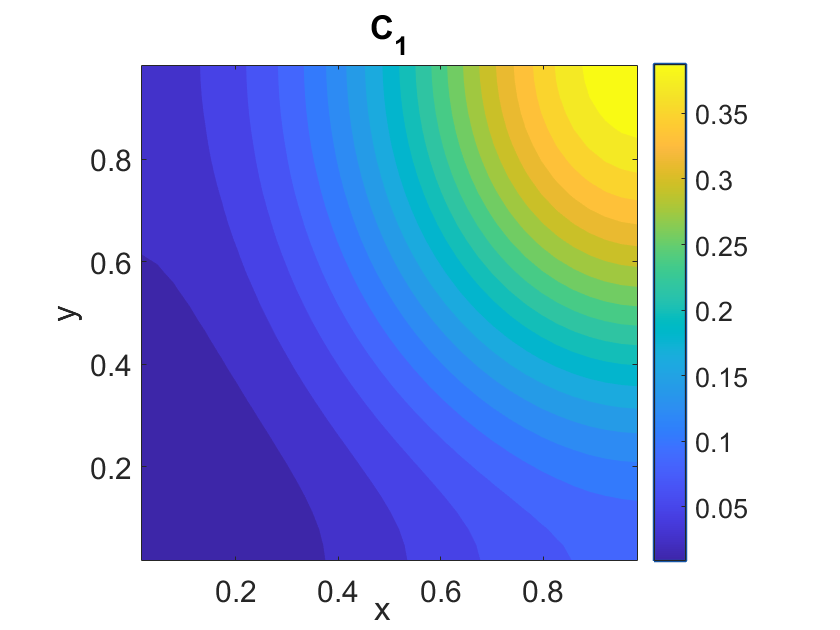}\;
            \includegraphics[width=0.32\textwidth]{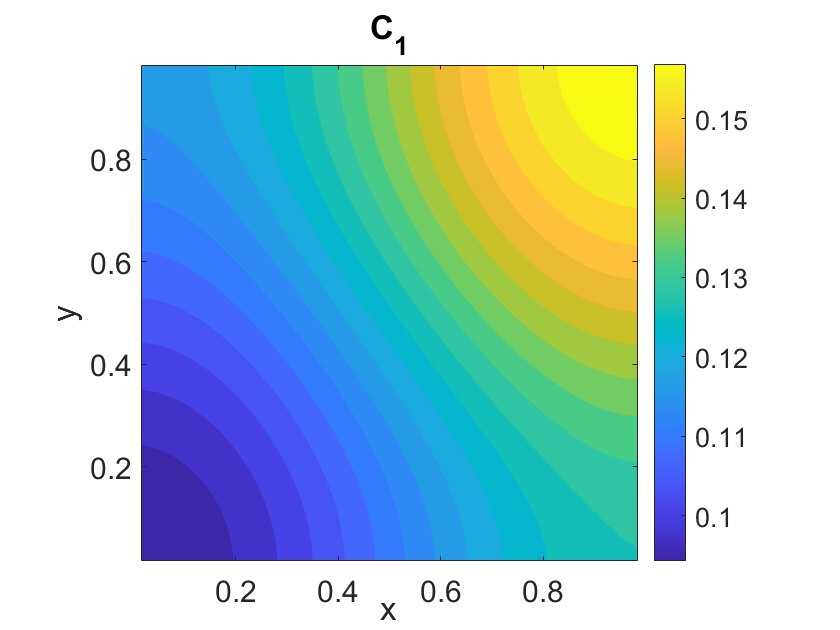}\\
			\includegraphics[width=0.32\textwidth]{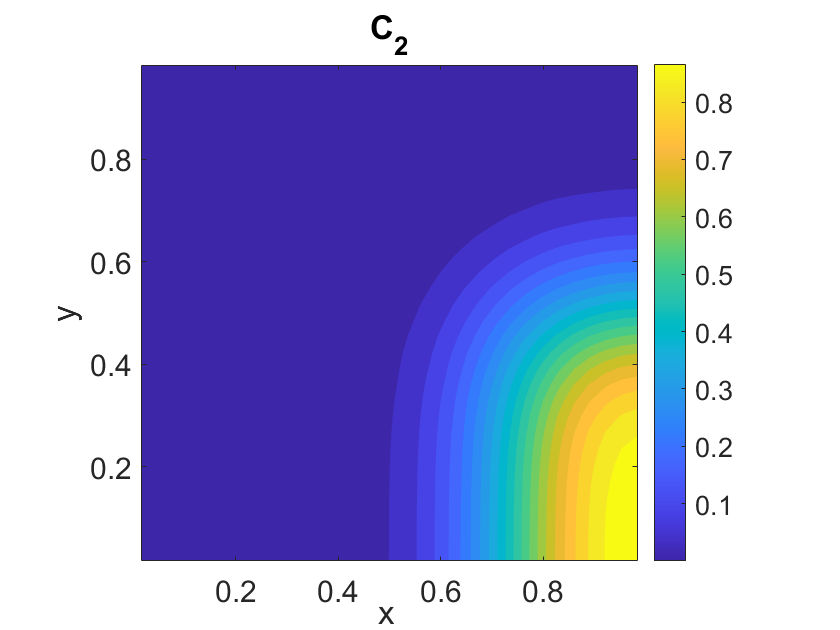}\;
            \includegraphics[width=0.32\textwidth]{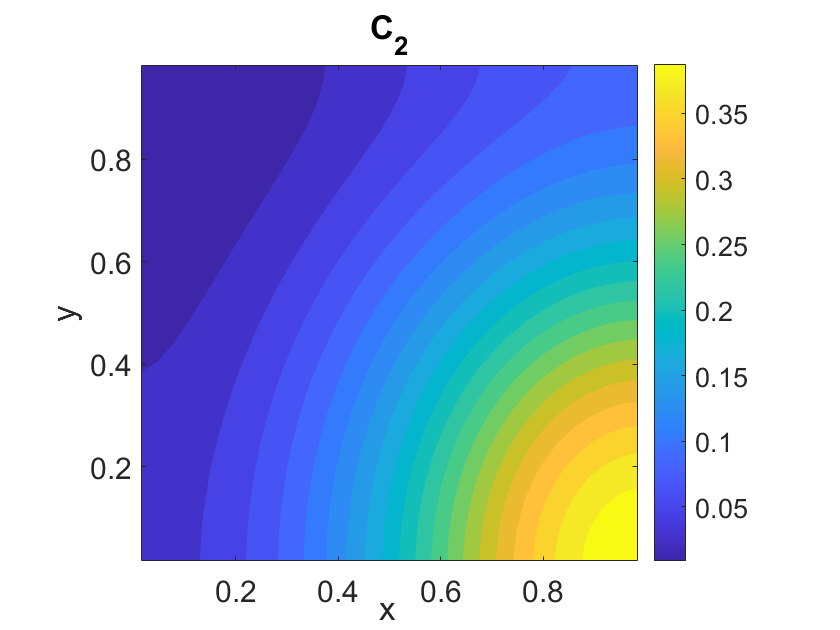}\;
            \includegraphics[width=0.32\textwidth]{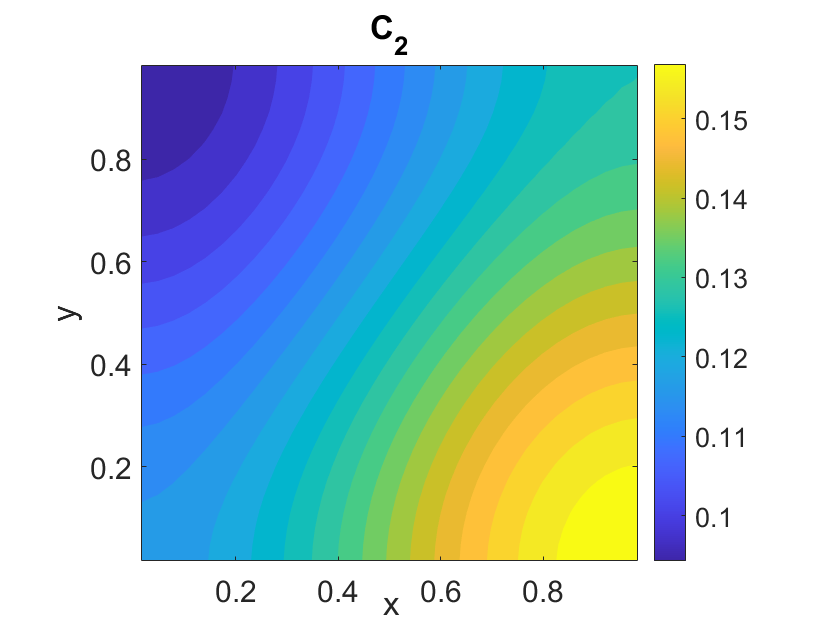}\\
            \includegraphics[width=0.32\textwidth]{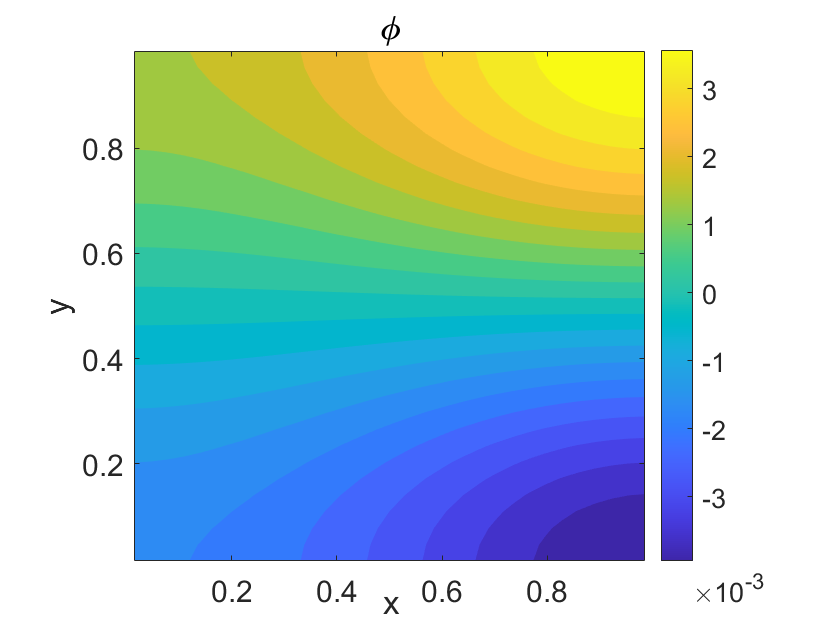}\;
            \includegraphics[width=0.32\textwidth]{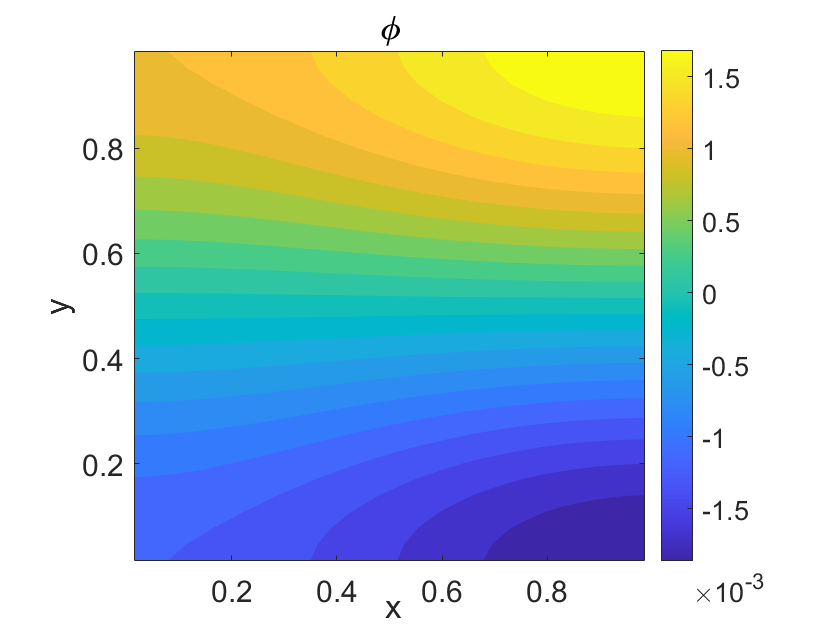}\;
            \includegraphics[width=0.32\textwidth]{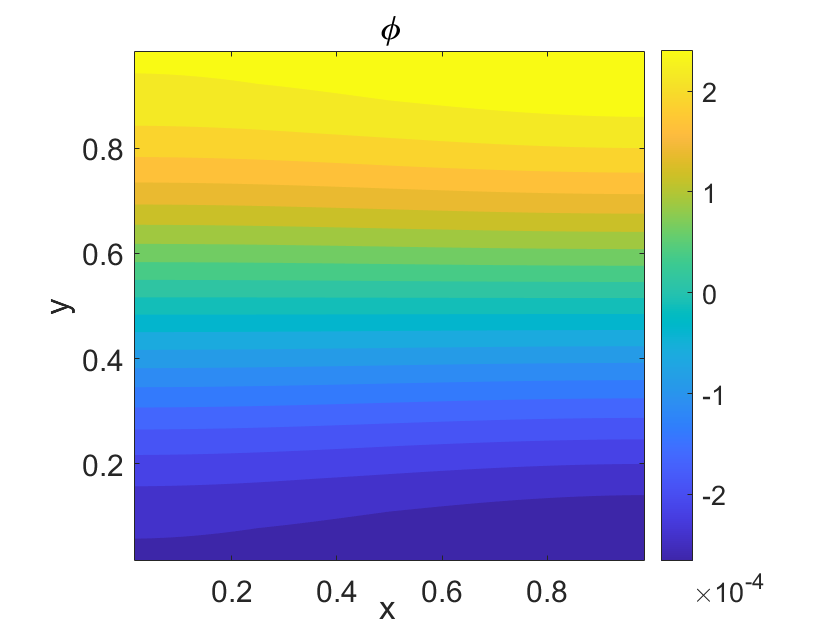}\\
            \includegraphics[width=0.32\textwidth]{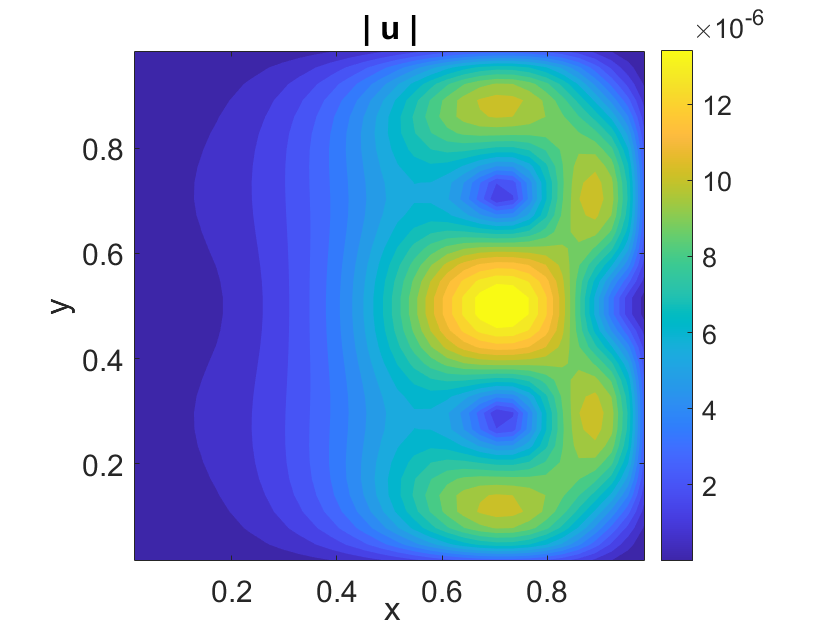}\;
            \includegraphics[width=0.32\textwidth]{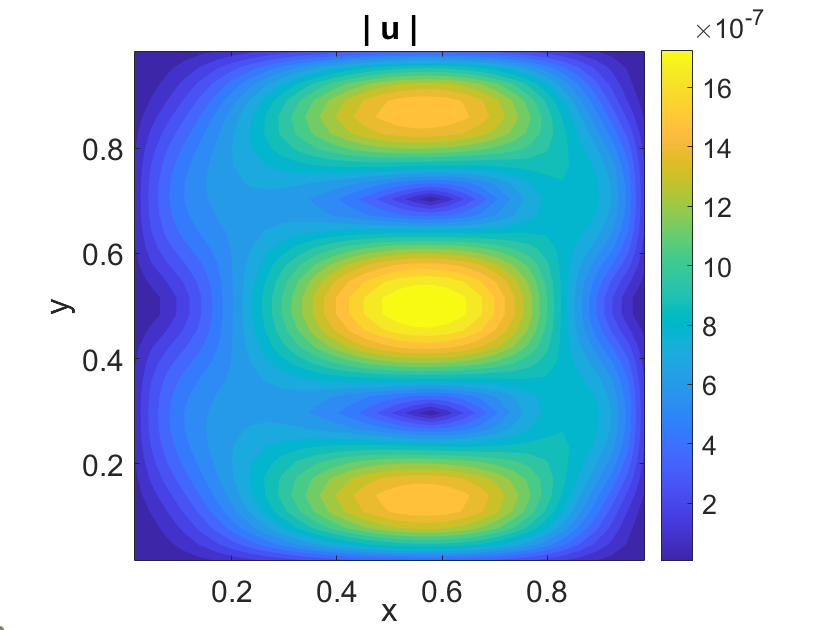}\;
            \includegraphics[width=0.32\textwidth]{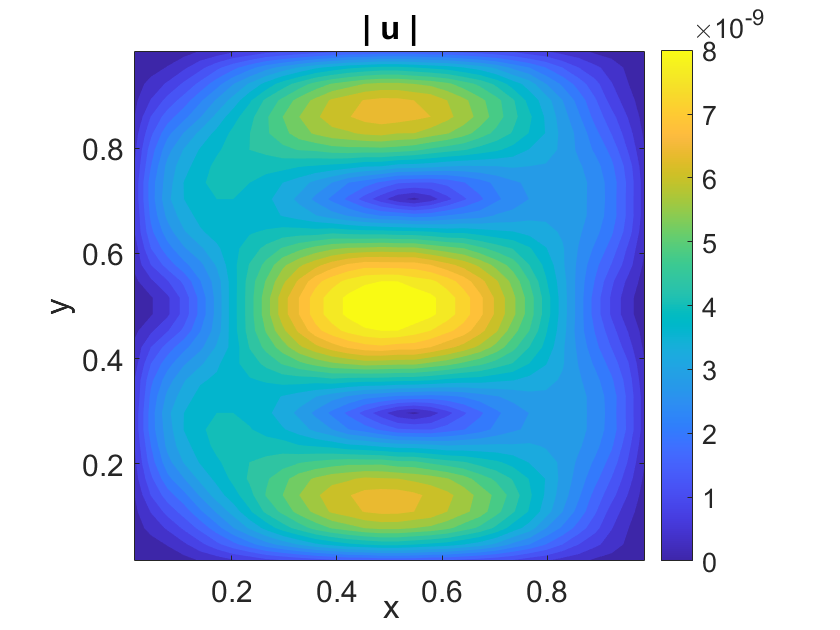}\\
            \includegraphics[width=0.32\textwidth]{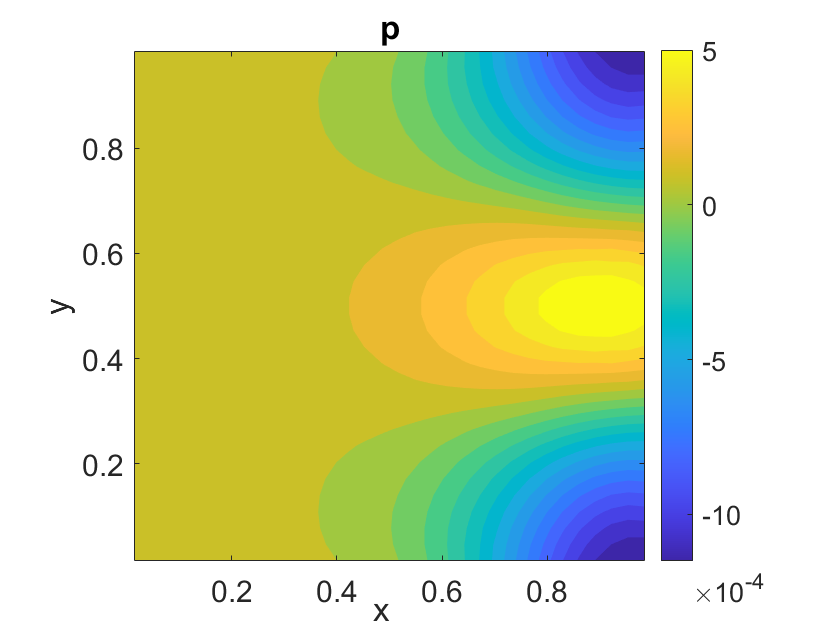}\;
            \includegraphics[width=0.32\textwidth]{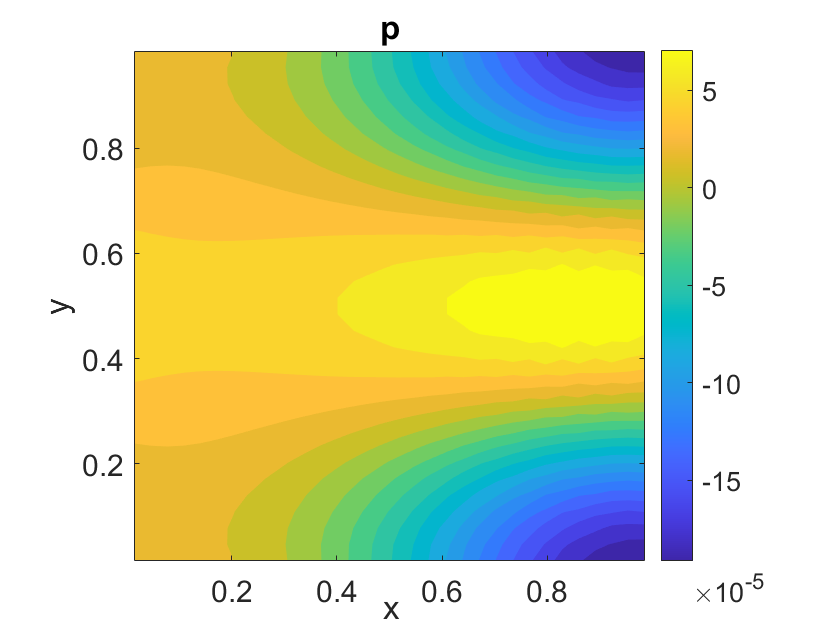}\;
            \includegraphics[width=0.32\textwidth]{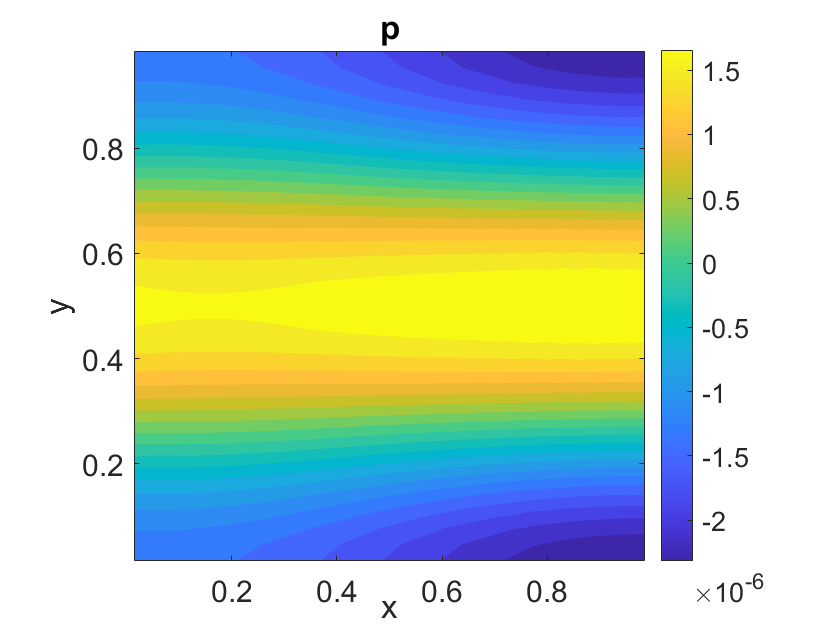}
\caption{Example~\ref{Disconti}. Temporal evolution of the solution. From top to bottom: concentration $C_1$, concentration $C_2$, electric potential, velocity magnitude, and pressure. From left to right: $t=0.01$, $t=0.07$, and $t=0.25$.}
\label{Fig-Dis2}
\end{figure}

\newpage
\section*{Acknowledgments}
AC is a member of Gruppo Nazionale per il Calcolo Scientifico (GNCS) of Istituto Nazionale di Alta Matematica (INdAM)and acknowledge the support of the European Research Council (ERC) under the European Union’s Horizon
2020 research and innovation programme (call HORIZON-EUROHPC-JU-2023-COE-03, grant agreement No. 101172493 ``dealii-X'').

\bibliographystyle{abbrv}
\bibliography{references}

\end{document}